\documentclass[11pt]{article}
\pdfoutput=1

\usepackage{lineno,hyperref}
\usepackage{color}
\usepackage{amsfonts}
\usepackage{amsmath}
\usepackage{amssymb}
\usepackage{amsthm}
\usepackage{subeqnarray}
\usepackage{lscape,graphicx,url}
\usepackage{anysize}
\usepackage{proof}
\usepackage{fancyhdr}
\usepackage{latexsym}
\usepackage{array}
\usepackage{multirow}
\usepackage{textcomp}
\usepackage{optprog}
\usepackage{mathtools}
\usepackage{proof}
\usepackage{natbib}
\usepackage[margin=1.5cm]{geometry}

\usepackage{proof}
\usepackage{paralist}
\usepackage[dvipsnames]{xcolor}
\usepackage{placeins}
\usepackage{enumitem}
\usepackage{csquotes}
\usepackage{comment}
\usepackage{bm}
\usepackage{array}

\makeatletter
\def\thickhline{%
  \noalign{\ifnum0=`}\fi\hrule \@height \thickarrayrulewidth \futurelet
   \reserved@a\@xthickhline}
\def\@xthickhline{\ifx\reserved@a\thickhline
               \vskip\doublerulesep
               \vskip-\thickarrayrulewidth
             \fi
      \ifnum0=`{\fi}}
\makeatother

\newlength{\thickarrayrulewidth}
\newcolumntype{H}{>{\setbox0=\hbox\bgroup}c<{\egroup}@{}}

\newtheorem{theorem}{Theorem}

\newcommand{\GG}[1]{}

\title{Order allocation, rack allocation and rack sequencing for pickers in a mobile rack environment}

\author{Cristiano Arbex Valle$^1$ \and John E Beasley$^2$}

\begin{document}

\maketitle

\begin{center} 
{\footnotesize

$^1$Departamento de Ci\^{e}ncia da Computa\c{c}\~{a}o, \\
 Universidade Federal de Minas Gerais, \\
 Belo Horizonte, MG 31270-010, Brasil \\
 arbex@dcc.ufmg.br \\ \vspace{0.3cm}
$^2$Brunel University \\ Mathematical Sciences, UK \\ john.beasley@brunel.ac.uk \\
}
\end{center} 

\begin{abstract}
In this paper we investigate the problem of simultaneously allocating orders and mobile storage racks to static pickers.  Here storage racks are allocated to pickers to enable them to pick all of the products for the orders that have been allocated to them. Problems of the type considered here arise in facilities operating as robotic mobile fulfilment systems.

We present a formulation of the problem of allocating orders and racks to pickers as an integer program and discuss the complexity of the problem. We present two heuristics (matheuristics) for the problem, one using partial integer optimisation, that are directly based upon our formulation. 

We also consider the problem of how to sequence the racks for presentation at each individual picker and formulate this problem as an integer program. We prove that, subject to certain conditions being satisfied, a feasible rack sequence for all orders can be produced by focusing on just a subset of the orders to be dealt with by the picker.

Computational results are presented, both for order and rack allocation, and for rack sequencing, for randomly generated test problems (that are made publicly available) involving up to 1000 products, 200 orders, 500 racks and 10 pickers.

\end{abstract}

{\bf Keywords:}  integer programming, inventory management, mobile storage racks, order picking, partial integer optimisation, rack sequencing, robotic mobile fulfilment systems

\section{Introduction}

In this paper we study two problems that arise in mobile robotic fulfilment centres. In such centres mobile robots bring moveable racks of shelves containing inventory to static pickers, so that these pickers can pick the items needed for customer orders.
Robotic Mobile Fulfilment Systems (RMFS)
date from 1987~\citep{koster2019} and 
are increasing in popularity and use, with there now being over 30 suppliers of such systems.
\cite{koster2019} notes that such systems are ideally suited for Internet retailers in the B2C (business to consumer) market, who require the picking of relatively small orders (so with only a few inventory items to consolidate per order) from a wide product range.

With respect to the relevance and usage of mobile robotic fulfilment systems  the archetypal example of such a B2C company would be Amazon, who are well-known for their use of Kiva robots to move racks of shelves to pickers in their more modern fulfilment centres. 
At the time of writing Amazon operates some 175 fulfilment centres around the world involving more than 150 million square feet of space. 26 of these centres make use of over 100,000 Kiva robots for bringing racks of inventory to static pickers for order picking~\citep{amazon19, amazon19a}. Other companies operate similar systems in the B2C market, for example Alibaba with Quicktron robots.

Since the pioneering use of such robots by Amazon many other companies have followed suit~\citep{banker16}, with a recent report~\citep{sanders19} estimating that the worldwide sales of warehousing and logistics robots will reach US\$30.8 billion by the end of 2022, with robot unit shipments reaching 938,000 units per year by 2022.

In this section we first illustrate the order allocation, rack allocation and rack sequencing problems considered in this paper. We then discuss the motivation behind our work and the research gap we address via our research questions. We also detail what we believe to be the 
contribution of this paper to the literature. 
Finally we outline the structure of the paper.

\subsection{Order allocation, rack allocation and rack sequencing}

To illustrate the problems dealt with in this paper consider Figure~\ref{fig1} which shows, in schematic form,  a 
two-dimensional view of a rack subdivided into storage locations in which products of different types can be stored. The key feature of a rack such as the one shown in 
Figure~\ref{fig1} is that it is \textbf{\emph{mobile}}, easily moved. Movement of the rack is achieved by a small battery-powered mobile robot positioning itself under the rack and raising the lifting platform shown in Figure~\ref{fig1}, thereby 
lifting the rack off the ground. The mobile robot can then move the rack to any desired location. Once moved to a new  location the robot can remain with the rack, or lower the lifting platform to deposit the rack on the ground and move off to perform other duties. 

Racks of the type shown in Figure~\ref{fig1} are typically square in nature (so with four equally sized faces when viewed in three dimensions). Depending on the types of products stored (e.g.~their sizes) all four faces may be used to store different products. Alternatively only two opposite faces may be used, or just one face. 

\begin{figure}[!ht]
\centering
\includegraphics[width=0.5\textwidth]{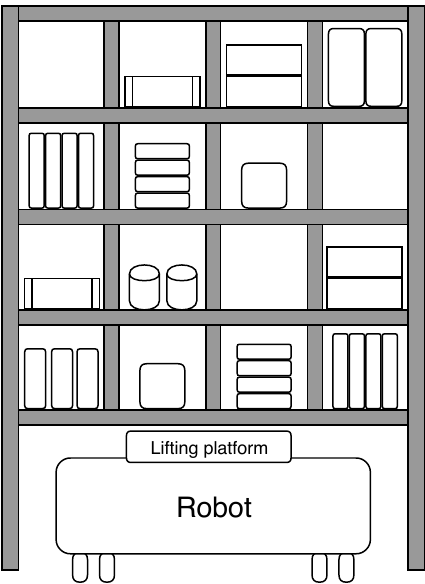}
  \caption{Mobile rack}
  \label{fig1}
\end{figure}

These mobile racks, also referred to as \textbf{\emph{pods}}, or \textbf{\emph{inventory pods}}, are commonly  found in  robotic mobile fulfilment systems. In such systems there are many such mobile racks within a facility. This can be seen in Figure~\ref{fig2} which, in schematic form, shows such a facility. In Figure~\ref{fig2}   each small square seen is a rack and they are placed in a standard rectangular pattern with aisles and cross-aisles. 
Also shown in Figure~\ref{fig2} are a number of stations at which pickers (typically human pickers) are positioned. Robots bring racks to these pickers in order that the set of products needed to fulfil a customer order can be picked from the racks for onward transmission to the customer.

\begin{figure}[!ht]
\centering
 \includegraphics[width=0.5\textwidth]{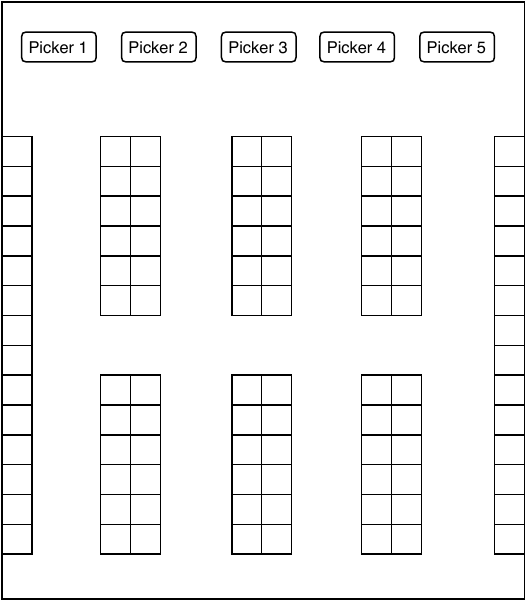}
  \caption{Rack storage and picking layout}
  \label{fig2}
\end{figure}

An example station at which a picker works can be seen in schematic form in 
Figure~\ref{fig3}. Racks are placed  in a queue leading past each picker and successively presented to the picker. The picker picks items from the presented rack and places them in one or more of the bins, each bin being associated with a single customer order (in some contexts these bins are known as crates or totes). 
Once all the items in an order have been picked the bin containing the completed order is put onto the conveyor, being carried away for further processing,  and a new empty bin is positioned in the space so created. 

\begin{figure}[!ht]
\centering
 \includegraphics[width=0.5\textwidth]{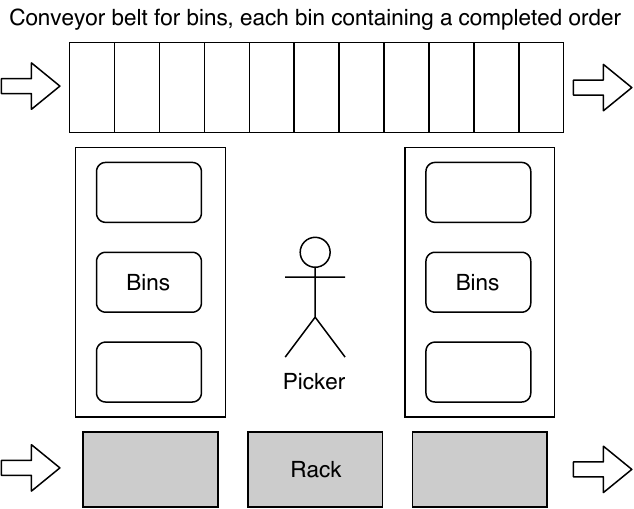}
  \caption{Picker station}
  \label{fig3}
\end{figure}

When picking items from the rack presented to them the picker can make use of multiple bins, 
i.e.~pick items associated with two or more orders from the same rack.
Note that the picker can only pick from a single face of the rack presented to them.
Once a picker has completely finished with a rack then it is moved on and (typically) taken back to the storage area. 
The next rack in the queue leading past the  picker is then presented and the process repeats. In some circumstances rack revisits (so the same rack being presented to the picker more than once in the sequence of rack presentations) are allowed.

Note here that Figures~\ref{fig1},\ref{fig2},\ref{fig3} have all been given in schematic form in order to illustrate to the reader the relevant fundamentals of a robotic mobile fulfilment system. Presenting such schematic form figures, rather than pictures of real-world systems, is common in the literature, e.g.~\citep{boysen17a, hansen18, li17, xiang18, yuan17, zou17}.

With regard to the operation of a facility such as shown in Figure~\ref{fig2} note that the robots themselves  can travel underneath any racks that are sitting immobile in storage, thereby leaving the aisles free for robots actually moving racks, and this significantly contributes to reducing aisle congestion. Moreover it is common to  operate a 
mixed-shelves/scattered storage policy, so that the same product is stored in multiple racks~\citep{weidinger18b}. Such a policy makes intuitive sense since if a popular product was only (for example) stored in one rack that rack may be very busy in moving between different pickers and so potentially delay the fulfilment of customer orders. Note here that with mobile racks one of the disadvantages of operating a scattered storage system when the racks are 
fixed~\citep{weidinger2019}, namely intensified replenishment due to having to workers having to walk amongst the fixed racks with new inventory to store in scattered locations, is much reduced as the racks are brought to replenishment stations by the robots.
Note also here that since the focus of this paper is on picking to fulfil customer orders we, for convenience, have not shown in Figure~\ref{fig2} other features that are also necessary for robotic mobile fulfilment systems to operate, such as replenishment stations and robot recharging stations.

\subsection{Motivation, research gaps, research questions}

Below we discuss our motivation in undertaking the research presented in this paper. We also discuss the research gaps we identified and the research questions our paper addresses.

\subsubsection{ Motivation}

The motivation behind the research undertaken in this paper was to investigate whether computationally effective algorithms for  order allocation, rack allocation and rack sequencing could be developed.
In particular we were motivated \emph{\textbf{to base any algorithms developed 
on explicit mathematical  formulations of the decision problems involved}}. 

As the reader may be aware metaheuristic algorithms 
for the solution of many decision problems encountered within Operations Research can often be developed without any recourse to an explicit mathematical  formulation of the problem under consideration. By contrast, in this paper, our motivation was both to present explicit mathematical  formulations of the problems involved, and to use these formulations as a basis for computationally effective solution algorithms.

\subsubsection{Research gaps}
Given the operation of a robotic mobile fulfilment system as discussed above this paper focuses on two decision problems:
\begin{itemize}
\item  Firstly, we 
consider the problem of simultaneously  allocating customer orders and mobile storage racks to pickers.  Here storage racks must be allocated to pickers to enable them to pick all of the products for the customer orders that have been allocated to them.  
\item Secondly, we consider  the problem of sequencing the racks for presentation to each picker so as to enable them to feasibly pick all of their allocated orders given that there are a limited number of bins that can be in use simultaneously 
at each picker station.
\end{itemize}

As a result of our survey of the literature (presented later below) we identified the following research gaps in relation to these two decision problems:
\begin{itemize}

\item the lack of mathematical formulations and solution approaches for simultaneously allocating both orders and racks to multiple pickers

\item the lack of mathematical formulations and solution approaches for the rack sequencing problem that explicitly consider rack inventory positions and decide the number of units of each product allocated to each order from each rack, both when a rack can only visit a picker once and when rack revisits are allowed

\end{itemize}

\subsubsection{Research questions}
Our research questions directly address the research gaps identified above, Specifically:

\begin{itemize}
\item Is it possible to develop a mathematical formulation for the problem for simultaneously allocating both orders and racks to multiple pickers? Moreover, if the answer to this question is positive, can that formulation then be used as an explicit basis for the development of computationally effective solution algorithms? 

\item  Is it possible to develop a mathematical formulation for the rack sequencing problem that explicitly considers rack inventory positions and decides the number of units of each product allocated to each order from each rack, both when a rack can only visit a picker once and when rack revisits are allowed? Moreover, if the answer to this question is positive, can that formulation then be used as an explicit basis for the development of a computationally effective solution algorithm?

\end{itemize}
We believe that the work presented below answers these research questions.

\subsection{Contribution}

We believe that  the contribution of this paper to the literature is:
\begin{itemize}

\item to be one of the first in the literature to present an optimisation based approach for simultaneously allocating both orders and racks to multiple pickers

\item to use our order and rack allocation formulation as the basis for two matheuristics for the problem of simultaneously allocating both orders and racks to multiple pickers

\item to present an innovative formulation for the rack sequencing problem that explicitly considers rack inventory positions and decides the number of units of each product allocated to each order from each rack, both when a rack can only visit a picker once and when rack revisits are allowed

\item to prove that, subject to certain conditions being satisfied, in   generating a feasible rack sequence for a single picker we can  initially neglect  all orders which comprise a single unit of one product since a feasible rack sequence for the remaining orders
can be easily modified to incorporate the neglected orders 

\item to investigate how our approaches for order and rack allocation, and rack sequencing, perform computationally for test problems that are made publicly available for use by future researchers

\end{itemize}

\subsection{Structure of the paper}

The structure of this paper is as follows. In 
Section~\ref{sec:review} we review the relevant literature on the two problems considered, order and rack allocation to pickers and rack sequencing
in a robotic mobile fulfilment system. 
In Section~\ref{sec:formjeb} we present our formulation which decides the orders and racks to be allocated to pickers. We present a number of additional constraints which can be added to the formulation and discuss the complexity of the problem.
In Section~\ref{sec:heuristics} we present two heuristics for the order and  rack allocation problem based upon our formulation. 
In Section~\ref{sec:formjebrack} we present our formulation to decide the rack sequence to be adopted at a picker when each mobile rack only visits the picker once (so no rack revisits to the picker are allowed). We  then extend this formulation to the case when rack revisits to the picker are allowed. We prove that, subject to certain conditions being satisfied, a feasible rack sequence for all orders can be produced by focusing on just a subset of the orders to be dealt with by the picker. 
In Section~\ref{sec:discussion} we discuss the application of the work presented in this paper within a dynamic environment (as new orders appear), as well as discuss the wider applicability of our work to other automated 
(non-mobile robot) fulfilment systems.
Section~\ref{sec:results} presents our computational results for the test problems we examined. Finally in Section~\ref{sec:conclusions} we present our conclusions and discuss future research directions.

\section{Literature review}
\label{sec:review}

Material handling activities can be differentiated as 
\emph{\textbf{parts-to-picker}} systems, in which (typically) automated units deliver the items to stationary pickers, and \emph{\textbf{picker-to-parts}} systems, in which pickers walk/ride through the warehouse collecting requested items. This paper deals with the parts-to-picker system 
associated with robotic mobile fulfilment systems,  
where automated units (robots) deliver racks of products to pickers in order to fulfil customer orders. In problems of this kind decisions need to be made as to the batching of  orders (i.e.~the set of orders to be assigned to each picker) as well as which racks should be moved to which picker, and too the sequence in which racks are presented to a picker.

There is an extensive literature on the problems of order batching and picking with reference to picker-to-parts systems. An early survey relating to different picking strategies can be found in~\cite{deKoster2007}, where order picking was shown to be a critical activity.
A very comprehensive state of the art classification and review of picking systems has also recently been presented 
by~\cite{gils2018}. 

By contrast however there is significantly less literature dealing with the parts-to-picker system as represented by robotic mobile fulfilment systems. Two very recent surveys make this point:
\begin{itemize}

\item \cite{azadeh2019} present a survey of work relating to warehousing systems involving automated and robotic handling. They state that robotic mobile fulfilment systems have hardly been studied scientifically, despite their increasing use in practice.

\item \cite{boysen19} present a survey of work relating to warehousing systems designed to allow online retailers to fulfil orders to personal consumers (so B2C, business to consumer).  With particular reference to one problem we consider in this paper, that of assigning orders to pickers, they state that to the best of  their knowledge no literature exists regarding the assignment of orders to picking stations.

\end{itemize}

\subsection{Literature review: structure}

Robotic mobile fulfilment systems  involve a number of different decision problems that are interrelated, but which for ease of solution are often treated separately.
\cite{weidinger18a}, in Appendix C of the online supplement associated with their paper, suggest a four level hierarchy for the picking process, which can be viewed as: 
\begin{enumerate}
\item order selection and order allocation to pickers; where order selection involves deciding the orders that will be processed next (from the complete set of orders not yet dealt with) 
\item rack assignment and rack sequencing at each picker; assigning racks for the orders allocated to each picker, deciding the rack sequence to be adopted at each picker 
\item rack storage assignment; assigning each rack a storage location when it needs one (e.g.~after a visit to a picker)
\item robot assignment for required rack movements (e.g.~storage to picker, picker to storage) and rack and robot movement path planning (route planning)
\end{enumerate}
In our view this hierarchy is ordered in the sense that the earlier levels involve decisions that have a larger impact on system effectiveness than the latter levels.  

The work presented in this paper deals with the first two levels of this hierarchy and is surveyed below. 
Work dealing with the remaining two levels is outside the scope of this paper, but the interested reader is referred to:
\begin{itemize}
\item \cite{beckschafer2014, nigam14, roy19, tessensohn2020, weidinger18a, xiang18, yuan18, zou18} for papers dealing with  storage issues; and~
\item \cite{defoort14, herrero11, qiu02, shahriari18, vis06, wang2019, xue19, yu16, zhang19}
for papers dealing with  movement issues.
\end{itemize}

To structure the literature survey presented below  of work relating to order selection/allocation and rack assignment/sequencing 
we make use of the structure suggested 
by~\cite{azadeh2019}, namely: 
\begin{compactitem}
\item system analysis, focusing on modelling techniques to estimate the performance
of the system without focusing on any optimisation
\item design optimisation, focusing on hardware optimisation of the system (e.g.~system layout)
\item  operations planning and control, focusing on the software optimisation
of the system (e.g.~rack sequencing).
\end{compactitem}
Table~\ref{table_lit} presents a short summary overview of the literature surveyed.

\begin{table}[!ht]
\centering
{\scriptsize
\renewcommand{\tabcolsep}{1mm} 
\renewcommand{\arraystretch}{1.3} 
\begin{tabular}{|llll|}
\hline
Category & Paper & Research focus & Methodology \\
\hline
System analysis &   \cite{lamballais17} &   Maximum order throughput, order cycle time, 
  & Queueing theory, simulation
   \\
&    &   Robot utilisation  &    \\
                           &   \cite{bozer18} &   Kiva  and miniload  systems &  Simulation  \\
                           &    \cite{hansen18} &   Interviews  &  Case study  \\       
                           &    &     &    \\
Design optimisation &  \cite{yuan17} &  Optimal number and velocity of  robots    &  Queueing theory, simulation \\
                           &    &     &    \\
Operations planning  &  \cite{boysen17a}  & Order and rack sequencing    &   Mixed-integer programming, \\
and control      &    &     &  heuristics   \\
                           &  \cite{li17}  &   Rack allocation  &    Integer programming, \\
     &    &     &  heuristic  \\
                           &    \cite{zou17} &  Robot allocation   & Queueing theory, simulation,    \\
                           &    &     &  heuristic  \\
                           &   \cite{xiang18} &  Order batching, rack allocation   &  Integer programming,  \\
   &    &     &   heuristics \\
\hline
\end{tabular}
}
\caption{Literature: summary overview}
\label{table_lit}
\end{table}

Note here that in the literature survey below we are concerned with the problem as described above, so robots bringing mobile racks to stationary pickers. It is important to be aware here that in the literature there are papers 
(e.g.~\cite{boysen17,foroughi20}) that deal with mobile racks, but where the mobility relates to creating aisles between racks. In such situations racks are typically long and closely packed together. When a product has to be picked it may be that a number of adjacent racks have to be moved so that an aisle can be created to access the rack holding the required product. This problem is outside the scope of this paper.  Similarly work dealing with automated warehouses, but where the automation does not relate to mobile robots moving racks to pickers, is outside the scope of this paper.

\subsection{System analysis}

\cite{lamballais17} consider a RMFS and used queueing theory to analytically estimate maximum order throughput, average order cycle time and robot utilisation for both single-line (one product) orders and multi-line (multi-product) orders. Simulation was used to validate their  analytic estimates. They conclude that their
analytical models accurately estimated robot utilisation, workstation utilisation and order cycle time. In addition they found that maximum order throughput was quite insensitive to the length-to-width ratio of the storage area and that maximum order throughput was affected by the location of the pickers around the storage area.

\cite{bozer18} present a paper dealing with a simulation based comparison of two goods-to-picker systems, one a Kiva system (such as we consider in this paper) and the other a miniload system. In a miniload  system products are typically stored  in trays (containers, bins) in long aisles of fixed racks. To satisfy a customer order one (or more) trays containing the appropriate products ordered are retrieved from the aisles
by an automated storage/retrieval  machine  and delivered to a picker by conveyor. Once products have been taken by the picker from a tray it is returned to the fixed racks for storage (again using the automated machine). They conclude that (in the context of their simulation)
a miniload system with four aisles yields approximately the same throughput as a Kiva system
with 50 robots. They state that their results indicate that a favourable balance between picker and equipment
(miniload or robot) utilisation is achieved when the expected tray retrieval-storage cycle time is
approximately equal to the expected pick time per tray.

\cite{hansen18} present a case study relating to the implementation and operation of a RMFS  for order picking of consumer goods within an e-commerce context. They use 
semi-structured interviews with selected personnel from both the provider of the RMFS and the operator of the RMFS (a third-party logistics provider).
 They note that interviewees stressed that the productivity of the system was affected by
the~\enquote{hit rate}, the number of
successive picks that could be made from an inventory pod at one
picker before the pod was moved on.

\subsection{Design optimisation}

\cite{yuan17}  consider a RMFS and use queueing theory to examine the differences between each picker having dedicated robots and all pickers sharing robots. Simulation was used to validate their queueing theory approach. They calculate the optimal number and velocity of warehouse robots in order to achieve effective 
operation in terms of the total throughput time, both with and without congestion.

\subsection{Operations planning and control}

\cite{boysen17a} consider the problem of a single picker with a given set of orders to be picked from a given set of racks. In such  situations decisions need to be made as to how the orders are sequenced for consideration by the picker, and too how racks are to be sequenced for presentation to the picker, given that there is a constraint on the number of orders that can be processed in parallel. They present a mixed-integer programming model for the problem based upon the use of time slots, where a time slot comprises the time interval where a certain 
subset of orders is being processed from a certain rack that has been presented to the picker. They discuss the computational complexity of the problem. Heuristic algorithms based upon simulated annealing are given as well as computational results.
In their paper they do not explicitly  consider the number of units of each product stored in a  rack, rather they assume that if a rack contains a product then it has sufficient of that product to satisfy all orders requiring that product. Rack revisits (so the same rack being presented to the picker more than once in the sequence of rack presentations) are allowed.
They report that, for the problems they examined, optimising order and rack sequencing has the potential to more than halve the number of robots required to timely supply a picker compared to a simple rule-based approach that is often applied in the real-world. 

\cite{li17} consider the situation where there is a single picker and formulated a zero-one integer program for the problem of deciding the racks to be allocated to the picker based upon minimising total rack travel distance (whilst supplying all orders). They also present a three stage heuristic for the problem. Rack sequencing is not considered in their work. 
Computational results are given for single picker problems using both their integer programming and heuristic approach.

\cite{zou17} propose allocating robots to pickers 
based on considering picker handling speeds and present a neighbourhood search algorithm to find such an allocation. 
They use a semi-open queueing network  approach which was validated using simulation to examine the performance of varying robot to pickers rules. 
They also consider the sizing of the rectangular storage areas on the shelves in each mobile rack.

\cite{xiang18} consider the problem of the products to store in each rack and formulate the problem as an  integer program maximising the total similarity of products stored in the same rack. They also consider the problem of batching orders together so that all orders are assigned to some batch, and too racks are assigned to batches, which they formulate as a zero-one integer program. For this order batching problem they present a heuristic to generate an initial solution as well as a variable neighbourhood search heuristic. 
Their order batching approach allows
racks to be allocated to more than one batch,
 so does not address the issue of
sequencing  batches for allocation to pickers.
  Computational results are given.

\subsection{Comment}
Table~\ref{table_lit} supports the insight from the two recent literature reviews mentioned above~\citep{azadeh2019, boysen19} that, in terms of the volume of work presented in the scientific literature, much remains to be done with regard to  robotic mobile fulfilment systems.

With regard to the contribution made by this paper to the literature our work fits within the operations and control section of the structure suggested 
by~\cite{azadeh2019}.

\section{Formulation - order and rack allocation}
\label{sec:formjeb}

In this section we first present our formulation for simultaneously 
 allocating orders and racks  to pickers. We then present a number of additional constraints which can be added to the formulation and discuss the complexity of the problem.

Our formulation assumes that
we are dealing with a fixed set of orders and we need to make decisions as to the orders and racks to allocate to each picker so as to enable the picking of the orders  allocated. In other words we are not attempting in our formulation to deal with the situation where previous decisions need to be revisited/revised as new orders appear, rather we need to make decisions with  regard to a fixed set of orders. We discuss in a later section below how to adapt the approach given here for the situation where new orders appear.

For ease of presentation we shall initially assume here that:
\begin{compactitem}
\item each rack (if used) can only be allocated to one picker, and;
\item each mobile rack has only one face from which products can be picked.
\end{compactitem}
However we  indicate later in this section  how both these  assumptions can be removed.

Suppose that we have  $N$ products  that can be ordered by customers.  We have $O$ orders to be allocated to pickers where each order $o$ requires $q_{io}$ units of product $i$. We have $R$ mobile racks which can be allocated (moved)  to pickers. There are $P$ pickers, where picker $p$   is to be allocated  $C_p$ orders, where $\sum_{p=1}^P C_p \leq O$. Rack $r$ contains $s_{ir}$ units of product $i$. Without loss of generality assume that the pickers are indexed in decreasing $C_p$ order (so $C_p \geq C_{p+1}~p=1,\ldots,(P-1)$)

Our formulation for allocating orders and racks to pickers involves the following zero-one decision variables:
\begin{itemize}
\item $x_{op}=1$ if order $o$ is allocated to picker $p$, zero otherwise
\item $u_r=1$ if  rack $r$ is used, so allocated to some picker, zero otherwise
\item $y_{rp}=1$ if  rack $r$ is allocated to picker $p$, zero otherwise
\end{itemize}

Our  formulation is:
\begin{optprog}
\optaction[]{min} &  \objective{\sum_{r=1}^R  u_r \label{eq1} }\\

subject to: & \sum_{o=1}^O x_{op} & = & C_p & \forall p \in \{1,\ldots,P\}
\label{eq2} \\

& \sum_{p=1}^P x_{op} & \leq & 1 & \forall o \in \{1,\ldots,O\}
\label{eq2a} \\

& \sum_{p=1}^P y_{rp} & = & u_r & \forall r \in \{1,\ldots,R\}
\label{eq3} \\

& \sum_{r=1}^R s_{ir}y_{rp} & \geq & \sum_{o=1}^O q_{io}x_{op} & 
\forall i \in \{1,\ldots,N\};  p \in \{1,\ldots,P\}:~\sum_{o=1}^O q_{io} \geq 1
 \label{eq4} \\

& x_{op} & \in & \{0,1\} & \forall o \in \{1,\ldots,O\};  p \in \{1,\ldots,P\}
\label{eq5} \\
& u_r & \in & \{0,1\} &\forall r \in \{1,\ldots,R\}
\label{eq5a} \\
& y_{rp} & \in & \{0,1\} & \forall r \in \{1,\ldots,R\}; p \in \{1,\ldots,P\}
\label{eq6} 
\end{optprog}

In our objective, Equation~(\ref{eq1}), we minimise the total number of racks allocated to pickers. As discussed 
in~\cite{boysen17a, hansen18} this objective contributes to less movement of racks from the storage area to the picking area, and encourages the picking of multiple products for different orders from the same rack.
Equation~(\ref{eq2}) ensures that we allocate orders to each picker so as to use each picker to capacity. 
Equation~(\ref{eq2a}) ensures that an order can be allocated to at most one picker. Equation~(\ref{eq3}) ensures that each rack is either allocated to one picker, or not used at all.

Equation~(\ref{eq4}) is the product supply constraint and ensures that, for each product $i$ and each picker $p$, the number of units of that product available from the racks assigned to that picker (so  $ \sum_{r=1}^R s_{ir}y_{rp}$) is sufficient to meet the required number of units of product $i$ at picker $p$ given the orders allocated to the picker (so  $ \sum_{o=1}^O q_{io}x_{op}$). This constraint is obviously only applicable if some order requires product $i$, so $\sum_{o=1}^O q_{io} \geq 1$.
Hence Equation~(\ref{eq4})  ensures that appropriate racks are assigned to picker $p$ so as to enable  the picking of all of the products associated with the orders assigned to the picker. Equations~(\ref{eq5})-(\ref{eq6}) are the integrality constraints.

With regard to the formulation presented above note especially here that \emph{\textbf{multiple orders can be picked from the same rack by a picker}}. In other words we are not restricted to just picking a single order from each rack chosen to be allocated to some picker.

In terms of the formulation presented above, Equations~(\ref{eq1})-(\ref{eq6}), note that the $[u_r]$ variables relating to the usage (or not) of rack $r$ could potentially be deleted from the formulation. This would necessitate changing the objective, Equation~(\ref{eq1}), to minimise $\sum_{r=1}^R \sum_{p=1}^P y_{rp}$ and changing Equation~(\ref{eq3}) to $ \sum_{p=1}^P y_{rp} \leq 1$. However we have introduced the $[u_r]$ variables into our formulation because their presence gives the solver we used, Cplex~\citep{cplex128}, an opportunity to choose such a variable for branching, thereby potentially identifying explicitly that a particular rack $r$ is not used in just one branching operation. If such a  variable is not present in the formulation the solver (potentially) has to branch on $P$ $y_{rp}$ variables before identifying that a particular rack $r$ is not used. Limited computational experience (not reported below for space reasons) supported our use of these variables in the formulation presented above.

\subsection{Other constraints}

There are a number of other constraints which are relevant to the formulation presented above and these are discussed below. The first  constraint presented below helps reduce the size of the problem we need to solve. The other constraints presented below relate to extending the formulation to deal with additional restrictions.

\subsubsection{Required products}
Given the $O$ orders then we know the products needed. If a rack contains  none of the required products then clearly it is redundant and will never be assigned to a picker, and so can be removed from the problem. The constraint that ensures this is:
\begin{equation}
u_r = 0~~\forall r \in \{1,\ldots,R\}:~\sum_{o=1}^O \sum_{i=1}^N s_{ir}q_{io}=0
\label{eq7}
\end{equation}
In Equation~(\ref{eq7}) the term $\sum_{o=1}^O \sum_{i=1}^N s_{ir}q_{io}$ is zero if and only if rack $r$ does not contain any of a required product.

\subsubsection{Rack faces}
The formulation presented above assumes that the mobile rack has only one face from which to pick products. In some applications the rack has different faces (typically two opposite faces, or all four faces) from which products can be picked, with the possibility of picking different products from each face. 
However, irrespective of the number of faces the rack has, the picker can only pick from a single face of the rack presented to them.
Our formulation can be easily modified to deal with such situations.  We create 
\enquote{artificial} copies of each rack, one copy for each face that can be used. 

To illustrate this suppose that we have just one of the $R$ racks,  rack $r$, where all four faces can be used. Arbitrarily associate one of the faces with rack $r$ and create three new racks which are artificial face copies of rack $r$ (one for each of the other three faces). Let these copies be the racks $R+1$, $R+2$ and $R+3$ so that $F_r$, the set of racks which are the faces of rack $r$, is given by $F_r =[r, R+1,R+2,R+3]$. Set $R \leftarrow R+3$. For each rack $j \in F_r$ the values for rack product supply ($s_{ij}$) are set equal to the values that apply to the face of rack $r$ which is to be presented to the picker.
Then we add the constraint:
\begin{equation}
\sum_{j \in F_r} u_j \leq 1 \label{eq9}
\end{equation}
Equation~(\ref{eq9}) ensures that for rack $r$ at most one of the racks in the set $F_r$ (i.e.~at most one rack face) can be used.

\subsubsection{Multiple picker allocation for racks}

Above we assumed  that each rack (if used) can only be allocated to one picker. If we remove this assumption, so allow each rack to be allocated to two (or more) pickers, then this can be easily done. Let $\pi_{irp}$,  integer $\geq0$, be the number of units of product $i$ associated with rack $r$ that are used at picker $p$. Here we need only be concerned with $\pi_{irp}$ if $ s_{ir} > 0$, i.e.~if rack $r$ contains some of product $i$. 
Then we remove Equations~(\ref{eq3}),(\ref{eq4})
from the formulation and introduce:
\begin{optprog}
& \sum_{r=1}^R \pi_{irp}  & \geq & \sum_{o=1}^O q_{io}x_{op} &
 \forall i \in \{1,\ldots,N\};  p \in \{1,\ldots,P\}:~\sum_{o=1}^O q_{io} \geq 1
\label{jebex1} \\


 &  \pi_{irp}  & \leq & s_{ir} y_{rp} &
\forall i \in \{1,\ldots,N\}; r \in \{1,\ldots,R\}; p \in \{1,\ldots,P\}
\label{jebex11} \\
& \sum_{p=1}^P \pi_{irp}  & \leq & s_{ir} & 
\forall i \in \{1,\ldots,N\}; r \in \{1,\ldots,R\} \label{jebex2} \\
& y_{rp}  & \leq & u_r & 
\forall r \in \{1,\ldots,R\}; p \in \{1,\ldots,P\}\label{jebex3} 
\end{optprog}
Equation~(\ref{jebex1}) is the product supply constraint and ensures that, for each product $i$ and each picker $p$, the number of units of that product supplied from the racks assigned to that picker (so  $ \sum_{r=1}^R \pi_{irp}$) is sufficient to meet the required number of units of product $i$ at picker $p$ given the orders allocated to the picker (so  $ \sum_{o=1}^O q_{io}x_{op}$). This constraint is analogous to Equation~(\ref{eq4}) presented previously above.
Equation~(\ref{jebex11}) ensures that if a rack is not allocated to a picker (i.e.~$y_{rp}=0$) then no product can be supplied from that rack to the picker (i.e.~we have $\pi_{irp}=0~i=1,\ldots,N$).
Equation~(\ref{jebex2}) ensures that over all pickers $p$ we do not use more units of product $i$ than are available on  rack $r$.
Equation~(\ref{jebex3}) ensures that $u_r$ is one if rack $r$ is used by some picker $p$, and conversely that $y_{rp}$ is zero if $u_r$ is zero.

In Equations~(\ref{jebex1})-(\ref{jebex3}) above we have explicitly defined the variables $\pi_{irp}$,  the number of units of product $i$ associated with rack $r$ that are used at picker $p$, as integer. This corresponds to the real-world situation modelled, where the number of units of any product picked can only be integer. Note here that relaxing this integrality assumption could lead to fractional solutions, i.e.~in any numeric solution the variables $\pi_{irp}$ are not naturally integer. As an illustration of this suppose we have an order requiring just one unit of a particular product and that this product is available on two racks which, because of other orders, are both allocated to the picker.
If we allow fractional $\pi_{irp}$ variables then to satisfy Equations~(\ref{jebex1})-(\ref{jebex3})
it would be possible to take 0.5 (say) of the product from one rack,  0.5 from the other. 
Of course if it is judged desirable (e.g.~for computational reasons)  to relax the integrality requirement on the $\pi_{irp}$ variables so that they become non-negative continuous variables then some appropriate rounding procedure can be used to generate a feasible integer solution from any fractional  $\pi_{irp}$ values.

As we allow each rack to be allocated to two (or more) pickers the objective function can either remain as minimise $\sum_{r=1}^R  u_r$ or be amended to account for multiple rack usage, so become minimise $\sum_{r=1}^R  \sum_{p=1}^P y_{rp}$.
On a technical note here if we allow multiple picker allocation for racks then in the individual picker rack sequencing formulation given below we, to respect rack inventory position constraints, need to use the values $\pi_{irp}$ as being the inventory of product $i$ on rack $r$ when considering rack sequencing at picker $p$.

\subsubsection{Picker workload}

With respect to Equation~(\ref{eq2}), which specifies  the number of orders allocated to each picker, we did consider allowing the number of orders allocated to a picker, which is a proxy for picker workload, to be decided by the optimisation. For example this could be done by changing Equation~(\ref{eq2}) to 
$\sum_{o=1}^O x_{op} \leq  C_p~\forall p \in \{1,\ldots,P\}$, so that we maintain an upper limit on the number of orders allocated to each picker, and 
introducing a constraint such as:
\begin{equation}
\sum_{o=1}^O \sum_{p=1}^P x_{op} \geq C^* \label{jebwork}
\end{equation}
to ensure that we process at least $C^*$ orders in total. 

However our view was that, since this might result in unbalanced picker workloads, it was better to simply specify the number of orders allocated to each picker. Varying the formulation, for example
to impose Equation~(\ref{jebwork}) as well as introduce lower and upper limits on the number of orders allocated to each picker; or
to introduce some measure of workload per order to more explicitly account for workload per picker,  
is clearly possible.

\subsection{Complexity}

\begin{theorem}
The optimisation version of the order and rack allocation problem, which aims to find the minimum number of racks required to allocate $\sum_{p = 1}^P C_p \leq O$ orders to pickers, is NP-hard in the strong sense.
\label{th1}
\end{theorem}

\begin{proof}
We can prove this theorem by reference to the minimum set cover problem, whose decision version is one of Karp's 21 
NP-complete problems \citep{garey79,karp1972}. Let $X$ be a finite set and let $\mathcal{F} = \{S_1, \ldots, S_T\}$ be a family of $T$ subsets of $X$, that is, $S_t \subseteq X$, such that $\bigcup_{t = 1}^T S_t = X$. A collection of subsets $C \in \mathcal{F}$ is a {\it cover} of $X$ if $X = \bigcup_{S \in C} S$. The size $|C|$ of a cover $C$ is the number of subsets in $C$. The minimum set cover problem (also known as the minimum cover problem) is to find a subset $C$ with minimum $|C|$.

Consider an instance of the order and rack allocation problem involving just a single order $o$ requiring one unit of a product corresponding to each element $i \in X$, that is, $q_{io} = 1, \;i \in X$. For each subset $S_t \in \mathcal{F}$, introduce a rack $r$ where $s_{ir} = 1$ if $i \in S_t$. Consider also a single picker $p$ with $C_p = 1$, so the picker only has to deal with the single order $o$. 

Then finding the minimum cover  of $\mathcal{F}$ 
is equivalent to finding the minimum number of racks required for $p$ to fully collect $o$. The solution can be transformed back to a solution of the minimum set cover problem by mapping each rack $r$ where $y_{rp} = 1$ to its corresponding subset $S_t$.

The transformation is 
pseudo-polynomial 
and thus the optimisation version of the order and rack allocation problem is NP-hard in the strong sense. 
\end{proof}

\subsection{Rack allocation only}

In the formulation presented above we  simultaneously allocated both orders and racks to multiple pickers. This involves two elements of choice:
\begin{itemize}
\item decide the orders to be allocated to each individual picker
\item decide the racks to be allocated to each individual picker
\end{itemize}
Clearly our formulation deals with both choice elements 
\textbf{\emph{simultaneously}}. It is possible that in particular applications  orders could have been preassigned to pickers, so the element of choice simply reduces to deciding the racks to be allocated to each individual picker. Since one element of choice has been removed then we might reasonable expect that this problem would be (computationally) less challenging than the specific problem  considered in this paper, where we have both elements of choice. Unfortunately, the remaining problem is still NP-hard in the strong sense as the proof of Theorem~\ref{th1}  shows.

Note here that our formulation is easily adapted to deal with the situation where we have just the one choice element which only relates to
 rack allocation. This is because preassignment of orders to pickers simply implies that  all of the variables $x_{op}$ (one if order $o$ is allocated to picker $p$, zero otherwise) have been assigned values. Our formulation 
(Equations~(\ref{eq1})-(\ref{eq7})) 
then reduces to one where we are seeking to minimise the total number of racks used.

\section{Heuristics}
\label{sec:heuristics}

As indicated by the complexity theorem above we are very likely to need to resort to a heuristic solution procedure for the order and rack allocation problem when any optimal procedure, for example linear programming relaxation based tree search, reaches its effective computational limit. In  this section we present two heuristics for the problem of allocating orders and racks to pickers. Both of these heuristics draw directly on the mathematical formulation of the problem presented in the previous section and hence can be classed as 
\emph{\textbf{matheuristics}}~\citep{boschetti2009}. Our heuristics make direct use of Cplex so we also comment as to why we might use these heuristics as compared with simply applying Cplex by itself.

\subsection{Single picker based (SPB) heuristic}

As will be seen in the computational results presented later below solving (to proven optimality) the order and rack allocation problem for just a single picker using standard optimisation software, such as Cplex~\citep{cplex128} which we used, is computationally relatively quick. Based on this observation we developed a heuristic which solves our order and rack allocation formulation optimally for each single picker in a  sequential fashion, until all pickers have been considered. 

As we assumed above that the pickers are indexed in decreasing $C_p$ order (so $C_p \geq C_{p+1}~p=1,\ldots,(P-1)$) we have a natural ordering for the pickers. Our single picker based 
(SPB) heuristic is therefore:
\vspace{-\topsep}
\begin{enumerate}[label=(\alph*), noitemsep]
\item Set $p=1$
\item Solve the order and rack allocation problem for just this single picker $p$ to proven optimality
\item Delete the orders and racks allocated to the picker from the problem 
\item Set $p=p+1$ and if $p\leq P$ go to (b)
\item Solve the order and rack allocation problem for all $P$ pickers simultaneously, but restricting attention to just those racks chosen when the pickers were considered individually
\end{enumerate}
\vspace{-\topsep}
This heuristic considers each picker in turn until all of the pickers have been considered. So after completing steps (a)--(d) above we have a heuristic solution for the problem. However it is possible to (potentially) improve this heuristic solution by performing step (e) above.
This final improvement step makes use of our order and rack allocation formulation to solve the problem, but restricting attention to the subset of racks chosen when the pickers were considered individually. This will mean that, computationally, our formulation need only consider far fewer racks.

In terms of the computational effort required SPB needs to perform $P+1$ optimisations, i.e.~solve the order and rack allocation formulation $P+1$ times.

Note here that if we have just a single picker (so $P=1$) then this heuristic is redundant as it  will give exactly the same solution  as produced by solving our order and rack allocation formulation to proven optimality (step (b) above with $p=1$).

\subsection{Partial integer optimisation (PIO) heuristic}
We also developed a heuristic for the order and rack allocation problem based upon partial integer optimisation, which is a relatively new approach in the literature.

In general terms in a partial integer optimisation (PIO) heuristic we start with a mathematical formulation of the optimisation problem under consideration. From this formulation a subset of the integer variables are declared as integer, with all of the remaining integer variables being declared as continuous. The resulting mixed-integer problem is then solved. The integer variables are then fixed at the values that they have in this solution and the process repeats with a new subset of integer variables being declared as integer. 
A related approach to 
partial integer optimisation 
is kernel search~\citep{angelelli2012, guastaroba2017}.

In brief the idea behind kernel search (for a pure integer program) is to find a set of variables which collectively contribute to a good feasible solution. In kernel search a subset of the variables is first identified, the kernel, with the remaining variables being subdivided into buckets. These buckets are sequenced for order of consideration. An iterative process is then applied by solving (either optimally or heuristically) the integer program  which just involves   the kernel and one bucket. The kernel is then adjusted (e.g.~by adding new variables from the considered bucket that appear promising) and the process repeats with the next bucket.

By contrast a POI heuristic always maintains consideration of all variables (so no subdivision into kernel and buckets). 
Instead, as the phrase \emph{partial integer optimisation} implies some  variables are treated as integer variables, other (integer) variables are relaxed to become continuous variables.
As best as we are aware partial integer optimisation  approaches, which essentially lead directly from a  mathematical  formulation to a heuristic, without the necessity of designing a 
problem-specific metaheuristic, have not been significantly explored in the literature. 

Our partial integer optimisation heuristic for the order and rack allocation problem considers the variables associated with $\tau$ pickers at each iteration as integer variables and decides the values for these integer variables. Other variables are either fixed to integer  values or regarded as continuous variables. Formally this heuristic is:
\vspace{-\topsep}
\begin{enumerate}[label=(\alph*), noitemsep]
\item Set $T \leftarrow \tau$
\item Amend the order and rack allocation formulation,
 Equations~(\ref{eq1})-(\ref{eq7}),
 so that although the order allocation variables ($x_{op}$) and the rack allocation variables ($y_{rp}$) associated with pickers $p=1,\ldots,T$ remain binary variables, the order allocation and rack allocation variables associated with pickers $p=T+1,\ldots,P$ become continuous variables lying between zero and one. The rack usage variables ($u_r$) become continuous variables lying between zero and one.
\item Solve the amended order and rack allocation formulation. In this  solution we will have a (integer variable) feasible  allocation of orders and racks to the first $T$ pickers. Let the solution values for the order and rack allocation variables 
($x_{op}$ and $y_{rp}$)   be $X_{op}$ and $Y_{rp}$.
\item Add the constraints $x_{op} = X_{op}~o=1,\ldots,O~p=1,\ldots,T$ and $y_{rp} = Y_{rp}~r=1,\ldots,R~p=1,\ldots,T$ to the amended formulation.
\item If $T \neq P$ set $T \leftarrow \mbox{min}[T+\tau,P]$  and  go to (c).
\item Solve the order and rack allocation problem for all pickers simultaneously, but restricting attention to just those racks chosen
in step (c) above.
\end{enumerate}
\vspace{-\topsep}
In this heuristic $T$ is the number of pickers for which the order and rack allocation variables will be integer. Given such a solution for $T$ pickers we (at step (d) above) constrain all future solutions to retain these integer values, add a further $\tau$ pickers to $T$ and repeat until all pickers have been considered. Note here that in this heuristic
there is no necessity to declare the rack usage variables $u_{r}$ as integer since from Equation~(\ref{eq3}) they will automatically be integer if the rack allocation variables $y_{rp}$ are integer. 
So after completing steps (a)--(e) above we have a heuristic solution for the problem. Step (f) is the same final improvement step as in the last step of our SPB heuristic above.

The essential difference between this PIO heuristic and our SPB heuristic 
is that in the SPB heuristic we take no account of pickers whose order/rack allocation is yet to be decided  (other  than the single picker under consideration), whereas in our PIO heuristic these other pickers are considered, albeit with their associated variables
being regarded as  continuous variables.

In terms of the computational effort required PIO needs to perform $\lceil P/\tau \rceil+ 1$ optimisations, i.e.~solve the order and rack allocation formulation $\lceil P/\tau \rceil+ 1$  times.

As for our SPB heuristic above this PIO heuristic is redundant  if we have just $\tau$ pickers (so $P=\tau$)  as it  will give exactly the same solution  as produced by solving our order and rack allocation formulation to proven optimality.

\subsection{Using Cplex alone}

The SPB and PIO heuristics given above require the use of an optimisation solver, in the results presented below we used Cplex 12.8~\citep{cplex128}.

As the reader may be aware modern optimisation solvers (such as Cplex) contain inbuilt heuristics to find a heuristic solution to any problem under consideration. These heuristics are applied not only at the root node of the tree, but also throughout the tree search as a search for the optimal solution is made.
The question therefore arises as to why not simply apply Cplex and terminate the search after (for example) a predefined time limit. We would make the following points:
\begin{compactitem}
\item Cplex (by its very nature) processes in a mechanistic manner the mathematical formulation it is given for the problem being solved. 
\item By contrast scientific/creative insight is necessary to  design the varying algorithmic components of 
any  heuristic for the problem being solved.
\item Given the power of modern optimisation software packages it seemed worthwhile in our view \textbf{\emph{to attempt to utilise that power as an algorithmic component}} in any scientifically/creatively designed heuristic.
\end{compactitem}
\noindent
Of course the key discriminator between just applying Cplex alone and applying any other heuristic is the results obtained. In particular the balance between the computation time required and the quality of results obtained. We believe that the results presented below indicate that in terms of this key discriminator our SPB and PIO heuristics are worthwhile as compared to simply applying Cplex alone.

\section{Formulation - picker rack sequencing}
\label{sec:formjebrack}

In this section we first outline the picker rack sequencing problem. This is the problem of sequencing the racks allocated to each picker  
given a (known) allocation of orders and racks to each of $P$ pickers (such as decided by the formulation we presented above).    Here, after order/rack allocation,  each picker is independent and so we focus on the problem of sequencing the racks allocated to a single picker.

We   assume in this section:
\begin{compactitem}
\item  that there are sufficient items of required products on the racks allocated to the picker to supply all of the orders allocated to the picker (since if not then no feasible rack sequence supplying all orders can exist)
\item that a picker needs to assemble a complete order in one of their bins before the bin containing the completed order is put onto the conveyor and carried away for further processing (c.f.~Figure~\ref{fig3}). 
\end{compactitem}
This second assumption seems appropriate in terms of the application we are addressing where (typically) a small number of product items needed to be picked for a customer order. Assembling all the items needed for each customer order in a single bin as they are picked seems more effective than picking items individually and passing these items on for processing and assembly into a customer order at a later stage.

In this section we first  formulate the picker rack sequencing problem assuming that
each mobile rack only visits the picker once (so no rack revisits to the picker are allowed). 
We then indicate how to extend our formulation to allow rack revisits.  We prove that, subject to certain conditions being satisfied, in   generating a feasible rack sequence for a single picker we can  initially neglect  all orders which comprise a single unit of one product since a feasible rack sequence for the remaining orders
can be easily modified to incorporate the neglected orders. 
Finally we discuss how we might integrate the order and rack allocation formulation given previously above with the rack sequencing formulation given in this section.

\subsection{Problem outline}
Any feasible solution to the formulation presented above
(Equations~(\ref{eq1})-(\ref{eq7}))
will give an allocation of orders and racks to pickers. For any picker $p$ the constraints of the problem will ensure that there are sufficient items of required products on the racks allocated to the picker to supply all of the orders allocated to the picker.

Referring to Figure~\ref{fig3} suppose that we have $B$ positions where bins can be placed which the picker can  fill with product items for  orders.
Any order for which the picker has picked at least one product item are~\textbf{\emph{open}} in the sense that there must be a dedicated bin available for that order. Orders where all the required products have been picked are~\textbf{\emph{closed}} in the sense that the bin containing all the required products has been moved to the conveyor for later processing, and a new empty bin positioned for a future order to be opened. 

Now if for picker $p$ the number of orders $C_p$ considered in  deciding the allocation of orders and racks to pickers is such that $C_p \leq B$ then irrespective as to how the racks are sequenced so as to be presented to the picker it will be possible to deal with all of the orders without having to consider which orders are open and which closed. This is because the $B$ bin positions will be sufficient to fill all of the orders allocated to the picker.
However if in solving our formulation 
(Equations~(\ref{eq1})-(\ref{eq7}))
we considered many orders, so that we had $C_p > B$, then the situation is more complex.
The essential reason why we might have $C_p > B$ is that we are adopting a longer time horizon than simply the time horizon associated with the filling of $B$ bins and wish to allocate orders and racks to pickers accordingly.

If $C_p > B$ then the decision problem we face is how to sequence the racks for presentation to the picker, as well as decide the number of units of each required product to take from each rack for orders which are currently open. Implicit in this decision problem is a requirement to decide which orders are open, and which unopened or closed, as each rack is presented to the picker.

To illustrate the importance of the rack sequence if $C_p > B$ suppose that we have one order that requires two products, such that
out of the racks allocated to the picker
one  of these two products is only stored in one rack, the other only stored in a different rack. If these two racks are next to each other in the rack sequence adopted for presentation to the picker then the order can be opened as the first rack is presented to the picker, closed as the next rack is presented to the picker. So the open order just involves  one bin position between the two successive rack presentations. 

However if these two racks were to be sequenced such that one rack is the first rack presented  to the picker, and the other rack is the last rack presented to the picker, 
then the order would be opened with the first rack and remain open until the last rack, thereby effectively occupying a single bin  position over the entire rack sequence. Clearly with just 
$B$ bin positions having one position occupied for the entire rack sequence by just one open order may not be the best use of limited bin position capacity.

In our view the rack sequencing problem at each picker is essentially a feasibility problem. By this we mean that our assumption here is that \textbf{\emph{a primary optimisation has already been performed in terms of choosing which orders and racks are allocated to each picker}}, as in our formulation above. Once orders and racks have been allocated to a picker as a result of a primary optimisation then the scope for further optimisation in terms of rack sequencing at each picker is much less. For example, all allocated racks will be presented to the picker and the picking of the items required for the orders allocated will  take the same time irrespective as to the rack sequence. As such finding a feasible rack sequence such that all orders can be picked using $B$ bin positions is, in our view, the principal concern. 

To assist in finding a feasible rack sequence  we adopt the approach that if there is some order allocated to a rack which is opened and closed by that rack (e.g.~an order just requiring a single unit of one product, or more generally an order for multiple products which is fully satisfied by that rack) then we ensure (via imposing appropriate constraints) that one of the $B$ bins is available to process that order. Such an operational strategy seems appropriate for environments, such as Amazon~\citep{weidinger18c}, where the vast number of orders are for only one or two items.

The optimisation adopted in our formulation above was designed to minimise the number of racks needed, and so encouraged the  picking of multiple products for the same order from the same rack. This therefore means that we can reasonably expect a number of orders which can be fully satisfied by making use of just a single rack. 

Note here that it is only necessary to choose one of the $B$ bins for opening and closing orders from the same rack. This is because once a rack has been first presented to the picker then they can open and close one order from the rack using the chosen bin, move the bin with the closed order to the conveyor for further processing, and then place a new empty bin in the now  vacated bin position for the next order than can be opened and closed from the  rack. This leaves $B-1$ bins available for opening and closing orders that need more than one rack to be satisfied.

\subsection{Formulation}

We commented above that given an allocation of orders and racks to pickers  then in deciding the rack sequence (as well as associated other decisions) we have $P$ independent problems, one for each of the $P$ pickers. In this section we present our rack sequencing formulation. For ease of presentation we shall initially assume here that each mobile rack only visits the picker once (so no rack revisits to the picker are allowed), but indicate later below how this assumption can be removed.

In contrast to earlier work presented in the literature related to mobile rack sequencing \citep{boysen17a} we explicitly consider in our formulation rack inventory positions and decide the number of units of each product allocated to each order from each 
rack.~\cite{boysen17a} did  not explicitly  consider the number of units of each product stored in a  rack, rather they assumed that if a rack contained a product then it had sufficient of that product to satisfy all orders requiring that product. In terms of our notation this means that either $s_{ir}=0$ or 
$s_{ir} = \sum_{o=1}^O q_{io}$.

Let $O^*$ be the set of orders allocated to the picker $p$ under consideration (so $|O^*| = C_p$) and let $R^*$ be the set of racks allocated to the picker, with $K=|R^*|$ being the number of racks allocated to the picker. Let $N^*$ be the set of products associated with the orders ($O^*$) allocated to the picker, so that $N^* = [i~|~ \sum_{o \in O^*} q_{io} \geq 1~i=1,\ldots,N]$.
It is important to note here that we will typically have $K \ll R$, since we are no longer considering all possible racks in the entire storage area, but only those racks that have been allocated to the picker so as  to satisfy the orders assigned to the picker. 
We will assume that $\sum_{o \in O^*} q_{io} \leq \sum_{r \in R^*} s_{ir} ~\forall i \in N^*$ since otherwise it will simply not be possible to supply all orders using the racks allocated.

Let the set of orders $o \in O^*$ for which the order just comprises a single unit of one product be denoted by $O_1^*$. This set can be defined using $O_1^* = [o~|~ \sum_{i \in N^*} q_{io} = 1~ o \in O^*]$, since 
if we have some order $o \in O^*$ for which $\sum_{i \in N^*} q_{io} = 1$ then this can only occur if the order is just for a single unit of one product. Define $\tau(o)$ as that unique product associated with order $o \in O_1^*$.

Our formulation involves the following decision variables:
\begin{itemize}
\item $z_{kr}=1$ if rack $r \in R^*$ is the $k$'th rack in the rack sequence presented to the picker, zero otherwise
\item $\alpha_{ok}=1$ if order $o \in O^*$ is open at some point during the time  when the $k$'th rack in the rack sequence  is presented to the picker, zero otherwise
\item $\beta_{ok}=1$ if order $o \in O^*$ is closed at some point during the time   when the $k$'th rack in the rack sequence is presented to the picker, zero otherwise
\item $\gamma_{iok}$, integer $\geq0$, is the number of units of product $i \in N^*$ associated with order $o\in O^*$ that are taken from the  $k$'th rack in the rack sequence presented to the picker (and placed into the bin associated with that order, which must therefore  be an open order)
\end{itemize}

In the  formulation presented below one bin out of the $B$ available 
is used for orders opened and closed by the same rack. One way to accomplish this would be for one of the $B$ bins to be permanently and solely dedicated to such orders for each and every rack presentation. However this is unnecessary since for some rack presentations we may be able process such orders via 
use of bins that are already empty, or bins that will become empty. Moreover for other rack presentations if there are no orders opened and closed by the rack there would no need for such a bin.
In more detail it is not necessary as rack $k$ is presented to have a single bin available for orders opened and closed by the rack if one or more of the following three conditions applies:
\begin{compactitem}
\item condition 1: the number of open orders passed on from the previous rack (rack $(k-1)$) is less than $B$ (so there is a free bin automatically available as rack $k$ is presented)
\item condition 2: there is some order in the open orders passed from  rack $(k-1)$ to rack $k$ that will be closed at rack $k$ (so the bin associated with this closed order can then be used for orders opened and closed at rack $k$) 
\item condition 3: there is no order that is opened and closed at rack $k$ 
\end{compactitem}
Although (for simplicity of expression) we refer to rack $k$ in these conditions  note that by this we mean the $k$'th rack in the rack sequence presented to the picker.

Hence introduce variables $U_k,~k=1,\ldots,K$ where $U_k=1$ if when rack $k$ is presented we can make use of all $B$ bins (so one or more of these three conditions applies and we have no need for a bin reserved for orders opened and closed by the same rack), zero otherwise. Let $W_{ik}=1$ if when rack $k$ is presented  condition $i~(i=1,2,3)$  applies, zero otherwise. We also need additional variables:
\begin{compactitem}
 \item $w_{ok},~o \in O^* \setminus O_1^*,~k=2,\ldots,(K-1)$ where $w_{ok}=1$ if order $o$ is an open order passed from  rack $(k-1)$ to rack $k$ that is closed at rack $k$, zero otherwise. \item $v_{ok},~o \in O^*,~k=2,\ldots,(K-1)$ where $v_{ok}=1$ if order $o$ is an order that is opened and closed at rack $k$, zero otherwise.
\end{compactitem}
\noindent The constraints in our formulation of the problem of picker rack sequencing are as follows:
\begin{optprog}
& \sum_{k=1}^K z_{kr} & = & 1 & \forall r \in R^* \label{req1} \\

& \sum_{r \in R^*} z_{kr} & = & 1 & \forall k \in \{1,\ldots,K\}
\label{req2} \\

& \sum_{o \in O^*} \gamma_{iok} & \leq & \sum_{r \in R^*} s_{ir}z_{kr} & \forall i \in N^*; k \in \{1,\ldots,K\}\label{req3} \\

& \sum_{k=1}^K \gamma_{iok} & = & q_{io} &  \forall i \in N^*; o \in O^* \label{req4} \\

& \gamma_{iok} & \leq & q_{io} \alpha_{ok} & \forall i \in N^*; o \in O^*; k \in \{1,\ldots,K\}\label{req5} \\

& \sum_{o \in O^*} (\alpha_{ok}-\beta_{ok}) & \leq & B-1 + U_k  & \forall k \in \{1,\ldots,K\}
 \label{fin1} \\

& \alpha_{ok} & = & \beta_{ok} & \forall o \in O_1^*; k \in \{1,\ldots,K\}
\label{req6a} \\

& \gamma_{\tau(o),o,k} & = & \alpha_{ok} &  \forall o \in O_1^*;  k \in \{1,\ldots,K\}
\label{req6b} \\


& \sum_{k=1}^K \beta_{ok} & = & 1  &  \forall o \in O^*  \label{req7} \\

& \beta_{o1} & \geq & 1 -  \sum_{i \in N^*} \sum_{k=2}^K \gamma_{iok} &  \forall o \in O^* \setminus O_1^*  \label{req7a} 
\end{optprog}
%
%
\newpage
\begin{optprog}

& \beta_{ok} & \geq & 1 - \sum_{i \in N^*} \sum_{n=k+1}^K \gamma_{ion} -\sum_{n=1}^{k-1}\beta_{on} &  \forall o \in O^* \setminus O_1^*; 
k \in \{2,\ldots,K-1\}
\label{req7b} \\

& \alpha_{ok} &\leq & \sum_{i \in N^*} \sum_{n=1}^k \gamma_{ion}  &  
\forall o \in O^* \setminus O_1^*; k \in \{1,\ldots,K\}
\label{req8a} \\

& \alpha_{ok} & \leq & 1-\sum_{n=1}^{k-1} \beta_{on}  &  \forall o \in O^* \setminus O_1^*;
k \in \{2,\ldots,K\}
\label{req9} \\

& \alpha_{o,k+1} &\geq & \alpha_{ok} - \beta_{ok}  &  
\forall o \in O^* \setminus O_1^*;
k \in \{1,\ldots,K-1\}
\label{req8} \\

& U_k & \geq & W_{ik}  & \forall k \in \{2,\ldots,K-1\}; i \in \{1,2,3\}
 \label{fin1a} \\

& U_k & \leq & \sum_{i=1}^3 W_{ik} & \forall k \in \{2,\ldots,K-1\} \label{fin1b} \\

& W_{1k} & \leq & B - \sum_{o \in O^*} (\alpha_{o,k-1}-\beta_{o,k-1}) 
& \forall k \in \{2,\ldots,K-1\}
\label{fin1c} \\

& w_{ok}  & \geq & \alpha_{o,k-1} + \beta_{ok} - 1  &  
\forall o \in O^* \setminus O_1^*;
k \in \{2,\ldots,K-1\}
\label{fin3} \\

& w_{ok}  & \leq & \alpha_{o,k-1} &  
\forall o \in O^* \setminus O_1^*;
k \in \{2,\ldots,K-1\}
\label{fin4} \\

& w_{ok}  & \leq & \beta_{ok}  &  
\forall o \in O^* \setminus O_1^*;
k \in \{2,\ldots,K-1\}
\label{fin4a} \\

& W_{2k}  & \geq & w_{ok} &  
\forall o \in O^* \setminus O_1^*;
k \in \{2,\ldots,K-1\}
\label{fin5} \\

& W_{2k}  & \leq & \sum_{o \in O^* \setminus O_1^*} w_{ok}  &  
k \in \{2,\ldots,K-1\}
\label{fin6}\\ 

& v_{ok}  & \geq &   \beta_{ok} - \alpha_{o,k-1}  &  
\forall o \in O^*;
k \in \{2,\ldots,K-1\}
\label{fin7} \\

& v_{ok}  & \leq & 1 - \alpha_{o,k-1}   &  
\forall o \in O^*;
k \in \{2,\ldots,K-1\}
\label{fin8} \\

& v_{ok}  & \leq & \beta_{ok}  &  
\forall o \in O^*;
k \in \{2,\ldots,K-1\}
\label{fin8a} \\

& W_{3k}  & \geq & 1- \sum_{o \in O^*} v_{ok} &  
k \in \{2,\ldots,K-1\} 
\label{fin9} \\

& W_{3k}  & \leq & 1- v_{ok} &  
\forall o \in O^*; k \in \{2,\ldots,K-1\} 
\label{fin10} \\

& U_k & = & 1 & k \in \{1,K\} \label{fin0} \\

& z_{kr} & \in & \{0,1\} & \forall k \in \{1,\ldots,K\};
r \in R^* \label{req10} \\
& \alpha_{ok},\beta_{ok} & \in & \{0,1\} & \forall o \in O^*; k \in \{1,\ldots,K\}
\label{req11} \\
& \gamma_{iok} & \geq 0, & \mbox{integer} & \forall i \in N^*; o \in O^*; k \in \{1,\ldots,K\} 
\label{req12} \\

& U_k & \in & \{0,1\} & \forall k \in \{1,\ldots,K\} \label{jebi1} \\

& W_{ik} & \in & \{0,1\} & \forall i \in \{1,2,3\};  k \in \{2,\ldots,(K-1)\} \label{jebi2} \\

 & w_{ok} &  \in & \{0,1\} & \forall o \in O^* \setminus O_1^*;
k \in \{2,\ldots,(K-1)\} \label{jebi3} \\

& v_{ok} & \in & \{0,1\} &
\forall o \in O^*; k \in \{2,\ldots,(K-1)\} \label{jebi4} 
\end{optprog}


Given that we have $K$ racks then these racks need to be sequenced  in some order to be presented to the picker, and this is addressed using 
Equation~(\ref{req1}) and  Equation~(\ref{req2}).
 Equation~(\ref{req1}) ensures that each rack is only assigned to one position in the sequence. Equation~(\ref{req2}) ensures that one rack is assigned to each of the positions $k$ in the sequence for presentation to the picker, $k=1,\ldots,K$. 

Equations~(\ref{req3})-(\ref{req5}) deal with the supply of product (from the racks presented) for the orders considered.
Equation~(\ref{req3}) ensures that the total number of units of product $i$ allocated to orders from the $k$'th rack presented does not exceed the supply of  product on the rack. Equation~(\ref{req4}) ensures that each order receives the required number of units of each product. 
Equation~(\ref{req5}) ensures that no product can be taken from the $k$'th presented rack for an order that  is not open. Note here that we do not impose the condition that if an order is open (so $\alpha_{ok}=1$) there must be some product supplied from the rack (so $\gamma_{iok} \geq 1$ for some $i \in N^*$). It is entirely possible, for orders that are open over a number of successive rack presentations, that one particular rack does not supply any product to an open order (that order remaining open as it needs product from a subsequent rack in the sequence).

Equation~(\ref{fin1}) ensures that the number of open orders associated with the $k$'th presented rack never exceeds the number of available  bin positions, which is  $B{-}1$ if 
$U_k$=0, $B$ if 
$U_k$=1. In other words we have a single bin reserved for orders opened and closed by the rack if necessary (i.e.~if $U_k$=0).
Note here that any order $o$ which is only opened as the 
$k$'th rack is presented, but also closed with that rack, will have $\alpha_{ok}=\beta_{ok}=1$ and hence not contribute to the left-hand side of  Equation~(\ref{fin1}). Orders with $\alpha_{ok}=\beta_{ok}=1$ will  include any order
$o \in O_1^*$, which by definition can be dealt with by the same rack. However (for example) an order $o \in O^* \setminus 
O_1^*$ requiring two or more different products, but where all these products are supplied by rack $k$, will also have $\alpha_{ok} = \beta_{ok} = 1$.

Equation~(\ref{req6a}) deals with orders $o \in O_1^*$ where we can guarantee that they are opened and closed with the same rack. In that equation if we have some order $o \in O_1^*$, so $\sum_{i \in N^*} q_{io} = 1$,
then the order is just for a single unit of one product. Such an order must be satisfied by just a single presented rack (by definition, since the order is just for a single unit of one product), and so that order can be opened and closed by the same rack using the single bin reserved for orders opened and closed by the same rack. This implies that $\alpha_{ok}  =  \beta_{ok}~ k=1,\ldots,K$ for that order. Note here that if $\sum_{i \in N^*} q_{io} \geq 2$ for some order $i \in O^*$ then even if there is just one product associated with this order the units required ($\geq 2$) may come from more than one rack and so we cannot guarantee use of the reserved bin.

Equation~(\ref{req6b}) deals with the same situation as Equation~(\ref{req6a}). In that equation product $\tau(o)$ is the unique product associated with an order $o \in O_1^*$ for which we only need  a single unit of one product.
For such orders we will have $\gamma_{\tau(o),o,k}=0$ for all racks  except for the single rack chosen to pick the single unit of that product, where we will have $\gamma_{\tau(o),o,k}=1$. From Equation~(\ref{req6a}) the single rack from which we supply the product will be the rack $k$ with $\alpha_{ok}=\beta_{ok}=1$. 
Hence setting $\gamma_{\tau(o),o,k}=\alpha_{ok}$ is valid.

Equation~(\ref{req7}) ensures that each order is closed once. 
Equations~(\ref{req7a})-(\ref{req8}) apply for orders $ o \in O^* \setminus O_1^*$ where we cannot guarantee that the order is opened and closed using the same rack.

Equations~(\ref{req7a})-(\ref{req7b}) deal with order closure
for orders $ o \in O^* \setminus O_1^*$.
Equation~(\ref{req7a}) ensures that an order is closed when the first rack in the sequence is presented if no product items for that order are picked in any future rack. Equation~(\ref{req7b}) ensures that an order is closed when the $k$'th rack in the sequence is presented if no product items for that order are picked in any future rack and the order has not already been closed.

Equations~(\ref{req8a})-(\ref{req8}) deal with whether orders are open or not
for orders $ o \in O^* \setminus O_1^*$. Equation~(\ref{req8a}) ensures that if no units of any product associated with an order have yet been picked then the order cannot be open.
Equation~(\ref{req9}) ensures that once an order has been closed it cannot be opened again. 
Equation~(\ref{req8}), in conjunction with Equation~(\ref{req9}),
 ensures that once an order has been opened (so $\alpha_{ok}=1$) then it stays open until it is closed. 

To illustrate  Equation~(\ref{req8}) 
suppose that we first open some order $o$ when the second rack is presented (so $\alpha_{o2}=1$). If that order is not closed by that rack then Equation~(\ref{req8}) will be $\alpha_{o3} \geq \alpha_{o2}$, which as  $\alpha_{o3}$ is a zero-one variable will mean that $\alpha_{o3}=1$, so the order is open as the third rack is presented. Suppose that the order is closed with this third rack, so $\beta_{o3}=1$. Then Equation~(\ref{req8}) will be $\alpha_{o4} \geq \alpha_{o3} - \beta_{o3}$, i.e.~$\alpha_{o4} \geq 0$. But Equation~(\ref{req9}) will be $  \alpha_{o4}  \leq  1-\sum_{n=1}^{3} \beta_{on}$, and as $\beta_{o3}=1$ this means that $  \alpha_{o4}  \leq  0$, so $  \alpha_{o4}=0$ as required as the order has been closed.

In order to explain Equations~(\ref{fin1a})-(\ref{fin0}) recall here that, as stated above,
it is not necessary as rack $k$ is presented to have a single bin available for orders opened and closed by the rack if one or more of the following three conditions applies:
\begin{compactitem}
\item condition 1: the number of open orders passed on from the previous rack (rack $(k-1)$) is less than $B$ (so there is a free bin automatically available as rack $k$ is presented)
\item condition 2: there is some order in the open orders passed from  rack $(k-1)$ to rack $k$ that will be closed at rack $k$ (so the bin associated with this closed order can then be used for orders opened and closed at rack $k$) 
\item condition 3: there is no order that is opened and closed at rack $k$ 
\end{compactitem}
\noindent Mathematically this first condition is $\sum_{o \in O^*} (\alpha_{o,k-1}-\beta_{o,k-1}) < B $ (equivalently $\sum_{o \in O^*} (\alpha_{o,k-1}-\beta_{o,k-1}) \leq B-1$) and this second condition is that there exists some order $ o \in O^* \setminus O_1^*$ such that $\alpha_{o,k-1} =  \alpha_{ok} = \beta_{ok} =1$. This third condition is that there is no order $ o \in O^*$ such that $\alpha_{o,k-1} = 0$ and  $\alpha_{ok} = \beta_{ok} =1$.

However note here that $\beta_{ok} =1$ automatically implies that $\alpha_{ok} =1$. For orders  $ o \in O_1^*$ this follows directly from Equation~(\ref{req6a}). For orders $ o \in O^* \setminus O_1^*$ Equations~(\ref{req7a}),(\ref{req7b}) imply that $\beta_{ok} =1$ for the last rack $k$ for which product is picked for order $o$. Since from Equation~(\ref{req5}) no product for order $o$ can be picked at rack $k$ unless  $\alpha_{ok} =1$ this automatically implies that $\alpha_{ok} =1$. Hence the second condition can be simplified to: there exists some order $ o \in O^* \setminus O_1^*$ such that $\alpha_{o,k-1} = \beta_{ok} =1$.  The third condition can be simplified to: there is no order $ o \in O^*$ such that $\alpha_{o,k-1} = 0$ and  $ \beta_{ok} =1$.

Equations~(\ref{fin1a}),(\ref{fin1b}) ensure that $U_k$ is  one if one or more of the three conditions is satisfied (zero otherwise).
Equation~(\ref{fin1c}) ensures that $W_{1k}=0$ if there are $B$ open orders passed to rack $k$. There is no requirement to have a constraint forcing $W_{1k}$ to one because if we have $\sum_{o \in O^*} (\alpha_{o,k-1}-\beta_{o,k-1}) < B$ we have that Equation~(\ref{fin1c}) becomes inactive and since $W_{1k}$ is a zero-one variable it will be assigned the value one if it is worthwhile so to do.
 
Equations~(\ref{fin3})-(\ref{fin4a}) ensure that $w_{ok}=1$ if and only if  $\alpha_{o,k-1} = \beta_{ok}=1$, so order $o$ is open at the 
$(k-1)$'th presented rack, passed to the $k$'th rack and closed with that rack.
Equations~(\ref{fin5}),(\ref{fin6}) ensure that $W_{2k}=1$ if there is at least one order $o$ with $w_{ok}=1$, zero otherwise. 

Equations~(\ref{fin7})-(\ref{fin8a}) ensure that $v_{ok}=1$ if and only if  $\alpha_{o,k-1} = 0$ and $\beta_{ok}=1$, so order $o$ is not open at the 
$(k-1)$'th presented rack, being opened and closed at the  $k$'th rack.
Equations~(\ref{fin9}),(\ref{fin10}) ensure that $W_{3k}=1$ if there is no order $o$ with $v_{ok}=1$, zero otherwise. 

In Equation~(\ref{fin0}) we automatically know that we can use $B$ bins for the first and last rack as for the first rack all bins are empty as it is  presented, and for the last rack all remaining orders will be closed by that rack.
Equations~(\ref{req10})-(\ref{jebi4}) are the integrality conditions. 

In summary here then our picker rack sequencing formulation (henceforth referred to as {\it\textbf{PRSF}}) is Equations~(\ref{req1})-(\ref{jebi4}).

With regard to a number of minor issues relating to PRSF note  that:
\begin{compactitem}
\item Equations~(\ref{req7a}),(\ref{req7b}) could be considered as redundant and removed from the  formulation. However if we were to do so we may end up with a solution where an order is left open (and so continues to occupy a bin) when all of the products for that order have been picked. For example this could occur if the bin so occupied is not needed for another order at any point over future rack presentations. Since this would seem somewhat illogical, after all in practice once an order has been fulfilled it would seem natural to close it immediately and pass the associated picked products on for further processing, we have retained Equations~(\ref{req7a}),(\ref{req7b}) in our formulation. 
A further reason for retaining these equations is that they are associated with  enforcing the logic underlying the three conditions discussed above.

\item Similarly Equation~(\ref{req8a}) is also technically redundant and could be removed from the  formulation. But as its removal means an order could be opened and occupy a bin over a number of successive rack presentations before even one item associated with the order is picked (which again seems somewhat illogical) we have retained Equation~(\ref{req8a}) in our formulation.   

\item 
Equation~(\ref{req12}) could equally well be expressed using two equations, namely $ \gamma_{iok}  \geq 0,  ~\forall i \in N^*~\forall o \in O_1^*~ k=1,\ldots,K $ and $ \gamma_{iok}  \geq 0,~  \mbox{integer}  ~\forall i \in N^*~\forall o \in O^* \setminus 
O_1^*~~ k=1,\ldots,K $. This is because $\gamma_{iok}$ will be naturally integer for $\forall o \in O_1^*$ as a result of Equations~(\ref{req6b}),(\ref{req11}). However we have 
left Equation~(\ref{req12}) as stated above for convenience.
Note here that the $\gamma_{iok}$ integer variables cannot be relaxed to be fractional variables $\forall o \in O^* \setminus O_1^*$ for the same reasons as discussed previously above when we considered the $\pi_{irp}$ variables associated with multiple picker allocation for racks. 
However, as in the discussion above in relation to the $\pi_{irp}$ variables, if it is judged desirable (e.g.~for computational reasons)  to relax the integrality requirement on the $\gamma_{iok}$  variables so that they become non-negative continuous variables then some appropriate rounding procedure can be used to generate a feasible integer solution from any fractional  $\gamma_{iok}$  values.

\item Computationally $w_{ok}$ and $v_{ok}$ can be treated as continuous variables lying between zero and one since the requirement for $\alpha_{ok}$ and $\beta_{ok}$ to be zero-one will make these variables naturally integer. $W_{ik}$ and $U_k$ can also be treated as continuous variables lying between zero and one.
\end{compactitem}

\noindent As commented previously above any (integer) feasible solution to PRSF will constitute a feasible rack sequence (together with picking details $\gamma_{iok}$) such that all orders can be dealt with. Note here however that we cannot guarantee that a feasible solution can be found. This is essentially because the  number of bin positions (here $B$) at a picker was never taken into account in our formulation
(Equations~(\ref{eq1})-(\ref{eq7}))
which gave an allocation of orders and racks to pickers. 

In order to find a feasible solution to PRSF we could propose using a metaheuristic algorithm. However modern optimisation packages, such as Cplex~\citep{cplex128} which we used, contain heuristic options to search for a feasible solution, as well as having the ability to conduct a full tree search for such a solution. For this reason we will use Cplex to find a solution.

PRSF involves approximately $K(K +  4|O^*| - |O^*_1|)$ zero-one variables and $K|N^*||O^*|$ integer variables. With respect to the number of constraints involved then we have a significant contribution from 
 Equation~(\ref{req3}):   $K|N^*|$ constraints;
Equation~(\ref{req4}): $|N^*||O^*|$ constraints
and Equation~(\ref{req5}):  $K|N^*||O^*|$ constraints.

However it is important to note here that typically we might expect many of the orders $o \in O^*$ not to involve all products $i \in N^*$ and any such order will have $q_{io}=0$ meaning that from 
Equation~(\ref{req5}) $K$ $\gamma_{iok}$ variables must automatically be zero and these variables  can be removed from consideration, reducing the size of the problem, both in terms of variables and constraints, that we need to solve.
Note also here  that a reduction in the size of the formulation can be achieved via algebraic substitution  by making use of 
Equations~(\ref{req6a}) and~(\ref{req6b}). The solver we used, 
Cplex~\citep{cplex128}, contains procedures to automatically perform such substitutions, as well as identify any variables or constraints which are redundant.

\subsection{Complexity}


With regard to the complexity of the picker rack sequencing problem~\cite{boysen17a},  by reference to complexity results for the interval scheduling 
problem~(\cite{kiel92, nakajima1982, spieksma92}), showed that the problem of deciding whether there is a feasible order sequence, given a specific  rack sequence, is strongly 
NP-complete even if there is just a single bin. 
In their work they  assume that if a rack contains a product then it has sufficient of that product to satisfy all orders requiring that product and too assume that rack revisits are allowed.

\cite{fusser19}, by reference to complexity results for the Hamiltonian path problem, showed that when
there is a single picker (and the objective is to minimise the number of racks used) the problem is strongly NP-hard  even if there is just a single bin. They also present a proof, making reference to complexity results for the offline sorting buffers problem~(\cite{asahiro2012,chan2012}), that the problem is strongly NP-hard even if the order sequence is given and each order involves just a single product.
In their work they assume that each rack holds just a single product and too holds sufficient of that product to supply all customer orders. They also assume that rack revisits are allowed.

In terms of the picker rack sequencing problem  considered above,  PRSF,  where each rack can only be used once (so no revisits); where we are seeking a feasible solution (so we have a feasibility, not an optimisation, problem); and where we have $B$ bins, then the complexity result of~\cite{boysen17a} applies.

\begin{theorem}
For the picker rack sequencing problem, PRSF,  considered in this paper the problem of deciding whether there is a feasible order sequence, given a specific  rack sequence, is strongly 
NP-hard even if $B=1$.
\label{th2}
\end{theorem}

\begin{proof}
To prove this theorem consider a set of racks each of which holds sufficient of each of two or more distinct products to satisfy all orders requiring those products.  Consider a set of orders, where each order is for one unit of two or more distinct products, but the set of products required by any order is not a subset of the products stored in any of the racks. In other words it is impossible for any single rack to satisfy an order, so any consideration of the single bin used for processing orders which are opened and closed by the same rack is irrelevant (i.e.~we automatically know that $U_k=1~\forall k \in \{1,\ldots,K\}$). Hence if
$B=1$ then we have just a single bin available to hold products for orders that (by construction) require two or more racks and the problem reduces to that considered 
in~\cite{boysen17a}. 

So the complexity result given previously  in the literature 
by~\cite{boysen17a} applies, and the problem of deciding whether there is a feasible order sequence, given a specific  rack sequence, is strongly 
NP-hard even if $B=1$.
\end{proof}

\subsection{Order and rack processing}

The picker rack sequencing formulation PRSF given above explicitly sequences the racks in terms of their order of presentation to the picker, but does not explicitly sequence the orders that the picker processes. Rather, as the variables $\alpha_{ok}$ and $\beta_{ok}$ imply, we decide the racks associated with the opening and closing of each order. However given a solution to our rack sequencing formulation than the process for picking product from the $k$'th rack as it is presented to the picker is easily stated. Note that we will have from the numeric solution for our rack sequencing problem values for $\alpha_{ok}$,  $\beta_{ok}$ and  $\gamma_{iok}$ which enable us to  identify for the $k$'th rack in the presented sequence:
\begin{itemize}

\item The set $\Delta_k$ of orders which are opened and closed using that rack (any order $o \in O^*$ with $\alpha_{ok} = \beta_{ok} =1$ and $\gamma_{iok} = q_{io}~\forall i \in N^*$)

\item The set $\Theta_k$ of orders which were open previously,
 but are closed as that rack is presented (any order $o \in O^*$, $o \notin \Delta_k $ for which $\alpha_{o,k-1}=\alpha_{ok}=\beta_{ok}=1$)

\item The set $\Phi_k$ of orders which are first opened as that rack is presented (any order $o \in O^*$, $o \notin \Delta_k \cup \Theta_k$ for which either $k=1$ and $\alpha_{o1}=1$ or  $k \geq 2$ and $\alpha_{o,k-1}=0$ and $\alpha_{ok}=1$).

\item The set $\Omega_k$ of orders which were opened previously and which will remain open (any order $o \in O^*$, $o  \notin \Delta_k \cup \Theta_k \cup \Phi_k $ for which $k \geq 2$ and $\alpha_{o,k-1}=\alpha_{ok}=1$)
\end{itemize}

The procedure for processing of the  $k$'th rack presented to the picker  is:
\begin{enumerate}
\item First process all orders $o \in \Theta_k$, thereby freeing up bin positions 
\item Next use the single bin reserved for orders opened and closed by the same rack to open and close all orders $o \in \Delta_k$ 
\item Next process all orders $o \in \Phi_k$, making use of free bin positions
\item Finally process all orders $o \in \Omega_k$
\end{enumerate}
In each of the above steps the number of units of product $i$ to add into the bin associated with an order $o$ is given by  $\gamma_{iok}$ and the orders in each of the sets $\Delta_k$, $\Theta_k$, $\Phi_k$ and $\Omega_k$ can be processed in any sequence.

We would stress here that one of the benefits of the formulation (PRSF) presented in this paper is that there is no need to include in the formulation variables/constraints associated with the sequencing of orders. Rather order sequencing can be carried out in the manner discussed above.

\subsection{Rack revisits}

In PRSF presented above we assumed that each mobile rack only visits the picker once. Given this assumption then there is no guarantee that a feasible solution (feasible rack sequence) exists for a particular value of $B$.
However it is trivial to see that if we allow racks to visit the picker more than once, so rack revisits are allowed, then a feasible solution must always exists provided  Equation~(\ref{eq4}) is satisfied.
First note that if Equation~(\ref{eq4}) is satisfied the racks  allocated to a picker contain sufficient of each product to supply all of the orders allocated to the picker. Consider the racks $R^*$ as constituting a pool of racks from which we can choose any rack. Then:
\begin{compactitem}
\item Consider each order $o \in O^*$ in turn.  
\item Choose any rack from the pool which contains at least one product that has not yet been fully supplied in the order $o$ being considered. Take products required for that order from the rack for the bin associated with that order (consistent with the number of units needed of each required product and the rack inventory position).
\item Update the inventory position associated with the rack used and return the rack to the pool.
 \item Repeat until all orders have been satisfied. 
\end{compactitem}
The sequence in which racks are taken from the pool therefore constitutes a feasible rack sequence, although rack revisits (so the same rack being picked from the pool more than once) may be required.
Although, clearly, the above could be accomplished in a more effective manner (by considering more than one order when dealing with a rack taken from the pool) it does indicate that a feasible solution must always exist if rack revisits are allowed, although (potentially) we might need a large number of rack revisits.
Note here that the argument presented above shows that the problem of finding a feasible rack sequence, given no restrictions on the number of revisits, can be solved in 
pseudo-polynomial  time  
$O(|R^*| \sum_{i \in N^*} \sum_{o \in O^*} q_{io})$.

As commented above PRSF assumes that each rack visits the picker only once. To remove this assumption let $M$ be the maximum number of rack revisits (so the number of return visits after the first visit) allowed.  Then essentially we create $M$ copies of each of the racks, but link together the number of units of each product taken at each visit.

More formally first recall that $K=|R^*|$ and let $L \leftarrow K$. Without loss of generality label the (uncopied) racks in $R^*$ as $r=1,\ldots,L$. Let the copies of rack $r$ be racks $r+mL, m=1,\ldots,M$. Each copy of rack $r$ has the same inventory position as the original (uncopied) rack, 
i.e.~$s_{i,r+mL} = s_{ir}~m=1,\ldots,M,~\forall i \in N^*,~r=1,\ldots,L$.
 Then as a result of this copying of racks we have that the set $R^*$ has been enlarged to contain both the initial set of racks and their copies so that  we have $|R^*| = K \leftarrow  (M+1)L$. 

Let $\Gamma_{ior}$, integer $\geq0$, be the number of units of product $i \in N^*$ associated with order $o\in O^*$ that are taken from rack $r \in R^*$. Here we need only be concerned with $\Gamma_{ior}$ if $q_{io}s_{ir} > 0$, i.e.~if order $o$ requires product $i$ and rack $r$ contains some of product $i$. Then we add to our formulation the constraints:
\begin{optprog}
& \Gamma_{ior} & = & \sum_{k=1}^{K}\gamma_{iok}z_{kr} & 
\forall i \in N^*;  o \in O^*;  
r \in R^*:~q_{io}s_{ir} > 0
 \label{jeb1} \\
& \sum_{m=0}^{M} \sum_{\substack{o \in O^*,\\q_{io}s_{ir} > 0}} \kern -0.4em  \Gamma_{io,r+mL} & \leq & s_{ir} & \forall i \in N^*;  r \in \{1,\ldots,L\} \label{jeb2} 
\end{optprog}
Equation~(\ref{jeb1}) defines the number of units of product $i$ allocated to order $o$ by rack $r$ in terms of the previously defined variables $\gamma_{iok}$. Equation~(\ref{jeb2}) ensures that across the original rack $r$ and its $M$ copies we do not exceed the inventory position $s_{ir}$ associated with the rack.

Equation~(\ref{jeb1}) is nonlinear because of the product terms $\gamma_{iok}z_{kr}$, but these can be linearised in the following manner. Introduce the variable $\Lambda_{iokr}$ (integer, $\geq 0$) representing the product term $\gamma_{iok}z_{kr}$.
Let $Q_{ior}= \mbox{min}[q_{io},s_{ir}]$ be the maximum number of units of product $i$ associated with order $o$ that can be supplied from rack $r$. Then we have:
\begin{optprog}
& \Gamma_{\tau(o)or} & \geq & \alpha_{ok} + z_{kr} -1  & 
\forall  o \in O_1^*; r \in R^*;  k \in \{1,\ldots,K\}:~
q_{\tau(o)o}s_{\tau(o)r} > 0
\label{jebrev3a} \\
& \Gamma_{ior} & = & \sum_{k=1}^{K}\Lambda_{iokr} & 
\forall i \in N^*; o \in O^*\setminus O_1^*;  
r \in R^*:~ 
q_{io}s_{ir} > 0
\label{jeb5} \\
 & \Lambda_{iokr} & \leq & Q_{ior}z_{kr} & \forall i \in N^*; o \in O^*\setminus O_1^*; k \in \{1,\ldots,K\};
r \in R^* \label{jeb3} 
\end{optprog}
\vspace{-0.5cm}
\begin{optprog}
& \gamma_{iok} {-} Q_{ior}(1{-}z_{kr}) & \leq & \Lambda_{iokr}
\leq & \gamma_{iok} {+} Q_{ior}(1{-}z_{kr}) &
\forall i \in N^*;  o \in O^*\setminus O_1^*;
k \in \{1,\ldots,K\};
r \in R^* \label{jeb4} 
\end{optprog}
Equation~(\ref{jebrev3a}) ensures that for the unique product $\tau(o)$ associated with orders $ o \in O_1^*$  we will have 
$\Gamma_{\tau(o)or}$ set to one if order $o$ is dealt with at the $k$'th  presented rack and rack $r$ is the $k$'th  presented rack. 
Equations~(\ref{jeb5})-(\ref{jeb4}) deal with orders $o \in O^* \setminus O_1^*$. Equation~(\ref{jeb5}) is the linearisation of Equation~(\ref{jeb1}). If $z_{kr}=0$ then Equations~(\ref{jeb3}),(\ref{jeb4}) force $\Lambda_{iokr}$ to be zero and if  $z_{kr}=1$ they force $\Lambda_{iokr}$ to be $\gamma_{iok}$. So $\Lambda_{iokr}$ does correspond to the product term $\gamma_{iok}z_{kr}$. Our formulation for rack sequencing when rack revisits are allowed is therefore PRSF with the addition of Equations~(\ref{jeb2})-(\ref{jeb4}).

For the problem of rack sequencing with $M$ revisits per rack the same complexity result as given in Theorem~\ref{th2} applies. Namely the problem of deciding whether there is a feasible order sequence, given a specific  rack sequence, is strongly 
NP-hard even if $B=1$.

\subsection{Rack sequencing simplification}
With regard to our PRSF, and its extension to rack revisits as presented above, we are aware that a simplification can be made. This simplification relates to the set of orders $o \in O_1^*$ for which the order just comprises a single unit of one product.

First define PRSF1 as the formulation given by PRSF, but with variables $[U_k], [W_{ik}], [w_{ok}]$ and $[v_{ok}]$ removed from the formulation. In particular note that this means that in PRSF1  the right-hand side of 
Equation~(\ref{fin1}) is $B-1$ which (in terms of the mathematical formulation) means that we always have a single bin permanently and solely dedicated for use by orders opened and closed by the same rack. This contrasts with the situation considered in PRSF where (potentially) for some rack 
$k$ all $B$ bins could have been occupied by orders opened previously (at some rack $< k$), these orders remaining open during the presentation of rack $k$, not being closed until some future rack presentation (at some rack $> k$).
\noindent We then have the following theorem.
\begin{theorem}
\label{jebth2}
In   generating a feasible rack sequence for a single picker we can  initially neglect  all orders which comprise a single unit of one product. This is because  a feasible rack sequence for the remaining orders
can be easily modified to incorporate the neglected orders provided that:
\begin{compactitem}
\item the set of racks allocated to the picker collectively contain sufficient of each product to satisfy all of the orders allocated to the picker; and
\item  we have a single bin permanently and solely dedicated for use by orders opened and closed by the same rack (so we are considering formulation PRSF1).
\end{compactitem}
\end{theorem}


\begin{proof}
To prove this theorem suppose that we first solve the rack sequencing formulation: PRSF1 if rack revisits are not allowed; PRSF1 with the addition of Equations~(\ref{jeb2})-(\ref{jeb4}) if rack revisits are allowed,
but ignoring all orders $o \in O_1^*$, so just considering orders $o \in O^* \setminus O_1^*$. 

Assuming a feasible solution is found
then we have a sequence  for processing all of the racks, 
as well as information as to how many units of each product $i$ to allocate to order $o$ from the $k$'th  presented rack (i.e.~$\gamma_{iok}$). Take this rack sequence and for each rack presented, after dealing with the picking of products as detailed in $\gamma_{iok}$, attempt (if possible given the inventory remaining on the rack after picking) to satisfy orders  $o \in O_1^*$. These orders (by definition) can make use of the single bin permanently and solely dedicated  to processing orders which are opened and closed by the same rack and so they will not interfere with any usage of the other $B{-}1$ bins associated with the processing of currently open orders $o \in O^* \setminus O_1^*$. 

Now if Equation~(\ref{eq4}) is satisfied, i.e.~collectively the racks  allocated to the picker, which have subsequently been sequenced, contain sufficient of each product to  supply all of the orders allocated to the picker then it must be true that   after processing the entire rack sequence  all of the orders $o \in O_1^*$ will have been dealt with. 

In other words  any rack sequence associated with the feasible processing of orders $o \in O^* \setminus O_1^*$ is a feasible sequence for the processing of all orders $o \in O^*$ provided that the set of racks allocated to the picker satisfy Equation~(\ref{eq4}) and provided that we have a single bin permanently and solely dedicated  to processing orders which are opened and closed by the same rack. 
\end{proof}

This theorem means that, computationally, if we use PRSF1  we need only focus on finding a rack sequence for orders $o \in O^* \setminus O_1^*$ given that, as in our approach 
(Equations~(\ref{eq1})-(\ref{eq7})),
the racks allocated satisfy 
Equation~(\ref{eq4}).

 With regard to the importance of this theorem recall here that, as discussed above, in our view the rack sequencing problem is essentially a feasibility problem in that we assume that a primary optimisation has already been performed in terms of choosing which orders and racks are allocated to a picker. Hence finding a feasible rack sequence such that all orders can be picked using $B$ bin positions is the principal concern.

Depending upon the problem instance under consideration it may be that PRSF1, with 
Theorem~\ref{jebth2} applied,
is considerably smaller 
(in terms of variables/constraints) 
than PRSF, and hence (computationally) quicker to solve.
As a consequence of this if, by making use of PRSF1 and Theorem~\ref{jebth2}, we can achieve a feasible rack sequence quicker than using PRSF, then there seems no reason not to do so. Although, in terms of the underlying mathematics, PRSF1 assumes that we have a single bin permanently and solely dedicated  to processing orders which are opened and closed by the same rack it is important to stress here that \textbf{\emph{this is a convenient mathematical assumption for computational effectiveness, not a requirement for the 
real-world picker system to operate with a dedicated bin}}. In other words any rack sequence generated by PRSF1 and 
Theorem~\ref{jebth2}  can be applied if the real-world picker system does not operate with a dedicated bin.

\subsection{Integrating order and rack allocation with rack sequencing}

Above we have formulated, separately, the problems of order and rack allocation and rack sequencing. In particular we first solve the order and rack allocation problem for all $P$ pickers considered simultaneously. Then, from the solution obtained which allocates orders and racks to pickers, we have $P$ independent problems with regard to rack sequencing, one for each of the $P$ pickers.

Integrating our order and rack allocation formulation with our rack sequencing formulation into one formulation which  decides order and rack allocation simultaneous with rack sequencing for all $P$ pickers would be a challenging task, both in terms of producing such an integrated formulation and its subsequent computational solution.

As an indication of the difficulties in integrating our currently separate formulations note that our rack sequencing formulation takes as known the set of racks ($R^*$) and the set of orders ($O^*$) allocated to each picker. Knowledge of these sets is used in defining appropriate decision variables, Equations~(\ref{req10})-(\ref{req12}). But if these are decision variables become dependent on sets which are themselves not fixed, but emerge as part of the optimisation solution, then substantive, and non-trivial, changes will be needed to the rack sequencing formulation in order to produce an integrated formulation. 

For this reason we have not, in this paper, proposed the integration of our two formulations, rather we leave such integration as a subject for further future research.

\section{Discussion}
\label{sec:discussion}

In this section we discuss the application of the work presented in this paper within a dynamic environment (as new orders appear), as well as discuss the wider applicability of the work presented in this paper to other automated (non-mobile robot) fulfilment systems.

\subsection{Dynamic problem}

The work presented in this paper deals with two interrelated problems encountered in robotic mobile fulfilment systems, order and rack allocation, and rack sequencing. 
We have considered what would be referred to as the static version of both problems, in that we are dealing with a time period (of arbitrary duration) during which the set of orders under consideration are fixed. 

In a dynamic environment (such as the B2C context within which problems such as those considered in this paper naturally arise) clearly the orders outstanding, as well as the number of units of each product on each rack, will change. 
However it is common in the literature to first study the static problem before considering the dynamic problem, e.g.~see~\cite{beasley00, beasley04} where this was the approach taken with regard to the scheduling of aircraft landings.

We indicate below how our static approaches for both of the problems considered in this paper can be adapted for a dynamic environment in a natural way. Essentially this involves  treating the dynamic problem as a related series of static problems, where each new static problem reflects the situation that applies 
(e.g.~orders outstanding, rack inventory positions, sequenced racks)  when we come to solve the problem.

In terms of a dynamic environment  then in any particular practical setting there are a number of issues concerned with tuning our approach for the environment in which it is to be used. Factors which need to be considered here include:
\begin{itemize}
\item the computation time allowed before a solution is required
\item the speed with which the environment changes (e.g.~how quickly new orders appear, old orders are picked)
\item the time that elapses before we choose to update the order and rack inventory positions and resolve for a new order and rack allocation
\item the value to assign to $C_p$, which governs the number of orders allocated to picker $p$, Equation~(\ref{eq2})
\end{itemize}
We would anticipate that such tuning  could only be accomplished by taking historic data and experimenting with possible tuning options to find an option that gives good results over time.

\subsubsection{Order and rack allocation}
To deal with a dynamic environment we would envisage using our formulation for order and rack allocation in the following repetitive manner:
\begin{itemize}
\item[(a)] Given the current set of orders and current rack inventory positions solve the order and rack allocation formulation.
\item[(b)] Implement the solution and after some time has elapsed add any previously unallocated orders to any new orders that have been received, update the rack inventory positions, and go to (a).
\end{itemize}
Adding a constraint to ensure that a particular order $o$ must be allocated, for example because it is a priority order, is trivial (simply add the constraint 
$\sum_{p=1}^P x_{op}  = 1$). In a similar fashion any order $o$ previously allocated to a picker $p$, but whose allocation cannot now be changed (e.g.~because the order is currently in the process of being picked) can be dealt with by setting $x_{op}=1$.
Adding constraints to constrain the allocation of racks (e.g.~if a rack has been previously allocated to a specific picker and is not currently free to be allocated to any other picker) is also easily done.

\subsubsection{Rack sequencing}
Any change to order and rack allocation as a result of a dynamic update of any previous order and rack allocation solution will potentially necessitate changing rack sequencing and this (for each picker) can be done in the same fashion as above for order and rack allocation:
\begin{itemize}
\item[(a)] Given the current set of orders  and racks allocated to a picker  solve the rack sequencing formulation.
\item[(b)] Implement the solution and once a new order and rack allocation is made, which implies that the set of allocated orders and racks as well as rack inventory positions have been updated, go to (a).
\end{itemize}
Adding constraints to ensure that any previously sequenced orders and/or racks must maintain a given position in any new sequence 
(e.g.~because an order has already been opened and racks to fulfil that order already sequenced) can be easily achieved by adding constraints. For example order $o$ can be fixed in the sequence by setting $\alpha_{ok}$,  $\beta_{ok}$ and  $\gamma_{iok}$ appropriately. Rack $r$ can be fixed as the $k$'th rack presented to the picker by setting $z_{kr}=1$.

\subsection{Other automated fulfilment systems}
Above we have discussed the problems of order and rack allocation, and rack sequencing, in the context of a robotic mobile fulfilment system, where mobile robots bring the racks to the picker. However the formulations (and hence the solution approaches) presented in this paper for both of these problems  potentially have much wider applicability. 

For example suppose we have a number of pickers, but where each picker picks products for customer orders from a conveyor belt which brings racks of product to 
them 
(as in miniload systems~\cite{bozer18, fusser19}). Then there is clearly still a decision problem here, namely which orders should be allocated to which pickers and which racks should be allocated (conveyed) to each picker to supply the orders so allocated? This is precisely the problem  our formulation for order and rack allocation addresses. Moreover we have presented two heuristics (a single picker based heuristic and a partial integer optimisation heuristic)
 for this problem.

Similarly if the picker is constrained by having to assemble complete customer orders where there are a limited number of bins that can be in use simultaneously 
 as the racks are conveyed past them, then we have exactly the same decision  problem relating to rack sequencing as considered in this paper. Namely given the orders and racks allocated to a picker how should the racks be sequenced in terms of their presentation to the picker so that they can fulfil the customer orders allocated to them within their limited bin capacity?  This is precisely the problem  our formulation for rack sequencing addresses. Moreover in terms of finding a feasible rack sequence we have proved that, subject to certain conditions being satisfied, a feasible rack sequence for all orders can be produced by focusing on just a subset of the orders to be dealt with by the picker.

We would note here that both for order and rack allocation, and rack sequencing, the work presented in this paper does not make any assumptions as to  rack inventory positions.  In particular we do not assume a rack contains just one product, nor do we assume (e.g.~as in~\cite{boysen17a}) that a rack contains multiple products, but with enough inventory to satisfy all orders for the products it contains.

\section{Computational results}
\label{sec:results}

In this section we present computational results for the optimal and heuristic solution of our order and rack allocation formulation. In addition we give results for our picker rack sequencing formulation. Results are also presented for  a limited number of instances with characteristics associated  with miniload systems. We discuss the research questions we addressed in this paper and the contribution of this paper to the literature.

In the computational results presented below we used an Intel Xeon @ 2.40GHz with 32GB of RAM and Linux as the operating system. The code was written in C++ and Cplex 12.8~\citep{cplex128} was used as the mixed-integer solver.
A maximum time limit of 5 CPU minutes (300 seconds) was imposed for each instance.

\subsection{Test problems}
\label{sec:test}

To generate orders such that 
the average order comprises just 1.6 items~\citep{boysen19, boysen19b} and the vast number of orders contain only one or two items~\citep{weidinger18c} we used a truncated geometric distribution, the geometric distribution being the discrete equivalent of the continuous exponential distribution (such as was used in~\cite{boysen17a}). In more detail given a parameter $\mu$  the probability that an order contains $m$ items is $\mu(1-\mu)^{(m-1)}/(\sum_{n=1}^4\mu(1-\mu)^{(n-1)})$ for $m=1,2,3,4$. Obviously orders for just one item involve just one unit of a single (randomly chosen) product. Orders for two items contain either two separate products requiring one unit each, or a single product requiring two units (with equal probability). Orders for three items are for one product requiring one unit and another product requiring two units. Orders for four items contain two separate products each requiring two units. The value used for $\mu$ was $\mu=1/1.73$, chosen such that the average order (in the truncated geometric distribution) comprises approximately 1.6 items. Moreover approximately 85\% of orders are for only one or two items. All of the test problems used in this paper have been made publicly available at
\href{https://www.dcc.ufmg.br/~arbex/mobileRacks.html}{https://www.dcc.ufmg.br/$\sim$arbex/mobileRacks.html}
in order that future researchers can compare their results with ours on exactly the same set of problem instances.

With regard to the racks then racks may have up to 50 storage locations (for potentially 50 different products) per face, 
e.g.~see~\cite{cnn18}. Although obviously some racks will have fewer storage locations because of the need to store larger  products.
In our instances we choose to regard racks as having one face and being able to store 25 different products per rack, these products being randomly generated, without replacement, from the entire set of $N$ products. The number of units of each product stored per rack was an integer randomly chosen from $[1,2,3,4]$. 

For a generated test problem with $N$ products, $O$ orders and $R$ racks it could be  infeasible because it is not possible to supply all of the orders given the products stored in the racks generated (i.e.~$\sum_{o=1}^O q_{io} > \sum_{r=1}^R s_{ir}$ for some product $i~i=1,\ldots,N$). Any test problem so generated was rejected and a new test problem generated until we had a feasible test problem.

\subsection{Order and rack allocation, optimal solutions}
\label{section72}

Table~\ref{table1} shows results for the optimal solution of our order and rack allocation formulation when solved using Cplex for instances involving between $N=100$ and $N=500$ products with $O=50$ orders (and a variety of racks $R$ and pickers $P$). In that table the value given for $C_p$ applies for all pickers $p,~p=1,\ldots,P$. The values used for $C_p$ were chosen to cover the range from a low proportion of the total number of orders to all possible orders ($O$).

In Table~\ref{table1} {\it\textbf{T(s)}} denotes the total computation time in seconds.
{\it\textbf{UB}} is the best solution (upper bound) obtained at the end of the search, either when the instance was solved to proven optimality or when the time limit was reached.
 {\it\textbf{B(s)}} denotes the  computation time in seconds before the (final) best solution {\it\textbf{UB}}  was first encountered in the solution process (either at the initial root node or during the tree search).

{\it\textbf{LB}} is the best lower bound obtained at the end of the search either when the instance was solved to proven optimality or when the time limit was reached. In order to improve readability the {\it\textbf{LB}} value is shown as {\it\textbf{UB}} if the problem is solved to proven optimality (in which case the {\it\textbf{UB}} value shown is the optimal solution value).

{\it\textbf{GAP}} is defined as $100 (\text{UB} - \text{LB})/\text{UB}$. {\it\textbf{FLB}} is the (fractional) lower bound obtained at the end of the root node of the branch-and-bound search tree as reported by Cplex and {\it\textbf{FGAP}} is defined as $100 (\text{UB} - \text{FLB})/\text{UB}$. This
means that, for instances solved to optimality, {\it\textbf{FGAP}} denotes the gap between the optimal solution and the lower bound obtained after solving the root node.
Finally, {\it\textbf{NS}} stands for the total number of nodes investigated during the search. 

So, for example, consider the last instance in Table~\ref{table1} 
with $N=500$ and $R=75$ when there were $P=10$ pickers, each picker with a $C_p$ value of 5, so being allocated 5 orders. This problem was solved in 13.23 seconds with the optimal solution being 23. This means that 23 racks was the minimum number of racks needed to deal with the total of $\sum_{p=1}^P C_p=50$ orders allocated to the pickers. A solution of value 23 was first obtained after 3.01 seconds and the lower bound obtained at the end of the root node of the branch-and-bound search tree was 20.00, so the gap (FGAP) value was $100(23-20.00)/23 = 13.04\%$. The tree search required 2138 nodes. As this instance was solved within our time limit of 300 seconds the solution obtained has been proved to be optimal.

As can be clearly seen from Table~\ref{table1} our order and rack allocation formulation solves all our instances with $O=50$  orders to proven optimality very quickly.

\begin{table}[tbp]
\centering
{\tiny
\renewcommand{\tabcolsep}{1mm} 
\renewcommand{\arraystretch}{1.3} 
\begin{tabular}{|c|c|c|rrrrrrrr|}
\hline
Instance & $P$ & $C_p$ & \multicolumn{1}{c}{T(s)} &
\multicolumn{1}{c}{B(s)} &
 \multicolumn{1}{c}{UB} & \multicolumn{1}{c}{GAP} & \multicolumn{1}{c}{LB} & \multicolumn{1}{c}{FGAP} & \multicolumn{1}{c}{FLB} & \multicolumn{1}{c|}{NS}\\
\hline
\multirow{2}{*}{\begin{tabular}{l}$N = 100$\\$R = 50$\end{tabular}} & 1 & 25 &     0.12  & 0.12 & 2 & -- & UB &     1.50 &     1.97 & 7\\
 &  & 50 &     0.10 & 0.01 & 7 & -- & UB &    28.00 &     5.04 & 133\\
\cline{2-11}
 & 5 & 5 &     0.09 & 0.08 & 5 & -- & UB &     0.00 &     5.00 & 1\\
 &  & 7 &     0.15& 0.15 & 5 & -- & UB &     0.00 &     5.00 & 1\\
 &  & 10 &    34.22 & 34.22 & 7 & -- & UB &    28.57 &     5.00 & 2923\\
\cline{2-11}
 & 10 & 2 &     0.08 & 0.03 & 10 & -- & UB &     0.00 &    10.00 & 1\\
 &  & 5 &     3.52 & 3.52 & 10 & -- & UB &     0.00 &    10.00 & 106\\
\cline{1-11}
\multirow{2}{*}{\begin{tabular}{l}$N = 200$\\$R = 50$\end{tabular}} & 1 & 25 &     0.02 & 0.01& 4 & -- & UB &     0.00 &     4.00 & 1\\
 &  & 50 &     0.01 & 0.01 & 11 & -- & UB &     0.00 &    11.00 & 1\\
\cline{2-11}
 & 5 & 5 &     0.13 & 0.10 & 5 & -- & UB &     0.00 &     5.00 & 1\\
 &  & 7 &    68.36 & 0.62 & 8 & -- & UB &    37.50 &     5.00 & 31875\\
 &  & 10 &     4.09 & 3.34 & 12 & -- & UB &    16.67 &    10.00 & 726\\
\cline{2-11}
 & 10 & 2 &     0.07 & 0.01 & 10 & -- & UB &     0.00 &    10.00 & 1\\
 &  & 5 &    91.25 & 86.69 & 15 & -- & UB &    33.33 &    10.00 & 4325\\
\cline{1-11}
\multirow{2}{*}{\begin{tabular}{l}$N = 300$\\$R = 50$\end{tabular}} & 1 & 25 &     0.01 & 0.00 & 5 & -- & UB &     0.00 &     5.00 & 1\\
 &  & 50 &     0.01 & 0.00 & 15 & -- & UB &     0.00 &    15.00 & 1\\
\cline{2-11}
 & 5 & 5 &     0.57 & 0.37 & 6 & -- & UB &    16.67 &     5.00 & 99\\
 &  & 7 &     1.43 & 0.57 & 9 & -- & UB &    24.43 &     6.80 & 364\\
 &  & 10 &     0.82 & 0.70 & 16 & -- & UB &     6.25 &    15.00 & 85\\
\cline{2-11}
 & 10 & 2 &     0.05 & 0.03 & 10 & -- & UB &     0.00 &    10.00 & 1\\
 &  & 5 &    49.65 & 46.44 & 18 & -- & UB &    16.67 &    15.00 & 3076\\
\cline{1-11}
\multirow{2}{*}{\begin{tabular}{l}$N = 400$\\$R = 75$\end{tabular}} & 1 & 25 &     0.03 & 0.03 & 5 & -- & UB &     0.00 &     5.00 & 7\\
 &  & 50 &     0.00 & 0.00 & 19 & -- & UB &     0.00 &    19.00 & 1\\
\cline{2-11}
 & 5 & 5 &     0.67 & 0.67 & 6 & -- & UB &    16.67 &     5.00 & 299\\
 &  & 7 &     4.29 & 0.66 & 10 & -- & UB &    22.10 &     7.79 & 2853\\
 &  & 10 &     0.40 & 0.40 & 19 & -- & UB &     2.63 &    18.50 & 70\\
\cline{2-11}
 & 10 & 2 &     0.05 & 0.01 & 10 & -- & UB &     0.00 &    10.00 & 1\\
 &  & 5 &     4.11 & 4.04 & 20 & -- & UB &     7.50 &    18.50 & 466\\
\cline{1-11}
\multirow{2}{*}{\begin{tabular}{l}$N = 500$\\$R = 75$\end{tabular}} & 1 & 25 &     0.00 & 0.00 & 6 & -- & UB &     0.00 &     6.00 & 1\\
 &  & 50 &     0.00 & 0.00 & 20 & -- & UB &     0.00 &    20.00 & 1\\
\cline{2-11}
 & 5 & 5 &     0.70 & 0.35 & 7 & -- & UB &    20.14 &     5.59 & 75\\
 &  & 7 &     0.85 & 0.85  & 10 & -- & UB &    10.40 &     8.96 & 238\\
 &  & 10 &     0.22 & 0.15 & 20 & -- & UB &     0.00 &    20.00 & 1\\
\cline{2-11}
 & 10 & 2 &     0.11 & 0.04 & 10 & -- & UB &     0.00 &    10.00 & 1\\
 &  & 5 &    13.23 & 3.01 & 23 & -- & UB &    13.04 &    20.00 & 2138\\
\hline
\multicolumn{3}{r}{\textbf{Average:}} & 
\multicolumn{1}{r}{\textbf{    7.98}} & 
\multicolumn{1}{r}{\textbf{    5.35}} & 
\multicolumn{1}{c}{} & \multicolumn{1}{r}{\textbf{0}} & \multicolumn{1}{c}{} & \multicolumn{1}{r}{\textbf{    8.63}} & \multicolumn{1}{c}{} & \multicolumn{1}{r}{\textbf{1425}}\end{tabular}
}
\caption{Results: order and rack allocation, $O=50$ orders}
\label{table1}
\end{table}

Given the results in Table~\ref{table1} we generated larger instances, increasing the number of orders from $O=50$ (using $O=100,150,200$) and also increasing the number of racks. The results for instances with $O=100,150,200$ are shown in 
Tables~\ref{table2}-\ref{table200}.

In Tables~\ref{table2}-\ref{table200},
by contrast with Table~\ref{table1} where all problems were solved to proven optimality, a number of problems were not solved to proven optimality within our (self-imposed) 5 CPU minute time limit. 
Instances that went to time limit are indicated by {\it\textbf{TL}} in these tables. Here the logic for imposing a 5 minute time limit was that we might reasonably expect that in a B2C environment decisions as to order and rack allocation have to be made relatively quickly. 12 of the 35 problems in Table~\ref{table2}, 17 of the 30 problems in Table~\ref{table150}, and 23 of the 30 problems in Table~\ref{table200} went to time limit.

Generalising for $P=5,10$ it appears from Table~\ref{table2} and Table~\ref{table150} that, for a fixed $P$, the problem becomes harder to solve as $\sum_{p=1}^P C_p$ becomes larger. This makes intuitive sense as for a fixed $P$ the term $\sum_{p=1}^P C_p$ defines, out of the $O$ orders, how many are to be allocated to the pickers. The more orders that have to be allocated the greater the number of  racks that will be needed.

With regard to  Table~\ref{table2} and Table~\ref{table150}, we would note that our order and rack allocation formulation seems very effective when there is just a single picker, so $P=1$, with problems of this type all being solved to proven optimality relatively quickly (under 20 seconds even for the most demanding problem with $P=1$ in Table~\ref{table150}). However only one of the six single picker problems with $O=200$ in Table~\ref{table200} was solved to proven optimality. In general we can conclude that the problems with $O=200$ in Table~\ref{table200} are challenging for Cplex to solve within our 5 CPU minute time limit.

Considering Tables~\ref{table2}-\ref{table200} we can see that, although we might impose a longer time limit, and hence resolve some of the current unsolved problems, we will at some point reach the effective computational limit of what can be achieved with linear programming relaxation based tree search. 
Recalling our discussion above with regard to the necessity of having a heuristic approach given the complexity of the order and rack allocation problem, we therefore decided not to impose a longer  time  limit, but instead to investigate our heuristic approaches for the order and rack allocation problem.

\begin{table}[tbp]
\centering
{\tiny
\renewcommand{\tabcolsep}{1mm} 
\renewcommand{\arraystretch}{1.3} 
\begin{tabular}{|c|c|c|rrrrrrrr|}
\hline
Instance & $P$ & $C_p$ & \multicolumn{1}{c}{T(s)} & 
\multicolumn{1}{c}{B(s)} &
\multicolumn{1}{c}{UB} & \multicolumn{1}{c}{GAP} & \multicolumn{1}{c}{LB} & \multicolumn{1}{c}{FGAP} & \multicolumn{1}{c}{FLB} & \multicolumn{1}{c|}{NS}\\
\hline
\multirow{2}{*}{\begin{tabular}{l}$N = 100$\\$R = 100$\end{tabular}} & 1 & 50 &     0.13 & 0.13 & 2 & -- & UB &     7.00 &     1.86 & 28\\
 &  & 100 &     0.26 & 0.20 & 7 & -- & UB &    27.29 &     5.09 & 285\\
\cline{2-11}
 & 5 & 10 &     0.27 & 0.26 & 5 & -- & UB &     0.00 &     5.00 & 1\\
 &  & 15 &     0.78 & 0.65 & 5 & -- & UB &     0.00 &     5.00 & 1\\
 &  & 20 & TL & 4.38 & 9 &    35.78 &     5.78 &    44.44 &     5.00 & 6294\\
\cline{2-11}
 & 10 & 5 &     0.43 & 0.28 & 10 & -- & UB &     0.00 &    10.00 & 1\\
 &  & 10 & TL & 189.30 & 11 &     9.09 &    10.00 &     9.09 &    10.00 & 3979\\
\cline{1-11}
\multirow{2}{*}{\begin{tabular}{l}$N = 200$\\$R = 100$\end{tabular}} & 1 & 50 &     0.04 & 0.02 & 4 & -- & UB &     0.00 &     4.00 & 1\\
 &  & 100 &     0.22 & 0.04 & 13 & -- & UB &    17.54 &    10.72 & 365\\
\cline{2-11}
 & 5 & 10 &     0.53 & 0.43 & 5 & -- & UB &     0.00 &     5.00 & 1\\
 &  & 15 & TL & 5.54 & 9 &    26.67 &     6.60 &    38.22 &     5.56 & 16864\\
 &  & 20 & TL & 5.45 & 18 &    36.50 &    11.43 &    42.67 &    10.32 & 7177\\
\cline{2-11}
 & 10 & 5 &     0.41 & 0.02 & 10 & -- & UB &     0.00 &    10.00 & 1\\
 &  & 10 & TL & 53.53 & 24 &    55.25 &    10.74 &    56.83 &    10.36 & 1894\\
\cline{1-11}
\multirow{2}{*}{\begin{tabular}{l}$N = 300$\\$R = 100$\end{tabular}} & 1 & 50 &     0.08 & 0.04 & 6 & -- & UB &     0.00 &     6.00 & 1\\
 &  & 100 &     0.12 & 0.05 & 18 & -- & UB &    12.61 &    15.73 & 158\\
\cline{2-11}
 & 5 & 10 & TL & 1.61 & 8 &    26.13 &     5.91 &    37.50 &     5.00 & 60413\\
 &  & 15 & TL & 221.00 & 11 &    10.27 &     9.87 &    24.91 &     8.26 & 16649\\
 &  & 20 & TL & 239.93 & 19 &    10.42 &    17.02 &    19.00 &    15.39 & 9009\\
\cline{2-11}
 & 10 & 5 &     0.63 & 0.44 & 10 & -- & UB &     0.00 &    10.00 & 1\\
 &  & 10 & TL & 66.61 & 28 &    41.11 &    16.49 &    45.07 &    15.38 & 3562\\
\cline{1-11}
\multirow{2}{*}{\begin{tabular}{l}$N = 400$\\$R = 100$\end{tabular}} & 1 & 50 &     0.09 & 0.03 & 7 & -- & UB &     0.00 &     7.00 & 1\\
 &  & 100 &     0.14 & 0.04 & 22 & -- & UB &    10.64 &    19.66 & 154\\
\cline{2-11}
 & 5 & 10 & TL & 1.38 & 9 &    18.22 &     7.36 &    28.67 &     6.42 & 46686\\
 &  & 15 &   243.56 & 243.56 & 13 & -- & UB &    16.77 &    10.82 & 9803\\
 &  & 20 &   202.59 & 202.59 & 22 & -- & UB &    12.91 &    19.16 & 6094\\
\cline{2-11}
 & 10 & 5 &     0.49 & 0.44 & 10 & -- & UB &     0.00 &    10.00 & 1\\
 &  & 10 & TL & 40.86 & 29 &    30.48 &    20.16 &    34.10 &    19.11 & 3815\\
\cline{1-11}
\multirow{2}{*}{\begin{tabular}{l}$N = 500$\\$R = 100$\end{tabular}} & 1 & 50 &     0.09 & 0.01 & 8 & -- & UB &     0.00 &     8.00 & 1\\
 &  & 100 &     0.06 & 0.01 & 26 & -- & UB &     0.00 &    26.00 & 1\\
\cline{2-11}
 & 5 & 10 &   169.10 & 2.12 & 10 & -- & UB &    25.80 &     7.42 & 33340\\
 &  & 15 &     9.19 & 3.04 & 15 & -- & UB &    13.80 &    12.93 & 388\\
 &  & 20 &    84.80 & 84.25 & 27 & -- & UB &     7.67 &    24.93 & 2436\\
\cline{2-11}
 & 10 & 5 &     5.48 & 5.48 & 10 & -- & UB &     0.00 &    10.00 & 49\\
 &  & 10 & TL & 273.72 & 33 &    22.27 &    25.65 &    26.30 &    24.32 & 5144\\
\hline
\multicolumn{3}{r}{\textbf{Average:}} 
& \multicolumn{1}{r}{\textbf{  123.41}} 
& \multicolumn{1}{r}{\textbf{  47.07}} 
& \multicolumn{1}{c}{} & \multicolumn{1}{r}{\textbf{    9.21}} & \multicolumn{1}{c}{} & \multicolumn{1}{r}{\textbf{   15.97}} & \multicolumn{1}{c}{} & \multicolumn{1}{r}{\textbf{6703}}\end{tabular}

}
\caption{Results: order and rack allocation, $O=100$ orders}
\label{table2}

\end{table}

\begin{table}[tbp]
\centering
{\tiny
\renewcommand{\tabcolsep}{1mm} 
\renewcommand{\arraystretch}{1.3} 
\begin{tabular}{|c|c|c|rrrrrrrr|}
\hline
Instance & $P$ & $C_p$ & \multicolumn{1}{c}{T(s)} & 
\multicolumn{1}{c}{B(s)} &
\multicolumn{1}{c}{UB} & \multicolumn{1}{c}{GAP} & \multicolumn{1}{c}{LB} & \multicolumn{1}{c}{FGAP} & \multicolumn{1}{c}{FLB} & \multicolumn{1}{c|}{NS}\\
\hline
\multirow{2}{*}{\begin{tabular}{l}$N = 100$\\$R = 100$\end{tabular}} & 1 & 150 &     6.24&     4.44 & 10 & -- & UB &    24.97 &     7.50  & 8961\\
\cline{2-11}
 & 5 & 10 &     0.60  &     0.54& 5 & -- & UB &     0.00 &     5.00 & 1\\
 &  & 20 &     4.61 &     4.61& 5 & -- & UB &     0.00 &     5.00  & 75\\
 &  & 30 & TL &    15.80 & 14 &    43.05 &     7.97 &    47.89 &     7.29  & 4288\\
\cline{2-11}
 & 10 & 10 &     3.26 & 3.20 & 10 & -- & UB &     0.00 &    10.00 &     1\\
 &  & 15 & TL & 64.58 &15 &    33.33 &    10.00 &    33.33 &    10.00 &     1450\\
\cline{1-11}
\multirow{2}{*}{\begin{tabular}{l}$N = 200$\\$R = 100$\end{tabular}} & 1 & 150 &    18.06 &     0.25 & 16 & -- & UB &    24.83 &    12.03 & 17185\\
\cline{2-11}
 & 5 & 10 &     0.55 &     0.49& 5 & -- & UB &     0.00 &     5.00  & 1\\
 &  & 20 & TL & 10.56 & 10 &    39.79 &     6.02 &    44.82 &     5.52 &    5283\\
 &  & 30 & TL & 21.69 &21 &    39.35 &    12.74 &    44.05 &    11.75 &     3483\\
\cline{2-11}
 & 10 & 10 &     7.20 & 5.46 &10 & -- & UB &     0.00 &    10.00 &      1\\
 &  & 15 & TL & 123.19 &23 &    48.32 &    11.89 &    49.15 &    11.70 &    494\\
\cline{1-11}
\multirow{2}{*}{\begin{tabular}{l}$N = 300$\\$R = 100$\end{tabular}} & 1 & 150 &     0.90 &     0.22 & 22 & -- & UB &    13.34 &    19.07 & 910\\
\cline{2-11}
 & 5 & 10 &     1.89 &      1.75 & 5 & -- & UB &     0.00 &     5.00 & 1\\
 &  & 20 & TL & 15.58 &12 &    27.76 &     8.67 &    31.47 &     8.22 &     1894\\
 &  & 30 & TL & 29.99 &29 &    33.05 &    19.42 &    36.13 &    18.52 &     2410\\
\cline{2-11}
 & 10 & 10 & TL & 56.42 &15 &    33.33 &    10.00 &    33.33 &    10.00 &     1695\\
 &  & 15 & TL & 108.34 & 34 &    44.74 &    18.79 &    45.48 &    18.54 &    464\\
\cline{1-11}
\multirow{2}{*}{\begin{tabular}{l}$N = 400$\\$R = 150$\end{tabular}} & 1 & 150 &    15.40 &    15.31& 24 & -- & UB &    14.48 &    20.52  & 11622\\
\cline{2-11}
 & 5 & 10 &     4.98 &      4.98 &5 & -- & UB &     0.00 &     5.00  & 14\\
 &  & 20 & TL & 17.45 &13 &    24.34 &     9.84 &    29.18 &     9.21 &     1964\\
 &  & 30 & TL & 26.51 &29 &    26.89 &    21.20 &    30.52 &    20.15 &     1657\\
\cline{2-11}
 & 10 & 10 & TL & 77.45 &17 &    41.18 &    10.00 &    41.18 &    10.00 &     972\\
 &  & 15 & TL & 129.53 & 33 &    38.25 &    20.38 &    39.08 &    20.11 &   484\\
\cline{1-11}
\multirow{2}{*}{\begin{tabular}{l}$N = 500$\\$R = 150$\end{tabular}} & 1 & 150 &     1.06 &     0.23 & 29 & -- & UB &    10.78 &    25.87 & 616\\
\cline{2-11}
 & 5 & 10 &    12.47 &    12.47& 6 & -- & UB &    16.67 &     5.00  & 359\\
 &  & 20 & TL & 15.48 & 15 &    16.38 &    12.54 &    21.10 &    11.84 &    4194\\
 &  & 30 & TL & 276.60 &35 &    24.04 &    26.59 &    26.91 &    25.58 &    2252\\
\cline{2-11}
 & 10 & 10 & TL &    44.29 & 18 &    31.67 &    12.30 &    32.93 &    12.07 & 495\\
 &  & 15 & TL & 96.55 &42 &    38.56 &    25.80 &    40.30 &    25.07 &     482\\
\hline
\multicolumn{3}{r}{\textbf{Average:}} 
& \multicolumn{1}{r}{\textbf{  172.57 }} 
& \multicolumn{1}{r}{\textbf{  39.47 }} 
& \multicolumn{1}{c}{} & \multicolumn{1}{r}{\textbf{    19.47 }} & \multicolumn{1}{c}{} & \multicolumn{1}{r}{\textbf{   24.40}} & \multicolumn{1}{c}{} & \multicolumn{1}{r}{\textbf{2457}}\end{tabular}

}
\caption{Results: order and rack allocation, $O=150$ orders}
\label{table150}

\end{table}

\begin{table}[tbp]
\centering
{\tiny
\renewcommand{\tabcolsep}{1mm} 
\renewcommand{\arraystretch}{1.3} 
\begin{tabular}{|c|c|c|rrrrrrrr|}
\hline
Instance & $P$ & $C_p$ & \multicolumn{1}{c}{T(s)} & 
\multicolumn{1}{c}{B(s)} &
\multicolumn{1}{c}{UB} & \multicolumn{1}{c}{GAP} & \multicolumn{1}{c}{LB} & \multicolumn{1}{c}{FGAP} & \multicolumn{1}{c}{FLB} & \multicolumn{1}{c|}{NS}\\
\hline

\multirow{2}{*}{\begin{tabular}{l}$N = 100$\\$R = 200$\end{tabular}} & 1 & 200 & TL &     0.17 & 11 &    23.32 &     8.44 &    35.76 &     7.07 & 249839\\
\cline{2-11}
 & 5 & 20 &     7.52 &     7.52 & 5 & -- & UB &     0.00 &     5.00 & 17\\
 &  & 30 & TL &    23.16 & 8 &    37.50 &     5.00 &    37.50 &     5.00 & 4508\\
 &  & 40 & TL &    38.02 & 13 &    45.25 &     7.12 &    46.97 &     6.89 & 987\\
\cline{2-11}
 & 10 & 10 &     8.29 &     8.23 & 10 & -- & UB &     0.00 &    10.00 & 1\\
 &  & 20 & TL &   279.83 & 17 &    41.18 &    10.00 &    41.18 &    10.00 & 611\\
\cline{1-11}
\multirow{2}{*}{\begin{tabular}{l}$N = 200$\\$R = 200$\end{tabular}} & 1 & 200 & TL &    96.89 & 16 &    13.07 &    13.91 &    26.57 &    11.75 & 113291\\
\cline{2-11}
 & 5 & 20 &   109.89 &   109.89 & 5 & -- & UB &     0.00 &     5.00 & 411\\
 &  & 30 & TL &    40.21 & 11 &    40.52 &     6.54 &    42.89 &     6.28 & 461\\
 &  & 40 & TL &    62.37 & 22 &    46.95 &    11.67 &    47.96 &    11.45 & 820\\
\cline{2-11}
 & 10 & 10 &    15.81 &    14.40 & 10 & -- & UB &     0.00 &    10.00 & 1\\
 &  & 20 & TL &   137.39 & 46 &    74.63 &    11.67 &    75.19 &    11.41 & 63\\
\cline{1-11}
\multirow{2}{*}{\begin{tabular}{l}$N = 300$\\$R = 200$\end{tabular}} & 1 & 200 & TL &     2.57 & 22 &    15.47 &    18.60 &    26.27 &    16.22 & 119122\\
\cline{2-11}
 & 5 & 20 & TL &    29.71 & 10 &    42.73 &     5.73 &    44.85 &     5.51 & 1425\\
 &  & 30 & TL &    59.21 & 14 &    36.13 &     8.94 &    38.57 &     8.60 & 471\\
 &  & 40 & TL &    86.67 & 29 &    44.30 &    16.15 &    45.05 &    15.94 & 499\\
\cline{2-11}
 & 10 & 10 &     9.67 &     9.41 & 10 & -- & UB &     0.00 &    10.00 & 1\\
 &  & 20 & TL &   135.47 & 77 &    79.48 &    15.80 &    79.59 &    15.71 & 75\\
\cline{1-11}
\multirow{2}{*}{\begin{tabular}{l}$N = 400$\\$R = 200$\end{tabular}} & 1 & 200 & TL &   124.01 & 25 &     6.70 &    23.33 &    20.70 &    19.83 & 94129\\
\cline{2-11}
 & 5 & 20 & TL &    17.55 & 10 &    31.38 &     6.86 &    35.74 &     6.43 & 4364\\
 &  & 30 & TL &    58.22 & 16 &    30.93 &    11.05 &    32.33 &    10.83 & 715\\
 &  & 40 & TL &    76.71 & 34 &    42.02 &    19.71 &    42.91 &    19.41 & 1020\\
\cline{2-11}
 & 10 & 10 &     7.00 &     6.17 & 10 & -- & UB &     0.00 &    10.00 & 1\\
 &  & 20 & TL &   214.08 & 35 &    44.45 &    19.44 &    45.23 &    19.17 & 213\\
\cline{1-11}
\multirow{2}{*}{\begin{tabular}{l}$N = 500$\\$R = 200$\end{tabular}} & 1 & 200 &    19.85 &     5.87 & 32 & -- & UB &    12.23 &    28.09 & 8100\\
\cline{2-11}
 & 5 & 20 & TL &    18.99 & 13 &    30.02 &     9.10 &    33.38 &     8.66 & 2693\\
 &  & 30 & TL &    56.03 & 22 &    32.88 &    14.77 &    33.95 &    14.53 & 528\\
 &  & 40 & TL &    65.18 & 40 &    30.39 &    27.85 &    31.63 &    27.35 & 799\\
\cline{2-11}
 & 10 & 10 & TL &   122.65 & 15 &    33.33 &    10.00 &    33.33 &    10.00 & 494\\
 &  & 20 & TL &   199.38 & 58 &    52.58 &    27.50 &    53.05 &    27.23 & 29\\
\cline{1-11}
\multirow{2}{*}{\begin{tabular}{l}$N = 1000$\\$R = 500$\end{tabular}} & 1 & 200 & TL &   289.03 & 42 &     6.34 &    39.34 &    12.11 &    36.91 & 63754\\
\cline{2-11}
 & 5 & 20 & TL &    53.75 & 15 &    13.01 &    13.05 &    16.14 &    12.58 & 636\\
 &  & 30 & TL &   202.89 & 27 &    20.81 &    21.38 &    21.61 &    21.16 & 165\\
 &  & 40 & TL &   251.12 & 55 &    33.04 &    36.83 &    33.04 &    36.83 & 10\\
\cline{2-11}
& 10 & 10 & TL &   170.88 & 24 &    47.73 &    12.55 &    &      & 0\\
 &  & 20 & TL & TL &  &  & 35.73 &  &  & 0\\

\hline
\multicolumn{3}{r}{\textbf{Average:}} 
& \multicolumn{1}{r}{\textbf{  246.61 }} 
& \multicolumn{1}{r}{\textbf{  93.71 }} 
& \multicolumn{1}{c}{} & \multicolumn{1}{r}{\textbf{    28.46 }} & \multicolumn{1}{c}{} & \multicolumn{1}{r}{\textbf{   30.76}} & \multicolumn{1}{c}{} & \multicolumn{1}{r}{\textbf{18618}}\end{tabular}

}
\caption{Results: order and rack allocation, $O=200$ orders}
\label{table200}

\end{table}

\subsection{Order and rack allocation, heuristic solutions}
\label{section73}

Tables~\ref{table3}-\ref{table200h} give the results for applying our single picker based  (SPB) heuristic and our partial integer optimisation (PIO) heuristic outlined previously above to the test problems in Tables~\ref{table1}-\ref{table200}. Note here that we exclude the case of a single picker (so $P=1$) from Tables~\ref{table3}-\ref{table200h} since, as commented above, the heuristics proposed will give exactly the same solution  as our optimal approach. For PIO we present results for $\tau=1$ and for $\tau=2$.

In Tables~\ref{table3}-\ref{table200h} we give the optimal (or best-known) solution values as found by Cplex and taken from Tables~\ref{table1}-\ref{table200}. We also give the computation time for our heuristics and the solution values obtained as well as the percentage deviations from the optimal  value, computed as 
100(heuristic solution value~-~optimal value)/(optimal value). If the optimal value is unknown, as occurs for some cases in  
Tables~\ref{table2}-\ref{table200}, we use the best-known solution value in computing the percentage deviation.  A positive percentage deviation indicates a worse result than Cplex, a negative percentage deviation indicates a better result than Cplex.

So here, for example,  for SPB in Table~\ref{table3} we have that for the problem with $O = 50, N = 500, R = 75, P=10, C_p=5$ we found a solution of value 25 in 1.84 seconds (with an associated percentage deviation from the optimal/best-known value of 23 of 8.70\%).

Considering Table~\ref{table3}, for the problems with $O=50$ orders, the average computation time is small in all cases. The lowest average percentage deviation of 2.58\% is associated with PIO ($\tau=2$), requiring an average of 1.37 seconds. This compares with an average percentage deviation of 0\% for Cplex (by definition since percentage deviation is calculated based on the solution value given by Cplex), but with an average computation time of 11.16 seconds.

In Table~\ref{table4}, for the problems with $O=100$ orders, our PIO heuristic with $\tau=1$ performs very well, with a percentage deviation of 0.02\% only marginally above that achieved by Cplex, but in an average computation time of 17.31 seconds, as compared to 172.73 seconds for Cplex. Here the appearance of instances with negative percentage deviations for our heuristics indicate that for these instances there are improved feasible solutions that have not been discovered by Cplex within the 300 second time limit imposed.

In  Table~\ref{table150h} and Table~\ref{table200h}, for the problems with $O=150,200$ orders, both our SPB heuristic and our PIO heuristic with $\tau=1$ give an average negative percentage deviation. Computationally the times required are (on average) below those of Cplex. For the largest problem seen in 
Table~\ref{table200h} (with $O=200, N=1000, R=500, P=10, C_p=20$) Cplex did not return any feasible solution within the computational time limit imposed.

\begin{table}[tbp]
\centering
{\tiny
\renewcommand{\tabcolsep}{1mm} 
\renewcommand{\arraystretch}{1.3} 
\begin{tabular}{|c|c|c|crr|rrr|rrr|rrr|}
\hline
\multirow{2}{*}{Instance} & \multirow{2}{*}{$P$} & \multirow{2}{*}{$C_p$} 
& \multicolumn{3}{c|}{Cplex}
& \multicolumn{3}{c|}{SPB}
& \multicolumn{3}{c|}{PIO, $\tau = 1$} 
& \multicolumn{3}{c|}{PIO, $\tau = 2$} 
\\
 &  &  & 
\multicolumn{1}{c}{Optimal or} &
\multicolumn{1}{c}{T(s)} & 
\multicolumn{1}{c|}{B(s)} &
\multicolumn{1}{c}{T(s)} & \multicolumn{1}{c}{Value} & 
\multicolumn{1}{c|}{\% dev} &
\multicolumn{1}{c}{T(s)} & \multicolumn{1}{c}{Value} & 
\multicolumn{1}{c|}{\% dev} &
\multicolumn{1}{c}{T(s)} & \multicolumn{1}{c}{Value} & 
\multicolumn{1}{c|}{\% dev} 
\\
 &  &  
&  \multicolumn{1}{c}{best-known} 
& \multicolumn{1}{c}{ } 
& \multicolumn{1}{c|}{ } 
&  \multicolumn{1}{c}{ } 
& \multicolumn{1}{c}{ } 
& \multicolumn{1}{c|}{ } 
&  \multicolumn{1}{c}{ } 
& \multicolumn{1}{c}{ } 
& \multicolumn{1}{c|}{ } 
&  \multicolumn{1}{c}{ } 
& \multicolumn{1}{c}{ } 
& \multicolumn{1}{c|}{ } 
\\
\hline
\multirow{2}{*}{\begin{tabular}{l}$N = 100$\\$R = 50$\end{tabular}} 
	&	5	&	5	&	5	& 0.09 & 0.08 &	0.06	&	5	&	0.00	&	0.13	&	5	&	0.00	&	0.12	&	5	&	0.00	\\	   
 	&	  	&	7	&	5	& 0.15 & 0.15 &	0.09	&	5	&	0.00	&	0.12	&	5	&	0.00	&	0.12	&	5	&	0.00	\\	   
 	&	  	&	10	&	7	& 34.22 & 34.22 &	0.21	&	8	&	14.29	&	0.77	&	8	&	14.29	&	5.74	&	7	&	0.00	\\	   
\cline{2-15}
 	&	10	&	2	&	10	& 0.08 & 0.03 &	0.09	&	10	&	0.00	&	0.20	&	10	&	0.00	&	0.18	&	10	&	0.00	\\	   
 	&	  	&	5	&	10	& 3.52 & 3.52 &	1.78	&	12	&	20.00	&	0.84	&	10	&	0.00	&	0.62	&	11	&	10.00	\\	   
\hline																																												   
\multirow{2}{*}{\begin{tabular}{l}$N = 200$\\$R = 50$\end{tabular}}
	&	5	&	5	&	5	& 0.13 & 0.10 &	0.03	&	5	&	0.00	&	0.12	&	5	&	0.00	&   0.16	&	5	&	0.00	\\	   
 	&	  	&	7	&	8	& 68.36 & 0.62 &	0.25	&	8	&	0.00	&	0.51	&	8	&	0.00	&	8.41	&	8	&	0.00	\\	   
 	&	  	&	10	&	12	& 4.09 & 3.34 &	0.20	&	12	&	0.00	&	0.68	&	13	&	8.33	&	2.47	&	12	&	0.00	\\	   
\cline{2-15} 
 	&	10	&	2	&	10	& 0.07 & 0.01 &   0.08	&	10	&	0.00	&	0.18	&	10	&	0.00	&	0.11	&	10	&	0.00	\\	   
 	&	  	&	5	&	15	& 91.25 & 86.69 &   11.52	&	16	&	6.67	&	2.17	&	16	&	6.67	&	3.40	&	16	&	6.67	\\	   
\hline																																												   
\multirow{2}{*}{\begin{tabular}{l}$N = 300$\\$R = 50$\end{tabular}}
	&	5	&	5	&	6	& 0.57 & 0.37 &	0.07	&	6	&	0.00	&	0.29	&	6	&	0.00	&	0.17	&	6	&	0.00	\\	   
 	&	  	&	7	&	9	& 1.43 & 0.57 &	0.20	&	9	&	0.00	&	0.58	&	9	&	0.00	&	1.25	&	9	&	0.00	\\	   
 	&	  	&	10	&	16	& 0.82 & 0.70 &	0.20	&	16	&	0.00	&	0.48	&	17	&	6.25	&	1.74	&	17	&	6.25	\\	   
\cline{2-15}
 	&	10	&	2	&	10	& 0.05 & 0.03 &	0.06	&	10	&	0.00	&	0.20	&	10	&	0.00	&	0.12	&	10	&	0.00	\\	   
 	&	  	&	5	&	18	& 49.65 & 46.44 &	12.54	&	19	&	5.56	&	1.00	&	18	&	0.00	&	1.47	&	19	&	5.56	\\	   
\hline																																												   
\multirow{2}{*}{\begin{tabular}{l}$N = 400$\\$R = 75$\end{tabular}}
	&	5	&	5	&	6	& 0.67 & 0.67 &	0.06	&	6	&	0.00	&	0.32	&	7	&	16.67	&	0.31	&	7	&	16.67	\\	   
 	&	  	&	7	&	10	& 4.29 & 0.66 &	0.18	&	10	&	0.00	&	0.63	&	10	&	0.00	&	3.56	&	10	&	0.00	\\	   
 	&	  	&	10	&	19	& 0.40 & 0.40 &	0.08	&	21	&	10.53	&	0.18	&	19	&	0.00	&	0.13	&	19	&	0.00	\\	   
\cline{2-15}
 	&	10	&	2	&	10	& 0.05 & 0.01 &	0.09	&	10	&	0.00	&	0.22	&	10	&	0.00	&	0.15	&	10	&	0.00	\\	   
 	&	  	&	5	&	20	& 4.11 & 4.04 &	0.31	&	22	&	10.00	&	0.79	&	21	&	5.00	&	0.83	&	21	&	5.00	\\	   
\hline																																												   
\multirow{2}{*}{\begin{tabular}{l}$N = 500$\\$R = 75$\end{tabular}} 
	&	5	&	5	&	7	& 0.70 & 0.35 &	0.15	&	7	&	0.00	&	0.50	&	7	&	0.00	&	0.65	&	7	&	0.00	\\	   
 	&	  	&	7	&	10	& 0.85 & 0.85 &   0.29	&	11	&	10.00	&	0.38	&	11	&	10.00	&	1.09	&	11	&	10.00	\\	   
 	&	  	&	10	&	20	& 0.22 & 0.15 &	0.26	&	21	&	5.00	&	0.34	&	20	&	0.00	&	0.30	&	20	&	0.00	\\	   
\cline{2-15}
 	&	10	&	2	&	10	& 0.11 & 0.04 &	0.16	&	10	&	0.00	&	0.32	&	10	&	0.00	&	0.22	&	10	&	0.00	\\	   
 	&	  	&	5	&	23	& 13.23 & 3.01 &	1.84	&	25	&	8.70	&	1.33	&	24	&	4.35	&	0.95	&	24	&	4.35	\\	 

\hline																	   
\multicolumn{3}{|r|}{\textbf{Average:}}	 & &  \textbf{11.16} 
 & \textbf{7.48} 
&
\textbf{1.23}	&	&	\textbf{3.63} &	\textbf{0.53}	&	&	\textbf{2.86}	&	\textbf{1.37}	&	&	\textbf{2.58}\\
\hline

\end{tabular}
}
\caption{Heuristic results: order and rack allocation, $O=50$}
\label{table3}
\end{table}

\begin{table}[tbp]
\centering
{\tiny
\renewcommand{\tabcolsep}{1mm} 
\renewcommand{\arraystretch}{1.3} 
\begin{tabular}{|c|c|c|crr|rrr|rrr|rrr|}
\hline
\multirow{2}{*}{Instance} & \multirow{2}{*}{$P$} & \multirow{2}{*}{$C_p$} 
& \multicolumn{3}{c|}{Cplex}
& \multicolumn{3}{c|}{SPB}
& \multicolumn{3}{c|}{PIO, $\tau = 1$} 
& \multicolumn{3}{c|}{PIO, $\tau = 2$} 
\\
 &  &  & 
\multicolumn{1}{c}{Optimal or} &
\multicolumn{1}{c}{T(s)} & 
\multicolumn{1}{c|}{B(s)} &
\multicolumn{1}{c}{T(s)} & \multicolumn{1}{c}{Value} & 
\multicolumn{1}{c|}{\% dev} &
\multicolumn{1}{c}{T(s)} & \multicolumn{1}{c}{Value} & 
\multicolumn{1}{c|}{\% dev} &
\multicolumn{1}{c}{T(s)} & \multicolumn{1}{c}{Value} & 
\multicolumn{1}{c|}{\% dev} 
\\
 &  &  
&  \multicolumn{1}{c}{best-known} 
& \multicolumn{1}{c}{ } 
& \multicolumn{1}{c|}{ } 
&  \multicolumn{1}{c}{ } 
& \multicolumn{1}{c}{ } 
& \multicolumn{1}{c|}{ } 
&  \multicolumn{1}{c}{ } 
& \multicolumn{1}{c}{ } 
& \multicolumn{1}{c|}{ } 
&  \multicolumn{1}{c}{ } 
& \multicolumn{1}{c}{ } 
& \multicolumn{1}{c|}{ } 
\\
\hline
\multirow{2}{*}{\begin{tabular}{l}$N = 100$\\$R = 100$\end{tabular}} 
	&	5	&	10	&	5	& 0.27 & 0.26 &	0.10	&	5	&	0.00	&	0.23	&	5	&	0.00	&	0.26	&	5	&	0.00	\\	   
 	&	  	&	15	&	5	&	0.78 & 0.65 & 0.16	&	5	&	0.00	&   0.38	&	5	&	0.00	&   0.52	&	5	&	0.00	\\	   
 	&	  	&	20	&	9	&TL & 4.38 &	19.33	&	9	&	0.00	&	12.95	&	9	&	0.00	&	168.34	&	9	&	0.00	\\	   
\cline{2-15}
 	&	10	&	5	&	10	& 0.43 & 0.28	& 0.18	&	10	&	0.00	&	0.56	&	10	&	0.00	&	0.37	&	10	&	0.00	\\	   
 	&	  	&	10	&	11	& TL & 189.30 &	29.27	&	12	&	9.09	&	5.95	&	12	&	9.09	&	46.05	&	12	&	9.09	\\	   
\hline
\multirow{2}{*}{\begin{tabular}{l}$N = 200$\\$R = 100$\end{tabular}}
	&	5	&	10	&	5	&	0.53 & 0.43 & 0.11	&	5	&	0.00	&	0.32	&	5	&	0.00	&	0.30	&	5	&	0.00	\\	   
 	&	  	&	15	&	9	& TL & 5.54 &	0.92	&	9	&	0.00	&	12.27	&	9	&	0.00	&	139.48	&	9	&	0.00	\\	   
 	&	  	&	20	&	18	& TL & 5.45 &	110.73	&	16	&	-11.11	&	28.68	&	16	&	-11.11	&	275.01	&	16	&	-11.11	\\	   
\cline{2-15}
 	&	10	&	5	&	10	&	0.41	&	0.02	&
0.14 & 10 & 0.00
		&   0.61	&	10 & 0.00	&	0.42	&	10	&	0.00	\\	   
 	&	  	&	10	&	24	&  TL & 53.53 &	TL	&	20	&	-16.67	&	44.30	&	19	&	-20.83	&	299.43	&	19	&	-20.83	\\	   
\hline
\multirow{2}{*}{\begin{tabular}{l}$N = 300$\\$R = 100$\end{tabular}}
	&	5	&	10	&	8	&  TL & 1.61 &	0.28	&	8	&	0.00	&	3.76	&	8	&	0.00	&	300.36	&	8	&	0.00	\\	   
 	&	  	&	15	&	11	& TL & 221.00 &	1.02	&	11	&	0.00	&	2.99	&	11	&	0.00	&	41.93	&	11	&	0.00	\\	   
 	&	  	&	20	&	19 & TL & 239.93	&	6.76	&	21	&	10.53	&	9.35	&	21	&	10.53	&	278.77	&	20	&	5.26	\\	   
\cline{2-15}
 	&	10	&	5	&	10	& 0.63 & 0.44 &	0.21	&	10	&	0.00	&   0.66	&	10	&	0.00	&   0.48	&	10	&	0.00	\\	   
 	&	  	&	10	&	28 & TL & 66.61 	&	241.04	&	25	&	-10.71	&	54.74	&	24	&	-14.29	&	240.35	&	23	&	-17.86	\\	   
\hline
\multirow{2}{*}{\begin{tabular}{l}$N = 400$\\$R = 100$\end{tabular}}
	&	5	&	10	&	9	& TL & 1.38 &	0.26	&	9	&	0.00	&	2.04	&	9	&	0.00	&   60.27	&	9	&	0.00	\\	   
 	&	  	&	15	&	13	& 243.56 & 243.56 &	0.32	&	14	&	7.69	&	3.06	&	14	&	7.69	&   83.44	&	13	&	0.00	\\	   
 	&	  	&	20	&	22 & 202.59 & 202.59	&	10.59	&	24	&	9.09	&	17.91	&	24	&	9.09	&   275.08	&	23	&	4.55	\\	   
\cline{2-15}
 	&	10	&	5	&	10	& 0.49 & 0.44 &	0.14	&	10	&	0.00	&   1.06	&	10	&	0.00	&	7.71	&	11	&	10.00	\\	   
 	&	  	&	10	&	29  & TL & 40.86 	&	212.97	&	26	&	-10.34	&	99.30	&	26	&	-10.34	&	196.49	&	26	&	-10.34	\\	   
\hline
\multirow{2}{*}{\begin{tabular}{l}$N = 500$\\$R = 100$\end{tabular}}
	&	5	&	10	&	10	& 169.10 & 2.12 &	0.74	&	10	&	0.00	&	3.23	&	10	&	0.00	&   58.21	&	10	&	0.00	\\	   
 	&	  	&	15	&	15	& 9.19 & 3.04 &	0.37	&	16	&	6.67	&	4.08	&	16	&	6.67	&	50.26	&	16	&	6.67	\\	   
 	&	  	&	20	&	27	& 84.80 & 84.25 &	0.37	&	29	&	7.41	&	8.34	&	27	&	0.00	&	184.82	&	28	&	3.70	\\	   
\cline{2-15}
 	&	10	&	5	&	10	& 5.48 & 5.48 &	1.17	&	11	&	10.00	&   3.86	&	12	&	20.00	&	2.16	&	11	&	10.00	\\	   
 	&	  	&	10	&	33	& TL & 273.72 &	TL	&	33	&	0.00	&	112.12	&	31	&	-6.06	&	298.31	&	32	&	-3.03	\\	   
\hline
\multicolumn{3}{|r|}{\textbf{Average:}}	 & &  \textbf{172.73} 
 & \textbf{65.87} 
&	\textbf{49.50}	&	&	\textbf{0.47} & \textbf{17.31} & & \textbf{0.02} & \textbf{120.35} & & \textbf{-0.56}\\
\hline

\end{tabular}
}
\caption{Heuristic results: order and rack allocation, $O=100$}
\label{table4}
\end{table}

\begin{table}[!htb]
\centering
{\tiny
\renewcommand{\tabcolsep}{1mm} 
\renewcommand{\arraystretch}{1.3} 
\begin{tabular}{|c|c|c|crr|rrr|rrr|rrr|}
\hline
\multirow{2}{*}{Instance} & \multirow{2}{*}{$P$} & \multirow{2}{*}{$C_p$} 
& \multicolumn{3}{c|}{Cplex}
& \multicolumn{3}{c|}{SPB}
& \multicolumn{3}{c|}{PIO, $\tau = 1$} 
& \multicolumn{3}{c|}{PIO, $\tau = 2$} 
\\
 &  &  & 
\multicolumn{1}{c}{Optimal or} &
\multicolumn{1}{c}{T(s)} & 
\multicolumn{1}{c|}{B(s)} &
\multicolumn{1}{c}{T(s)} & \multicolumn{1}{c}{Value} & 
\multicolumn{1}{c|}{\% dev} &
\multicolumn{1}{c}{T(s)} & \multicolumn{1}{c}{Value} & 
\multicolumn{1}{c|}{\% dev} &
\multicolumn{1}{c}{T(s)} & \multicolumn{1}{c}{Value} & 
\multicolumn{1}{c|}{\% dev} 
\\
 &  &  
&  \multicolumn{1}{c}{best-known} 
& \multicolumn{1}{c}{ } 
& \multicolumn{1}{c|}{ } 
&  \multicolumn{1}{c}{ } 
& \multicolumn{1}{c}{ } 
& \multicolumn{1}{c|}{ } 
&  \multicolumn{1}{c}{ } 
& \multicolumn{1}{c}{ } 
& \multicolumn{1}{c|}{ } 
&  \multicolumn{1}{c}{ } 
& \multicolumn{1}{c}{ } 
& \multicolumn{1}{c|}{ } 
\\
\hline
\multirow{2}{*}{\begin{tabular}{l}$N = 100$\\$R = 100$\end{tabular}} 
& 5 & 10 & 5 
&	0.60	&	0.54
&     0.21 & 5 &     0.00 &     0.66 & 5 &     0.00 &     0.51 & 5 &     0.00\\
 &  & 20 & 5 
&	4.61	&	4.61
&     2.20 & 5 &     0.00 &     1.84 & 6 &    20.00 &     2.28 & 6 &    20.00\\
 &  & 30 & 14 
&	TL	&	15.80
&    15.96 & 12 &   -14.29 &    31.35 & 11 &   -21.43 &   275.07 & 12 &   -14.29\\
\cline{2-15}
 & 10 & 10 & 10 
&	3.26	&	3.20
&     1.60 & 10 &     0.00 &     1.64 & 10 &     0.00 &     1.27 & 10 &     0.00\\
 &  & 15 & 15 
&	TL	&	64.58
&   TL & 15 &     0.00 &   TL & 16 &     6.67 &   TL & 16 &     6.67\\
\cline{1-15}
\multirow{2}{*}{\begin{tabular}{l}$N =200$\\$R = 100$\end{tabular}}
& 5 & 10 & 5 
&	0.55	&	0.49
&     0.23 & 5 &     0.00 &     0.61 & 5 &     0.00 &     0.62 & 5 &     0.00\\
 &  & 20 & 10 
&	TL	&	10.56
&    30.58 & 9 &   -10.00 &    72.09 & 9 &   -10.00 &   276.58 & 9 &   -10.00\\
 &  & 30 & 21 
&	TL	&	21.69
&    10.37 & 19 &    -9.52 &   131.26 & 19 &    -9.52 &   276.89 & 18 &   -14.29\\
\cline{2-15}
 & 10 & 10 & 10 
&	7.20	&	5.46
&     1.26 & 10 &     0.00 &     3.10 & 10 &     0.00 &     8.55 & 10 &     0.00\\
 &  & 15 & 23 
&	TL	&	123.19
&   TL& 21 &    -8.70 &   TL & 21 &    -8.70 &   TL& 23 &     0.00\\
\cline{1-15}
\multirow{2}{*}{\begin{tabular}{l}$N = 300$\\$R = 100$\end{tabular}}
 & 5 & 10 & 5 
&	1.89	&	1.75
&     0.24 & 5 &     0.00 &     0.70 & 5 &     0.00 &     0.74 & 5 &     0.00\\
 &  & 20 & 12 
&	TL	&	15.58
&    13.93 & 11 &    -8.33 &    22.42 & 12 &     0.00 &   276.26 & 12 &     0.00\\
 &  & 30 & 29 
&	TL	&	29.99
&     4.69 & 26 &   -10.34 &   207.36 & 25 &   -13.79 &   TL& 25 &   -13.79\\
\cline{2-15}
 & 10 & 10 & 15 
&	TL	&	56.42
&   244.79 & 13 &   -13.33 &    69.39 & 13 &   -13.33 &   293.23 & 14 &    -6.67\\
 &  & 15 & 34 
&	TL	&	108.34 &
   TL& 34 &     0.00 &   TL & 30 &   -11.76 &   TL& 42 &    23.53\\
\cline{1-15}
\multirow{2}{*}{\begin{tabular}{l}$N = 400$\\$R = 150$\end{tabular}}
& 5 & 10 & 5 
&	4.98	&	4.98
&     0.50 & 6 &    20.00 &     1.20 & 6 &    20.00 &     1.06 & 6 &    20.00\\
 &  & 20 & 13 
&	TL	&	17.45
&     8.57 & 12 &    -7.69 &    18.03 & 12 &    -7.69 &   196.67 & 12 &    -7.69\\
 &  & 30 & 29 
&	TL	&	26.51
&     2.03 & 29 &     0.00 &   278.36 & 29 &     0.00 &   276.54 & 29 &     0.00\\
\cline{2-15}
 & 10 & 10 & 17 
&	TL	&	77.45
&   TL& 15 &   -11.76 &   245.52 & 15 &   -11.76 &   TL& 15 &   -11.76\\
 &  & 15 & 33 
&	TL	&	129.53
&   TL& 36 &     9.09 &   278.63 & 30 &    -9.09 &   TL& 46 &    39.39\\
\cline{1-15}
\multirow{2}{*}{\begin{tabular}{l}$N = 500$\\$R = 150$\end{tabular}}
& 5 & 10 & 6 
&	12.47	&	12.47
&     0.51 & 7 &    16.67 &    16.41 & 7 &    16.67 &    49.91 & 7 &    16.67\\
 &  & 20 & 15 
&	TL	&	15.48
&     3.56 & 14 &    -6.67 &     9.09 & 15 &     0.00 &   189.23 & 14 &    -6.67\\
 &  & 30 & 35 
&	TL	&	276.60
&    27.74 & 35 &     0.00 &   259.25 & 33 &    -5.71 &   275.63 & 32 &    -8.57\\
\cline{2-15}
 & 10 & 10 & 18 
&	TL	&	44.29
&   257.29 & 17 &    -5.56 &   266.82 & 17 &    -5.56 &   TL& 18 &     0.00\\
 &  & 15 & 42 
&	TL	&	96.55
&   TL& 37 &   -11.90 &   168.43 & 36 &   -14.29 &   TL& 53 &    26.19\\
\hline
\multicolumn{3}{|r|}{\textbf{Average:}}	 & &  \textbf{205.42} 
 & \textbf{46.54} 
&
\textbf{97.05}	&	&	\textbf{-2.89} &	\textbf{119.37}	&	&	\textbf{-3.17}	&	\textbf{192.04}	&	&	\textbf{2.35}\\
\hline
\end{tabular}
}
\caption{Heuristic results: order and rack allocation, $O=150$}
\label{table150h}
\end{table}

\begin{table}[!htb]
\centering
{\tiny
\renewcommand{\tabcolsep}{1mm} 
\renewcommand{\arraystretch}{1.3} 
\begin{tabular}{|c|c|c|crr|rrr|rrr|rrr|}
\hline
\multirow{2}{*}{Instance} & \multirow{2}{*}{$P$} & \multirow{2}{*}{$C_p$} 
& \multicolumn{3}{c|}{Cplex}
& \multicolumn{3}{c|}{SPB}
& \multicolumn{3}{c|}{PIO, $\tau = 1$} 
& \multicolumn{3}{c|}{PIO, $\tau = 2$} 
\\
 &  &  & 
\multicolumn{1}{c}{Optimal or} &
\multicolumn{1}{c}{T(s)} & 
\multicolumn{1}{c|}{B(s)} &
\multicolumn{1}{c}{T(s)} & \multicolumn{1}{c}{Value} & 
\multicolumn{1}{c|}{\% dev} &
\multicolumn{1}{c}{T(s)} & \multicolumn{1}{c}{Value} & 
\multicolumn{1}{c|}{\% dev} &
\multicolumn{1}{c}{T(s)} & \multicolumn{1}{c}{Value} & 
\multicolumn{1}{c|}{\% dev} 
\\
 &  &  
&  \multicolumn{1}{c}{best-known} 
& \multicolumn{1}{c}{ } 
& \multicolumn{1}{c|}{ } 
&  \multicolumn{1}{c}{ } 
& \multicolumn{1}{c}{ } 
& \multicolumn{1}{c|}{ } 
&  \multicolumn{1}{c}{ } 
& \multicolumn{1}{c}{ } 
& \multicolumn{1}{c|}{ } 
&  \multicolumn{1}{c}{ } 
& \multicolumn{1}{c}{ } 
& \multicolumn{1}{c|}{ } 
\\
\hline

\multirow{2}{*}{\begin{tabular}{l}$N = 100$\\$R = 200$\end{tabular}} 
& 5 & 20 & \multicolumn{1}{c}{5} &     7.52 &     7.52 &     1.11 & 5 &     0.00 &     1.73 & 5 &     0.00 &     2.81 & 5 &     0.00\\
 &  & 30 & \multicolumn{1}{c}{8} & TL &    23.16 &     5.03 & 7 &   -12.50 &   157.28 & 7 &   -12.50 &   158.35 & 7 &   -12.50\\
 &  & 40 & \multicolumn{1}{c}{13} & TL &    38.02 &    18.82 & 13 &     0.00 &   282.83 & 12 &    -7.69 &   275.19 & 14 &     7.69\\
\cline{2-15}
 & 10 & 10 & \multicolumn{1}{c}{10} &     8.29 &     8.23 &     1.66 & 10 &     0.00 &     2.85 & 10 &     0.00 &     2.26 & 10 &     0.00\\
 &  & 20 & \multicolumn{1}{c}{17} & TL &   279.83 &   TL & 17 &     0.00 &   TL & 16 &    -5.88 &   TL & 19 &    11.76\\
\cline{1-15}
\multirow{2}{*}{\begin{tabular}{l}$N = 200$\\$R = 200$\end{tabular}} 
& 5 & 20 & \multicolumn{1}{c}{5} &   109.89 &   109.89 &     4.93 & 6 &    20.00 &     3.31 & 6 &    20.00 &    12.16 & 6 &    20.00\\
 &  & 30 & \multicolumn{1}{c}{11} & TL &    40.21 &     9.18 & 10 &    -9.09 &   139.59 & 9 &   -18.18 &   287.25 & 10 &    -9.09\\
 &  & 40 & \multicolumn{1}{c}{22} & TL &    62.37 &   137.59 & 20 &    -9.09 &   TL & 21 &    -4.55 &   TL & 21 &    -4.55\\
\cline{2-15}
 & 10 & 10 & \multicolumn{1}{c}{10} &    15.81 &    14.40 &     1.48 & 10 &     0.00 &     4.32 & 10 &     0.00 &     3.36 & 10 &     0.00\\
 &  & 20 & \multicolumn{1}{c}{46} & TL &   137.39 &   TL & 27 &   -41.30 &   TL & 26 &   -43.48 &   TL & 58 &    26.09\\
\cline{1-15}
\multirow{2}{*}{\begin{tabular}{l}$N = 300$\\$R = 200$\end{tabular}} 
& 5 & 20 & \multicolumn{1}{c}{10} & TL &    29.71 &   TL & 10 &     0.00 &   TL & 10 &     0.00 &   293.35 & 10 &     0.00\\
 &  & 30 & \multicolumn{1}{c}{14} & TL &    59.21 &    15.00 & 13 &    -7.14 &   109.57 & 13 &    -7.14 &   279.93 & 12 &   -14.29\\
 &  & 40 & \multicolumn{1}{c}{29} & TL &    86.67 &    12.03 & 27 &    -6.90 &   294.90 & 26 &   -10.34 &   285.44 & 26 &   -10.34\\
\cline{2-15}
 & 10 & 10 & \multicolumn{1}{c}{10} &     9.67 &     9.41 &     1.07 & 10 &     0.00 &     4.35 & 10 &     0.00 &     3.86 & 10 &     0.00\\
 &  & 20 & \multicolumn{1}{c}{77} & TL &   135.47 &   TL & 36 &   -53.25 &   TL & 34 &   -55.84 &   TL & 52 &   -32.47\\
\cline{1-15}
\multirow{2}{*}{\begin{tabular}{l}$N = 400$\\$R = 200$\end{tabular}} 
& 5 & 20 & \multicolumn{1}{c}{10} & TL &    17.55 &   TL & 10 &     0.00 &   TL & 10 &     0.00 &   TL & 10 &     0.00\\
 &  & 30 & \multicolumn{1}{c}{16} & TL &    58.22 &     5.71 & 16 &     0.00 &    52.57 & 15 &    -6.25 &   275.55 & 15 &    -6.25\\
 &  & 40 & \multicolumn{1}{c}{34} & TL &    76.71 &    15.17 & 31 &    -8.82 &   TL & 37 &     8.82 &   TL & 30 &   -11.76\\
\cline{2-15}
 & 10 & 10 & \multicolumn{1}{c}{10} &     7.00 &     6.17 &     1.50 & 10 &     0.00 &     4.67 & 10 &     0.00 &     3.23 & 10 &     0.00\\
 &  & 20 & \multicolumn{1}{c}{35} & TL &   214.08 &   TL & 38 &     8.57 &   TL & 38 &     8.57 &   TL & 39 &    11.43\\
\cline{1-15}
\multirow{2}{*}{\begin{tabular}{l}$N = 500$\\$R = 200$\end{tabular}} 
& 5 & 20 & \multicolumn{1}{c}{13} & TL &    18.99 &     3.85 & 12 &    -7.69 &    86.74 & 12 &    -7.69 &   181.14 & 11 &   -15.38\\
 &  & 30 & \multicolumn{1}{c}{22} & TL &    56.03 &     7.94 & 21 &    -4.55 &   128.75 & 20 &    -9.09 &   281.18 & 20 &    -9.09\\
 &  & 40 & \multicolumn{1}{c}{40} & TL &    65.18 &    21.44 & 39 &    -2.50 &   TL & 50 &    25.00 &   283.03 & 40 &     0.00\\
\cline{2-15}
 & 10 & 10 & \multicolumn{1}{c}{15} & TL &   122.65 &    28.15 & 13 &   -13.33 &    90.19 & 13 &   -13.33 &   182.42 & 13 &   -13.33\\
 &  & 20 & \multicolumn{1}{c}{58} & TL &   199.38 &   TL & 43 &   -25.86 &   TL & 44 &   -24.14 &   TL & 54 &  -6.90\\
\cline{1-15}
\multirow{2}{*}{\begin{tabular}{l}$N = 1000$\\$R = 500$\end{tabular}} 
& 5 & 20 & \multicolumn{1}{c}{15} & TL &    53.75 &     2.24 & 15 &     0.00 &   118.29 & 15 &     0.00 &   148.14 & 15 &     0.00\\
 &  & 30 & \multicolumn{1}{c}{27} & TL &   202.89 &     8.67 & 27 &     0.00 &   252.53 & 26 &    -3.70 &   151.27 & 28 &     3.70\\
 &  & 40 & \multicolumn{1}{c}{55} & TL &   251.12 &     5.40 & 51 &    -7.27 &   279.93 & 47 &   -14.55 &   276.57 & 49 &   -10.91\\
\cline{2-15}
 & 10 & 10 & \multicolumn{1}{c}{24} & TL &   170.88 &   TL & 19 &   -20.83 &   TL & 19 &   -20.83 &   TL & 19 &   -20.83\\
 &  & 20 &  & TL &    TL &   TL & 54 &  &   TL & 56 &  &   TL & 60 & \\

\hline
\multicolumn{3}{|r|}{\textbf{Average:}}	 & &  \textbf{245.27} 
 & \textbf{95.17} 
&
\textbf{100.48}	&	&	\textbf{-6.95} &	\textbf{187.15}	&	&	\textbf{-6.99}	&	\textbf{212.88}	&	&	\textbf{-3.35}\\
\hline
\end{tabular}
}
\caption{Heuristic results: order and rack allocation, $O=200$}
\label{table200h}
\end{table}

\subsection{Summary heuristic comparison}
\label{section74}

Table~\ref{table_compare} presents a summary comparison of a direct application of Cplex alone with the heuristics we have proposed in this paper. Since the percentage deviations are all calculated with regard to the best solutions given by Cplex the percentage deviation for Cplex is shown as 0\%. The percentage deviations values given are  the averages as shown in Tables~\ref{table3}-\ref{table200h}. 
Here a negative average \% deviation indicates a result better than that achieved by Cplex. The three values shown in the {\it\textbf{\# instances}} column in Table~\ref{table_compare} are respectively: 
\begin{compactitem}
\item the number of instances in which the corresponding heuristic improved on the solution found by Cplex alone (so found a solution involving fewer racks); labelled {\it\textbf{lt}} (less than) 
\item  the number of instances in  which the corresponding heuristic gave a worse solution than that found by Cplex alone (so found a solution involving more racks); labelled {\it\textbf{gt}}  (greater than)
\item  the number of instances in  which the corresponding heuristic gave the same solution as that found by Cplex alone (so found a solution involving the same number of racks); labelled {\it\textbf{eq}}  (equal) 
\end{compactitem}

\noindent We mentioned above that the key discriminator between just applying Cplex alone and applying any other heuristic is the results obtained. In particular the balance between the computation time required and the quality of results obtained. 
In general terms Table~\ref{table_compare} illustrates that the performance of the heuristics, as compared to Cplex alone, improves as  problem size increases.
This can be seen  both in the increasingly negative percentage deviation values (indicating an increasing reduction in overall rack usage as compared with Cplex alone) and also in the increasing number of instances labelled as {\it\textbf{lt}}, indicating solutions better than those achieved by Cplex alone.
For the 30 largest problems considered in 
Table~\ref{table200h} with $O=200$ orders, as summarised in 
Table~\ref{table_compare},  we have:
\begin{compactitem}
\item Cplex requires an average computation time of $T(s)=245.27$ seconds, with an associated average percentage deviation of 0\% 
\item Cplex requires an average computation time of $B(s)=95.17$ seconds, with an associated average percentage deviation of 0\%, \textbf{\emph{under the best case assumption that we have perfect foresight and can terminate Cplex as soon as it first encounters what would have been the final best solution if we had allowed Cplex to continue to run to completion/time limit}}.
\item We would note here however that this best case (so using $B(s)$) is unrealistic since (as we lack perfect foresight) we in practice, as in the computational results presented, terminate an algorithm  either when it terminates naturally, or when a 
pre-set computational time limit is reached. 
Using the more realistic $T(s)$ measure we have that SPB and PIO ($\tau=1$ and $\tau=2$) all have a lower average computation time than Cplex alone; have negative percentage deviations, so an overall reduction in rack usage; and too lead to a significant number of improved solutions (between 15 and 18 improved solutions out of the 30 instances for which $O=200$).
\item  SPB, requires an average computation time of 100.48 seconds, with an associated average percentage deviation 
of -6.95\%. It gives 16 improved solutions, 2 worse solutions and 12 equal solutions as compared with Cplex alone. So here, comparing SPB with Cplex alone, shows that even under the best case assumption of perfect foresight we can significantly improve upon Cplex alone (giving 16 improved solutions with an overall reduction in rack usage of 6.95\%) for only a slight increase in computation time from $B(s)=95.17$.  

\end{compactitem}

\noindent Considering Table~\ref{table_compare} and comparing SPB with PIO (for $\tau=1$) it is seems reasonable to conclude  that 
PIO is superior for $O=150$,  
SPB is superior for $O=200$, Obviously the balance between these two heuristics could change as we alter the computational time limit and/or use different problem instances. Often in the Operations Research literature the academic focus is  upon choosing 
the~\enquote{best heuristic} for a problem. As such then using Table~\ref{table_compare} we could if necessary choose a single heuristic (balancing quality of result with computation time) from those presented there.

However in this paper we are considering a problem motivated by a practical situation and it is obvious that operating a robotic mobile fulfilment system will be expensive. As such making use of a number of different solution approaches (in parallel) and taking the best solution found after a predefined (elapsed, wall clock) time limit might be a sensible operational strategy that would involve only a very small marginal cost. In other words the focus in a practical situation such as considered in this paper is no longer upon choosing one~\enquote{best} heuristic, but on having an array of heuristics/other approaches to use. 

As an illustration of this the column headed 
{\it\textbf{Parallel}} in Table~\ref{table_compare} 
shows the results we would obtain were
we to run, for 300 seconds, four approaches in parallel. Namely our order and rack allocation formulation using Cplex alone (as in Tables~\ref{table1}-\ref{table200}); SPB and PIO with both $\tau=1$  and $\tau=2$  (as in Tables~\ref{table3}-\ref{table200h}). The values shown are easily deduced from the results presented in those tables.

So for example for $O=200$ if we were to adopt this parallel approach then on termination after 300 seconds of elapsed time we would have (as compared to Cplex alone) better solutions for 20 of the 30 instances in Table~\ref{table200h} and an improvement in average percentage deviation from 0\% to -10.36\%. In other words a reduction in overall rack usage of 10.36\%.
 This is clearly superior to any of the approaches applied individually. 
Even if we were to exclude Cplex alone from this comparison (so just run our three SPB and PIO heuristics in parallel) then on termination after 300 seconds of elapsed time we would have (as compared to Cplex alone) better solutions for 20 of the 30 instances in Table~\ref{table200h} and an improvement in average percentage deviation from 0\% to -9.38\% (so a reduction in overall rack usage of 9.38\%).
This illustrates, in our view, the worth of our SPB and PIO heuristics since these results,  improving upon the results obtained from Cplex alone,
would not be achievable unless we had our heuristics to run in parallel.

\begin{table}[!htb]
\centering
{\tiny
\renewcommand{\tabcolsep}{1mm} 
\renewcommand{\arraystretch}{1.3} 
\begin{tabular}{|c|rrr|rrrrr|rrrrr|rrrrr|rrrr|}
\hline
\multicolumn{1}{|c|}{$O$} &
\multicolumn{3}{c|}{Cplex}
& \multicolumn{5}{c|}{SPB}
& \multicolumn{5}{c|}{PIO, $\tau = 1$} 
& \multicolumn{5}{c|}{PIO, $\tau = 2$} 
& \multicolumn{4}{c|}{Parallel}
\\
\multicolumn{1}{|c|}{ }
&  \multicolumn{1}{c}{T(s)} &
\multicolumn{1}{c}{B(s)} & 
\multicolumn{1}{c|}{\% dev} &
\multicolumn{1}{c}{T(s)} & \multicolumn{1}{c}{\% dev} 
& \multicolumn{3}{c|}{\# instances} &
\multicolumn{1}{c}{T(s)} & \multicolumn{1}{c}{\% dev} &
 \multicolumn{3}{c|}{\# instances} &
\multicolumn{1}{c}{T(s)} & \multicolumn{1}{c}{\% dev} 
& \multicolumn{3}{c|}{\# instances} 
& \multicolumn{1}{c}{\% dev} 
& \multicolumn{3}{c|}{\# instances} 
\\	& 	& 
&	&	& 	& lt &  gt &  eq & 	& & lt &  gt &  eq 
&		& & lt &  gt &  eq 
&  & lt &  gt &  eq 
\\
\hline
50	& 11.16	
& 7.48 
&	0	&1.23	& 3.63	& 0 & 9 &16 & 0.53	& 2.86  & 0 &13 &12 &	1.37	& 2.58  & 0 &9 &16 
& 0 & 0 &0& 25
\\
100	& 172.73  
&	65.87	
& 0	& 49.50	& 0.47 &4 &7 &14	& 17.31	& 0.02 & 5	&6 &14 & 120.35& 	-0.56 & 5 &7 &13
&  -2.65 & 5& 0 &20
\\
150	& 205.42 
&	46.54	
& 0	& 97.05	& -2.89 & 12 &3 &10	& 119.37	& -3.17	& 13 &4& 8 & 192.04& 2.35 & 9 &7 &9
& -6.61 & 15 &0 &10
\\
200 & 245.27 & 95.17 & 0 & 100.48 & -6.95 & 16 &2 &12 & 187.15 & -6.99 & 
18 &4 &8 & 212.88 & -3.35 & 15 &6 &9 & -10.36 & 20 &0 &10 \\
\hline

\end{tabular}
}
\caption{Heuristic results: summary comparison}
\label{table_compare}
\end{table}

\subsection{Picker rack sequencing}
\label{section75}

Recall that we treat the picker rack sequencing problem as a feasibility problem, so for a given set of racks we have to find a feasible rack sequence with which we can fulfil all of the orders assigned to the single picker, whilst having only $B$ bin positions available. It is clear that (in general, for a fixed number of orders) this problem will be more challenging the larger the number of racks assigned to the picker.

In order to investigate PRSF and PRSF1 we took the instances shown in 
Tables~\ref{table1}-\ref{table150} and solved them to proven optimality for $P=1$, so a single picker, but for a much larger range of $C_p$ values than considered in
Tables~\ref{table1}-\ref{table150}. 
Given the solution to an instance for $P=1$ and a given $C_p$ value we applied our picker rack sequencing formulations in order to find (if possible) a feasible rack sequence given $B$ bin positions.

\subsubsection{Rack sequencing with no revisits,  PRSF}

The results for PRSF for problems with $O=50,100,150$ orders are shown in Tables~\ref{table5a}-\ref{table150e}. These tables deal with the case where we assume
 that each  rack only visits the picker once (so no rack revisits to the picker are allowed).

In these tables the value for $K$ is the number of racks that have to be sequenced for each instance shown. The values shown in the body of these tables are the computation time (in seconds) required to solve the  problem. 
In these tables we give computation times to two decimal places, so a time of 0.00 means that the time required was less than 0.005 seconds.
A value in square brackets indicates that the tree search terminated (in the time shown) having proved that there was no feasible rack sequence given the orders and racks allocated to the single picker. If TL is shown then it indicates that the problem terminated at time limit (5 CPU minutes). In such cases we have not found a feasible rack sequence, but neither have we proved that the problem is infeasible, so its status remains open (i.e.~there may, or may not, be a feasible rack sequence). 

To clarify the meaning of these tables consider the problem $N=400$, $R=150$ and $C_p=90$ with $K=9$ in Table~\ref{table150e}. This problem involves $N=400$ products  and requires $K=9$ racks to be appropriately sequenced so that the $C_p=90$ orders allocated to the single picker can be fulfilled.
This value of $K=9$ was derived by solving to proven optimality our order and rack allocation formulation for a single picker required to deal with 90 orders and represents the minimum  number of racks that can be used to supply all of the products required for the orders allocated to the picker.

For this  problem $N=400$, $R=150$ and $C_p=90$  with $K=9$ in Table~\ref{table150e} then for $B{=}2$, so two bin positions, no feasible rack sequence exists, and this was proved in 4.75 seconds. For $B{=}3$  a feasible rack sequence was found in 0.44 seconds, for $B{=}4$ a feasible rack sequence was found in 0.91 seconds, etc. These times exclude the time taken to solve the initial problem with $P=1$ to find the orders and racks allocated to the single picker. 

Note  here that all of the problems shown in these tables were solved independently. By logic we know that if there is a feasible rack sequence for $B$ bin positions then there is also a feasible rack sequence for $B{+}1$ bin positions (simply leave the extra bin position unused). However we made no use of the solution for $B$ bin positions in solving for $B{+}1$ bin positions in the results shown.

Tables~\ref{table5a}-\ref{table150e} indicate that, as we might expect, PRSF generally becomes easier to solve as the number of bin positions $B$ increases, i.e.~it requires less time to find a feasible rack sequence as the number of bins increases.
  Taken together Tables~\ref{table5a}-\ref{table150e}  contain some 720 problems involving up to $N=500$ products, $O=150$ orders and $K=29$ racks to be sequenced. 

Examining these tables we believe that, for the problems considered, PRSF is very effective at finding feasible rack sequences. In particular note the low computation times seen for many of the 720 problems considered. Moreover  for all but a few instances  we found a feasible rack sequence with $B \geq 4$ where each  rack only visits the picker once, so no rack revisits to the picker were needed to achieve feasibility.

\begin{table}[tbp]
\centering
{\tiny
\renewcommand{\tabcolsep}{1mm} 
\renewcommand{\arraystretch}{1.3} 
\begin{tabular}{|c|c|c|rrrrrrrrr|}
\hline
Instance & $C_p$ & $K$ & \multicolumn{1}{c}{$B {=} 2$} & \multicolumn{1}{c}{$B {=} 3$} & \multicolumn{1}{c}{$B {=} 4$} & \multicolumn{1}{c}{$B {=} 5$} & \multicolumn{1}{c}{$B {=} 6$} & \multicolumn{1}{c}{$B {=} 7$} & \multicolumn{1}{c}{$B {=} 8$} & \multicolumn{1}{c}{$B {=} 9$} & \multicolumn{1}{c|}{$B {=} 10$}\\
\hline
\multirow{2}{*}{\begin{tabular}{l}$N = 100$\\$R = 50$\end{tabular}} 
& 20 & 2 &     0.00 &     0.00 &     0.00 &     0.00 &     0.00 &     0.00 &     0.00 &     0.00 &     0.00\\
 & 30 & 3 &     0.01 &     0.01 &     0.01 &     0.01 &     0.01 &     0.01 &     0.01 &     0.01 &     0.01\\
 & 40 & 4 &     0.02 &     0.01 &     0.01 &     0.01 &     0.01 &     0.01 &     0.01 &     0.01 &     0.01\\
 & 50 & 7 &     1.50 &     0.27 &     0.34 &     0.28 &     0.29 &     0.04 &     0.04 &     0.16 &     0.16\\
\hline
\multirow{2}{*}{\begin{tabular}{l}$N = 200$\\$R = 50$\end{tabular}} 
& 20 & 3 &     0.00 &     0.00 &     0.00 &     0.00 &     0.00 &     0.00 &     0.00 &     0.00 &     0.00\\
 & 30 & 5 &     0.01 &     0.01 &     0.01 &     0.01 &     0.01 &     0.02 &     0.01 &     0.01 &     0.02\\
 & 40 & 7 &     1.06 &     0.30 &     0.18 &     0.03 &     0.03 &     0.03 &     0.03 &     0.03 &     0.03\\
 & 50 & 11 &    25.69 &     3.06 &     2.57 &     1.05 &     0.80 &     0.82 &     0.81 &     0.82 &     0.81\\
\hline
\multirow{2}{*}{\begin{tabular}{l}$N = 300$\\$R = 50$\end{tabular}} 
& 20 & 4 &     0.01 &     0.01 &     0.01 &     0.01 &     0.01 &     0.01 &     0.01 &     0.01 &     0.01\\
 & 30 & 6 &     0.02 &     0.02 &     0.02 &     0.02 &     0.02 &     0.02 &     0.02 &     0.02 &     0.02\\
 & 40 & 9 &     2.47 &     0.87 &     0.34 &     0.33 &     0.34 &     0.34 &     0.34 &     0.33 &     0.33\\
 & 50 & 15 &     3.41 &     3.17 &     1.21 &     1.21 &     1.35 &     0.56 &     0.58 &     0.58 &     0.58\\
\hline
\multirow{2}{*}{\begin{tabular}{l}$N = 400$\\$R = 75$\end{tabular}} 
& 20 & 4 &     0.01 &     0.01 &     0.01 &     0.01 &     0.01 &     0.01 &     0.01 &     0.01 &     0.01\\
 & 30 & 7 &     0.02 &     0.02 &     0.02 &     0.02 &     0.02 &     0.02 &     0.02 &     0.02 &     0.02\\
 & 40 & 11 &     0.39 &     0.04 &     0.04 &     0.04 &     0.04 &     0.04 &     0.04 &     0.29 &     0.26\\
 & 50 & 19 &     5.38 &     3.09 &     1.43 &     1.42 &     1.57 &     1.57 &     1.57 &     0.79 &     0.79\\
\hline
\multirow{2}{*}{\begin{tabular}{l}$N = 500$\\$R = 75$\end{tabular}} 
& 20 & 4 &     0.01 &     0.01 &     0.01 &     0.01 &     0.01 &     0.01 &     0.01 &     0.01 &     0.01\\
 & 30 & 8 &     0.09 &     0.09 &     0.09 &     0.09 &     0.03 &     0.03 &     0.03 &     0.03 &     0.03\\
 & 40 & 12 &     1.13 &     0.88 &     0.47 &     0.27 &     0.28 &     0.28 &     0.28 &     0.47 &     0.28\\
 & 50 & 20 &    36.51 &    31.82 &     4.98 &     3.64 &     3.61 &     1.17 &     1.16 &     1.17 &     1.16\\
\hline
\multicolumn{3}{r}{\textbf{Average:}} & \multicolumn{1}{r}{\textbf{    3.89}} & \multicolumn{1}{r}{\textbf{    2.18}} & \multicolumn{1}{r}{\textbf{    0.59}} & \multicolumn{1}{r}{\textbf{    0.42}} & \multicolumn{1}{r}{\textbf{    0.42}} & \multicolumn{1}{r}{\textbf{    0.25}} & \multicolumn{1}{r}{\textbf{    0.25}} & \multicolumn{1}{r}{\textbf{    0.24}} & \multicolumn{1}{r}{\textbf{    0.23}}
\\
\end{tabular}
}
\caption{Results: picker rack sequencing, no revisits allowed, $O=50$ orders, PRSF }
\label{table5a}
\end{table}

\begin{table}[tbp]
\centering
{\tiny
\renewcommand{\tabcolsep}{1mm} 
\renewcommand{\arraystretch}{1.3} 
\begin{tabular}{|c|c|c|rrrrrrrrr|}
\hline
Instance & $C_p$ & $K$ & \multicolumn{1}{r}{$B {=} 2$} & \multicolumn{1}{r}{$B {=} 3$} & \multicolumn{1}{r}{$B {=} 4$} & \multicolumn{1}{r}{$B {=} 5$} & \multicolumn{1}{r}{$B {=} 6$} & \multicolumn{1}{r}{$B {=} 7$} & \multicolumn{1}{r}{$B {=} 8$} & \multicolumn{1}{r}{$B {=} 9$} & \multicolumn{1}{r|}{$B {=} 10$}\\
\hline
\multirow{2}{*}{\begin{tabular}{l}$N = 100$\\$R = 100$\end{tabular}} 
& 40 & 2 &     0.00 &     0.00 &     0.00 &     0.00 &     0.00 &     0.00 &     0.00 &     0.00 &     0.00\\
 & 50 & 2 &     0.00 &     0.00 &     0.00 &     0.00 &     0.00 &     0.00 &     0.00 &     0.00 &     0.00\\
 & 60 & 3 &     0.01 &     0.01 &     0.01 &     0.01 &     0.01 &     0.01 &     0.01 &     0.01 &     0.01\\
 & 70 & 4 &     0.05 &     0.02 &     0.02 &     0.02 &     0.02 &     0.02 &     0.02 &     0.02 &     0.02\\
 & 80 & 4 & [0.06] & [0.13] &     0.09 &     0.02 &     0.02 &     0.02 &     0.02 &     0.02 &     0.02\\
 & 90 & 5 & [0.13] & [0.41] &     0.50 &     0.34 &     0.19 &     0.03 &     0.03 &     0.03 &     0.03\\
 & 100 & 7 & [17.70] & [75.43] &     7.01 &     7.16 &     0.62 &     0.55 &     0.55 &     0.55 &     0.55\\
\hline
\multirow{2}{*}{\begin{tabular}{l}$N = 200$\\$R = 100$\end{tabular}} 
& 40 & 3 &     0.01 &     0.01 &     0.01 &     0.01 &     0.01 &     0.01 &     0.01 &     0.01 &     0.01\\
 & 50 & 4 & [0.05] &     0.01 &     0.01 &     0.01 &     0.01 &     0.01 &     0.01 &     0.01 &     0.01\\
 & 60 & 5 & [0.18] & [0.64] &     0.22 &     0.16 &     0.03 &     0.03 &     0.03 &     0.03 &     0.03\\
 & 70 & 6 & [0.37] & [1.34] &     0.93 &     0.44 &     0.38 &     0.39 &     0.30 &     0.29 &     0.32\\
 & 80 & 8 & [12.51] &     2.68 &     0.66 &     0.97 &     1.01 &     0.56 &     0.67 &     0.57 &     0.57\\
 & 90 & 10 &    16.90 &     8.94 &     6.30 &     3.03 &     1.66 &     1.89 &     1.56 &     1.82 &     1.55\\
 & 100 & 13 & TL & TL &    52.11 &    35.21 &    37.96 &    36.01 &     9.78 &    19.38 &     7.40\\
\hline
\multirow{2}{*}{\begin{tabular}{l}$N = 300$\\$R = 100$\end{tabular}} 
& 40 & 5 &     0.09 &     0.07 &     0.02 &     0.02 &     0.02 &     0.02 &     0.02 &     0.02 &     0.02\\
 & 50 & 6 & [0.79] &     0.35 &     0.16 &     0.03 &     0.03 &     0.03 &     0.03 &     0.03 &     0.03\\
 & 60 & 7 &     0.32 &     0.26 &     0.29 &     0.27 &     0.27 &     0.27 &     0.27 &     0.27 &     0.27\\
 & 70 & 9 &     5.95 &     2.66 &     1.93 &     0.36 &     0.36 &     0.66 &     0.66 &     0.66 &     0.66\\
 & 80 & 11 &    21.17 &     3.46 &     3.68 &     2.74 &     1.13 &     1.12 &     1.13 &     1.12 &     1.12\\
 & 90 & 13 & TL &    21.12 &    35.22 &    10.18 &     2.97 &     6.87 &     3.04 &     2.26 &     2.26\\
 & 100 & 18 & TL & TL & TL &    86.86 &    93.68 &   208.08 &    50.48 &    42.55 &    54.67\\
\hline
\multirow{2}{*}{\begin{tabular}{l}$N = 400$\\$R = 100$\end{tabular}} 
& 40 & 6 &     0.02 &     0.02 &     0.02 &     0.02 &     0.02 &     0.02 &     0.02 &     0.02 &     0.02\\
 & 50 & 7 & [2.08] &     0.45 &     0.23 &     0.04 &     0.04 &     0.04 &     0.04 &     0.04 &     0.04\\
 & 60 & 9 &     0.77 &     0.35 &     0.35 &     0.35 &     0.34 &     0.34 &     0.34 &     0.34 &     0.35\\
 & 70 & 11 &     5.32 &     3.58 &     2.05 &     0.69 &     0.68 &     0.70 &     0.70 &     0.70 &     0.71\\
 & 80 & 13 & TL & TL &   131.27 &     7.27 &     7.35 &     2.49 &     2.86 &     2.48 &     2.48\\
 & 90 & 17 & TL & TL & TL &    46.90 &    76.13 &    28.49 &    33.58 &    21.99 &    12.37\\
 & 100 & 22 & TL & TL & TL &   108.12 &   173.21 &    98.03 &    29.68 &    29.32 &    29.95\\
\hline
\multirow{2}{*}{\begin{tabular}{l}$N = 500$\\$R = 100$\end{tabular}} 
& 40 & 6 &     0.08 &     0.09 &     0.08 &     0.08 &     0.08 &     0.08 &     0.08 &     0.08 &     0.08\\
 & 50 & 8 &     0.25 &     0.22 &     0.16 &     0.16 &     0.16 &     0.15 &     0.20 &     0.15 &     0.15\\
 & 60 & 10 & [7.04] &     1.87 &     0.54 &     0.41 &     0.55 &     0.31 &     0.31 &     0.31 &     0.31\\
 & 70 & 13 & TL &   211.68 &     5.75 &     4.50 &     2.82 &     2.19 &     1.34 &     1.26 &     1.26\\
 & 80 & 16 & TL &    13.82 &    14.33 &     9.96 &     4.82 &     3.74 &     2.12 &     1.09 &     1.10\\
 & 90 & 20 & TL & TL &    57.88 &    62.29 &    75.56 &    53.75 &    28.37 &    38.53 &    27.28\\
 & 100 & 26 & TL & TL & TL &   188.06 &   163.21 &   101.29 &   195.32 &   125.53 &   116.30\\
\hline
\multicolumn{3}{r}{\textbf{Average:}} & \multicolumn{1}{r}{\textbf{   88.34}} & \multicolumn{1}{r}{\textbf{   69.99}} & \multicolumn{1}{r}{\textbf{   43.48}} & \multicolumn{1}{r}{\textbf{   16.48}} & \multicolumn{1}{r}{\textbf{   18.44}} & \multicolumn{1}{r}{\textbf{   15.66}} & \multicolumn{1}{r}{\textbf{   10.39}} & \multicolumn{1}{r}{\textbf{    8.33}} & \multicolumn{1}{r}{\textbf{    7.48}}\\
\end{tabular}
}
\caption{Results: picker rack sequencing, no revisits allowed, $O=100$ orders, PRSF}
\label{table6a}
\end{table}

\begin{table}[tbp]
\centering
{\tiny
\renewcommand{\tabcolsep}{1mm} 
\renewcommand{\arraystretch}{1.3} 
\begin{tabular}{|c|c|c|rrrrrrrrr|}
\hline
Instance & $C_p$ & $K$ & \multicolumn{1}{r}{$B {=} 2$} & \multicolumn{1}{r}{$B {=} 3$} & \multicolumn{1}{r}{$B {=} 4$} & \multicolumn{1}{r}{$B {=} 5$} & \multicolumn{1}{r}{$B {=} 6$} & \multicolumn{1}{r}{$B {=} 7$} & \multicolumn{1}{r}{$B {=} 8$} & \multicolumn{1}{r}{$B {=} 9$} & \multicolumn{1}{r|}{$B {=} 10$}\\
\hline
\multirow{2}{*}{\begin{tabular}{l}$N = 100$\\$R = 100$\end{tabular}} & 70 & 3 & [0.10] &     0.01 &     0.01 &     0.01 &     0.01 &     0.01 &     0.01 &     0.01 &     0.01\\
 & 90 & 4 & [0.13] & [0.62] &     0.69 &     0.69 &     0.43 &     0.03 &     0.03 &     0.03 &     0.03\\
 & 110 & 5 & [0.98] & [2.58] & [3.23] & [7.31] &     0.78 &     0.95 &     0.59 &     0.45 &     0.80\\
 & 130 & 7 & [16.94] & [66.75] &     9.63 &    13.20 &     8.57 &    12.62 &     4.67 &     6.13 &     2.79\\
 & 150 & 10 & TL & TL &   291.61 &    58.97 &   101.76 &    51.00 &    23.82 &    28.04 &    18.06\\
\hline
\multirow{2}{*}{\begin{tabular}{l}$N = 200$\\$R = 100$\end{tabular}} & 70 & 5 &     0.13 &     0.16 &     0.12 &     0.09 &     0.03 &     0.03 &     0.03 &     0.03 &     0.03\\
 & 90 & 6 & [0.22] & [1.30] &     0.87 &     1.00 &     0.54 &     0.39 &     0.36 &     0.40 &     0.22\\
 & 110 & 8 & [2.28] & [33.92] & [77.11] &     3.21 &     3.89 &     1.32 &     1.07 &     0.90 &     0.98\\
 & 130 & 11 & TL & TL & TL &    15.65 &    15.03 &    70.80 &    13.15 &    13.50 &     7.32\\
 & 150 & 16 & TL & TL & TL & TL &   248.77 &   192.10 &   253.36 &   124.41 &   106.64\\
\hline
\multirow{2}{*}{\begin{tabular}{l}$N = 300$\\$R = 100$\end{tabular}}  & 70 & 6 & [0.20] &     0.14 &     0.15 &     0.15 &     0.15 &     0.14 &     0.15 &     0.14 &     0.14\\
 & 90 & 9 & [51.07] &     2.58 &     2.87 &     0.58 &     0.57 &     0.54 &     0.55 &     0.55 &     0.55\\
 & 110 & 12 & TL & TL &    15.20 &   154.79 &    11.29 &     4.57 &     4.13 &     3.42 &     3.95\\
 & 130 & 15 & TL & TL & TL & TL & TL & TL &    55.54 &   117.03 &    34.59\\
 & 150 & 22 & TL & TL & TL & TL & TL & TL & TL & TL & TL\\
\hline
\multirow{2}{*}{\begin{tabular}{l}$N = 400$\\$R = 150$\end{tabular}} & 70 & 7 & [0.48] &     0.46 &     0.22 &     0.24 &     0.04 &     0.04 &     0.04 &     0.04 &     0.04\\
 & 90 & 9 & [4.75] &     0.44 &     0.91 &     0.43 &     0.43 &     0.43 &     0.30 &     0.30 &     0.29\\
 & 110 & 12 & TL &   215.15 &    10.27 &     5.19 &     6.34 &     4.56 &     3.93 &     3.78 &     3.77\\
 & 130 & 17 & TL & TL & TL & TL &    48.30 &    26.15 &    15.66 &    14.27 &    30.89\\
 & 150 & 24 & TL & TL & TL & TL & TL &   264.02 & TL & TL & TL\\
\hline
\multirow{2}{*}{\begin{tabular}{l}$N = 500$\\$R = 150$\end{tabular}} & 70 & 8 &     0.28 &     0.25 &     0.20 &     0.20 &     0.26 &     0.26 &     0.26 &     0.24 &     0.20\\
 & 90 & 11 & [194.23] &     1.50 &     0.66 &     0.99 &     0.63 &     0.62 &     0.64 &     0.64 &     0.64\\
 & 110 & 15 & TL & TL &    14.28 &     3.98 &     4.72 &     6.19 &     2.66 &     4.24 &     2.22\\
 & 130 & 20 & TL & TL & TL & TL & TL &    75.10 &    54.37 & TL &    44.70\\
 & 150 & 29 & TL & TL & TL & TL & TL & TL & TL & TL & TL\\\hline
\multicolumn{3}{r}{\textbf{Average:}} & \multicolumn{1}{r}{\textbf{  154.87}} & \multicolumn{1}{r}{\textbf{  145.03}} & \multicolumn{1}{r}{\textbf{  113.12}} & \multicolumn{1}{r}{\textbf{   94.67}} & \multicolumn{1}{r}{\textbf{   78.10}} & \multicolumn{1}{r}{\textbf{   64.47}} & \multicolumn{1}{r}{\textbf{   53.41}} & \multicolumn{1}{r}{\textbf{   60.74}} & \multicolumn{1}{r}{\textbf{   46.35}}
\\
\end{tabular}
}
\caption{Results: picker rack sequencing, no revisits allowed, $O=150$ orders, PRSF}
\label{table150e}
\end{table}

\subsubsection{Rack sequencing with no revisits, PRSF1}

The results for problems with $O=50,100,150$ orders are shown in Tables~\ref{table5}-\ref{table150org} respectively. These tables deal with the case where we assume
 that each  rack only visits the picker once (so no rack revisits to the picker are allowed)
and were produced using the (no revisit) rack sequencing formulation (PRSF1) given above, with Theorem~\ref{jebth2} applied.  It is clear that (on average) for the larger problems with $O=100,150$  PRSF1 requires less computation time than PRSF as $B$ increases.

As discussed above, in our view the rack sequencing problem is essentially a feasibility problem in that we assume that a primary optimisation has already been performed in terms of choosing which orders and racks are allocated to a picker. Hence finding a feasible rack sequence such that all orders can be picked using $B$ bin positions is the principal concern.
In terms of rack sequencing therefore we believe that a similar situation as for order and rack allocation applies, namely that it is worthwhile making use of a number of different solution approaches (in parallel) and stopping as soon as a feasible rack sequence  is found.

As an illustration of this parallel approach 
Table~\ref{tableseqcom} shows the results that would be obtained from PRSF and PRSF1 if they are run separately and if they are run in parallel (terminating as soon as a feasible rack sequence is found by either formulation, or at time limit). In general terms we can see that the parallel approach combines the best of both formulations, achieving feasible solutions (if possible) for $B$ small by making use of PRSF and solving more quickly as $B$ increases by making use of  PRSF1 (with Theorem~\ref{jebth2} applied)  which requires fewer variables/constraints.

\begin{table}[tbp]
\centering
{\tiny
\renewcommand{\tabcolsep}{1mm} 
\renewcommand{\arraystretch}{1.3} 
\begin{tabular}{|c|c|c|rrrrrrrrr|}
\hline
Instance & $C_p$ & $K$ & \multicolumn{1}{c}{$B {=} 2$} & \multicolumn{1}{c}{$B {=} 3$} & \multicolumn{1}{c}{$B {=} 4$} & \multicolumn{1}{c}{$B {=} 5$} & \multicolumn{1}{c}{$B {=} 6$} & \multicolumn{1}{c}{$B {=} 7$} & \multicolumn{1}{c}{$B {=} 8$} & \multicolumn{1}{c}{$B {=} 9$} & \multicolumn{1}{c|}{$B {=} 10$}\\
\hline
\multirow{2}{*}{\begin{tabular}{l}$N = 100$\\$R = 50$\end{tabular}} & 20 & 2 &     0.00 &     0.00 &     0.00 &     0.00 &     0.00 &     0.00 &     0.00 &     0.00 &     0.00\\
 & 30 & 3 & [0.01] &     0.00 &     0.00 &     0.00 &     0.00 &     0.00 &     0.00 &     0.00 &     0.00\\
 & 40 & 4 & [0.02] &     0.01 &     0.01 &     0.00 &     0.00 &     0.00 &     0.00 &     0.00 &     0.00\\
 & 50 & 7 & [1.65] &     0.78 &     0.45 &     0.18 &     0.14 &     0.17 &     0.17 &     0.17 &     0.17\\
\hline
\multirow{2}{*}{\begin{tabular}{l}$N = 200$\\$R = 50$\end{tabular}} & 20 & 3 & [0.00] &     0.00 &     0.00 &     0.00 &     0.00 &     0.00 &     0.00 &     0.00 &     0.00\\
 & 30 & 5 & [0.07] &     0.01 &     0.01 &     0.01 &     0.01 &     0.01 &     0.01 &     0.01 &     0.01\\
 & 40 & 7 & [0.24] &     0.11 &     0.17 &     0.09 &     0.10 &     0.01 &     0.01 &     0.01 &     0.01\\
 & 50 & 11 & [156.16] &     1.96 &     0.54 &     0.86 &     1.07 &     0.94 &     0.68 &     1.06 &     0.56\\
\hline
\multirow{2}{*}{\begin{tabular}{l}$N = 300$\\$R = 50$\end{tabular}} & 20 & 4 &     0.00 &     0.00 &     0.00 &     0.00 &     0.00 &     0.00 &     0.00 &     0.00 &     0.00\\
 & 30 & 6 & [0.09] &     0.01 &     0.01 &     0.01 &     0.01 &     0.01 &     0.01 &     0.01 &     0.01\\
 & 40 & 9 & [0.08] &     1.41 &     0.36 &     0.30 &     0.19 &     0.17 &     0.09 &     0.09 &     0.09\\
 & 50 & 15 & TL &     3.57 &     1.02 &     0.63 &     0.75 &     0.75 &     0.75 &     0.75 &     0.75\\
\hline
\multirow{2}{*}{\begin{tabular}{l}$N = 400$\\$R = 75$\end{tabular}} & 20 & 4 &     0.00 &     0.00 &     0.00 &     0.00 &     0.00 &     0.00 &     0.00 &     0.00 &     0.00\\
 & 30 & 7 &     0.01 &     0.01 &     0.01 &     0.01 &     0.01 &     0.01 &     0.01 &     0.01 &     0.01\\
 & 40 & 11 & [4.02] &     0.19 &     0.12 &     0.12 &     0.12 &     0.12 &     0.13 &     0.13 &     0.13\\
 & 50 & 19 &     2.98 &     0.89 &     0.75 &     0.81 &     0.29 &     0.29 &     0.92 &     0.31 &     0.29\\
\hline
\multirow{2}{*}{\begin{tabular}{l}$N = 500$\\$R = 75$\end{tabular}} & 20 & 4 &     0.00 &     0.00 &     0.00 &     0.00 &     0.00 &     0.00 &     0.00 &     0.00 &     0.00\\
 & 30 & 8 &     0.07 &     0.01 &     0.01 &     0.01 &     0.01 &     0.01 &     0.01 &     0.01 &     0.01\\
 & 40 & 12 & [75.25] &     0.82 &     0.42 &     0.22 &     0.22 &     0.13 &     0.13 &     0.13 &     0.14\\
 & 50 & 20 &     0.01 &     5.26 &     3.75 &     3.52 &     1.95 &     1.70 &     1.72 &     1.44 &     1.52\\
\hline
\multicolumn{3}{r}{\textbf{Average:}} & \multicolumn{1}{r}{\textbf{   27.03}} & \multicolumn{1}{r}{\textbf{    0.75}} & \multicolumn{1}{r}{\textbf{    0.38}} & \multicolumn{1}{r}{\textbf{    0.34}} & \multicolumn{1}{r}{\textbf{    0.24}} & \multicolumn{1}{r}{\textbf{    0.22}} & \multicolumn{1}{r}{\textbf{    0.23}} & \multicolumn{1}{r}{\textbf{    0.21}} & \multicolumn{1}{r}{\textbf{    0.19}}\\
\end{tabular}
}
\caption{Results: picker rack sequencing, no revisits allowed, $O=50$ orders, PRSF1}
\label{table5}
\end{table}

\begin{table}[tbp]
\centering
{\tiny
\renewcommand{\tabcolsep}{1mm} 
\renewcommand{\arraystretch}{1.3} 
\begin{tabular}{|c|c|c|rrrrrrrrr|}
\hline
Instance & $C_p$ & $K$ & \multicolumn{1}{r}{$B {=} 2$} & \multicolumn{1}{r}{$B {=} 3$} & \multicolumn{1}{r}{$B {=} 4$} & \multicolumn{1}{r}{$B {=} 5$} & \multicolumn{1}{r}{$B {=} 6$} & \multicolumn{1}{r}{$B {=} 7$} & \multicolumn{1}{r}{$B {=} 8$} & \multicolumn{1}{r}{$B {=} 9$} & \multicolumn{1}{r|}{$B {=} 10$}\\
\hline
\multirow{2}{*}{\begin{tabular}{l}$N = 100$\\$R = 100$\end{tabular}} & 40 & 2 &     0.00 &     0.00 &     0.00 &     0.00 &     0.00 &     0.00 &     0.00 &     0.00 &     0.00\\
 & 50 & 2 &     0.00 &     0.00 &     0.00 &     0.00 &     0.00 &     0.00 &     0.00 &     0.00 &     0.00\\
 & 60 & 3 & [0.00] &     0.00 &     0.00 &     0.00 &     0.00 &     0.00 &     0.00 &     0.00 &     0.00\\
 & 70 & 4 & [0.06] &     0.00 &     0.00 &     0.00 &     0.00 &     0.00 &     0.00 &     0.00 &     0.00\\
 & 80 & 4 & [0.02] & [0.02] & [0.06] &     0.03 &     0.03 &     0.01 &     0.01 &     0.01 &     0.01\\
 & 90 & 5 & [0.05] & [0.05] & [0.20] &     0.12 &     0.06 &     0.06 &     0.06 &     0.04 &     0.04\\
 & 100 & 7 & [0.63] & [4.41] & [49.28] &     2.48 &     0.69 &     0.34 &     0.28 &     0.27 &     0.45\\
\hline
\multirow{2}{*}{\begin{tabular}{l}$N = 200$\\$R = 100$\end{tabular}} & 40 & 3 & [0.01] &     0.00 &     0.00 &     0.00 &     0.00 &     0.00 &     0.00 &     0.00 &     0.00\\
 & 50 & 4 & [0.02] & [0.03] &     0.00 &     0.00 &     0.00 &     0.00 &     0.00 &     0.00 &     0.00\\
 & 60 & 5 & [0.05] & [0.15] & [0.37] &     0.06 &     0.12 &     0.01 &     0.01 &     0.01 &     0.01\\
 & 70 & 6 & [0.05] & [0.26] & [1.41] &     0.35 &     0.16 &     0.16 &     0.02 &     0.02 &     0.02\\
 & 80 & 8 & [0.66] & [17.50] &     0.98 &     0.51 &     0.33 &     0.46 &     0.41 &     0.32 &     0.31\\
 & 90 & 10 & [11.47] &    42.41 &    13.16 &     2.87 &     3.17 &     0.95 &     0.53 &     0.93 &     1.02\\
 & 100 & 13 & TL & TL & TL &    29.60 &    22.79 &     9.55 &     8.64 &     5.33 &     5.76\\
\hline
\multirow{2}{*}{\begin{tabular}{l}$N = 300$\\$R = 100$\end{tabular}} & 40 & 5 & [0.12] &     0.04 &     0.01 &     0.01 &     0.01 &     0.01 &     0.01 &     0.01 &     0.01\\
 & 50 & 6 & [0.25] & [0.72] &     0.10 &     0.05 &     0.05 &     0.05 &     0.05 &     0.05 &     0.05\\
 & 60 & 7 & [0.54] &     0.12 &     0.15 &     0.13 &     0.12 &     0.12 &     0.11 &     0.12 &     0.12\\
 & 70 & 9 & [6.69] &     3.61 &     1.77 &     0.51 &     0.56 &     0.36 &     0.44 &     0.44 &     0.41\\
 & 80 & 11 & [47.60] &   281.85 &     8.34 &     2.25 &     1.23 &     0.56 &     1.05 &     0.60 &     0.60\\
 & 90 & 13 & [51.26] & TL & TL &   136.61 &     2.46 &     3.10 &     2.60 &     1.93 &     1.61\\
 & 100 & 18 & TL & TL & TL &    57.75 &    30.93 &    24.78 &    24.70 &    12.40 &    16.47\\
\hline
\multirow{2}{*}{\begin{tabular}{l}$N = 400$\\$R = 100$\end{tabular}} & 40 & 6 &     0.01 &     0.01 &     0.01 &     0.01 &     0.01 &     0.01 &     0.01 &     0.01 &     0.01\\
 & 50 & 7 & [0.07] & [1.63] &     0.17 &     0.13 &     0.11 &     0.10 &     0.10 &     0.09 &     0.09\\
 & 60 & 9 & [0.19] &     0.36 &     0.20 &     0.18 &     0.12 &     0.12 &     0.12 &     0.12 &     0.12\\
 & 70 & 11 & [13.32] &    59.16 &     0.53 &     0.59 &     0.18 &     0.18 &     0.18 &     0.18 &     0.18\\
 & 80 & 13 & [30.12] & TL & TL &     1.58 &     2.62 &     1.33 &     1.57 &     1.29 &     1.29\\
 & 90 & 17 & TL & TL & TL & TL &   156.76 &    10.31 &    10.98 &    11.18 &     5.74\\
 & 100 & 22 & TL & TL &    86.36 &   235.77 &    35.47 &    46.89 &    20.23 &    22.28 &    36.15\\
\hline
\multirow{2}{*}{\begin{tabular}{l}$N = 500$\\$R = 100$\end{tabular}} & 40 & 6 & [0.06] &     0.04 &     0.03 &     0.03 &     0.03 &     0.03 &     0.03 &     0.03 &     0.03\\
 & 50 & 8 &     0.16 &     0.10 &     0.06 &     0.06 &     0.06 &     0.06 &     0.08 &     0.06 &     0.06\\
 & 60 & 10 & [1.68] & [18.89] &     0.62 &     0.42 &     0.35 &     0.22 &     0.23 &     0.34 &     0.21\\
 & 70 & 13 & [0.43] & TL &     3.96 &     1.43 &     0.66 &     0.65 &     0.68 &     0.68 &     0.68\\
 & 80 & 16 & [239.20] & TL &   108.15 &     7.43 &     2.66 &     2.33 &     1.99 &     2.90 &     2.99\\
 & 90 & 20 & TL & TL & TL &    15.25 &    15.54 &     5.38 &    10.90 &    14.16 &     7.56\\
 & 100 & 26 & TL & TL &   183.78 &   131.30 &    45.96 &    47.92 &    53.72 &    36.43 &    48.94\\
\hline
\multicolumn{3}{r}{\textbf{Average:}} & \multicolumn{1}{r}{\textbf{   62.99}} & \multicolumn{1}{r}{\textbf{   98.04}} & \multicolumn{1}{r}{\textbf{   64.56}} & \multicolumn{1}{r}{\textbf{   26.50}} & \multicolumn{1}{r}{\textbf{    9.24}} & \multicolumn{1}{r}{\textbf{    4.46}} & \multicolumn{1}{r}{\textbf{    3.99}} & \multicolumn{1}{r}{\textbf{    3.21}} & \multicolumn{1}{r}{\textbf{    3.74}}\\
\end{tabular}
}
\caption{Results: picker rack sequencing, no revisits allowed, $O=100$ orders, PRSF1}
\label{table6}
\end{table}

\begin{table}[tbp]
\centering
{\tiny
\renewcommand{\tabcolsep}{1mm} 
\renewcommand{\arraystretch}{1.3} 
\begin{tabular}{|c|c|c|rrrrrrrrr|}
\hline
Instance & $C_p$ & $K$ & \multicolumn{1}{r}{$B {=} 2$} & \multicolumn{1}{r}{$B {=} 3$} & \multicolumn{1}{r}{$B {=} 4$} & \multicolumn{1}{r}{$B {=} 5$} & \multicolumn{1}{r}{$B {=} 6$} & \multicolumn{1}{r}{$B {=} 7$} & \multicolumn{1}{r}{$B {=} 8$} & \multicolumn{1}{r}{$B {=} 9$} & \multicolumn{1}{r|}{$B {=} 10$}\\
\hline
\multirow{2}{*}{\begin{tabular}{l}$N = 100$\\$R = 100$\end{tabular}} & 70 & 3 & [0.02] & [0.07] &     0.01 &     0.01 &     0.01 &     0.01 &     0.01 &     0.01 &     0.01\\
 & 90 & 4 & [0.05] & [0.10] & [0.31] &     0.30 &     0.31 &     0.09 &     0.09 &     0.08 &     0.01\\
 & 110 & 5 & [0.23] & [0.45] & [1.14] & [1.95] & [3.54] &     0.82 &     0.56 &     0.37 &     0.28\\
 & 130 & 7 & [1.36] & [10.13] & [57.04] &     3.19 &     3.44 &    13.32 &     2.39 &     1.59 &     1.42\\
 & 150 & 10 & [209.73] & TL & TL & TL &    67.27 & TL &    94.23 &    14.48 &    16.19\\
\hline
\multirow{2}{*}{\begin{tabular}{l}$N = 200$\\$R = 100$\end{tabular}} & 70 & 5 & [0.03] &     0.07 &     0.04 &     0.01 &     0.01 &     0.01 &     0.01 &     0.01 &     0.01\\
 & 90 & 6 & [0.06] & [0.08] & [0.92] &     0.52 &     0.12 &     0.12 &     0.11 &     0.07 &     0.08\\
 & 110 & 8 & [0.37] & [3.64] & [15.59] & [179.80] &     1.11 &     1.16 &     0.50 &     1.17 &     0.41\\
 & 130 & 11 & [3.70] & TL & TL & TL &    91.51 &     5.45 &     6.46 &     6.33 &     5.97\\
 & 150 & 16 & TL & TL & TL & TL & TL &    77.24 &    62.47 &    55.01 &    76.32\\
\hline
\multirow{2}{*}{\begin{tabular}{l}$N = 300$\\$R = 100$\end{tabular}}  & 70 & 6 & [0.04] & [0.13] &     0.05 &     0.06 &     0.05 &     0.05 &     0.05 &     0.05 &     0.05\\
 & 90 & 9 & [0.65] & [25.06] &    19.96 &     0.25 &     0.33 &     0.27 &     0.25 &     0.25 &     0.25\\
 & 110 & 12 & [2.29] & TL & TL & TL &     5.09 &     3.82 &     2.65 &     2.31 &     2.08\\
 & 130 & 15 & [1.91] & TL & TL & TL & TL &   182.14 &   144.90 &    30.38 &   215.78\\
 & 150 & 22 & TL & TL & TL & TL & TL & TL & TL &   130.83 &   248.68\\
\hline
\multirow{2}{*}{\begin{tabular}{l}$N = 400$\\$R = 150$\end{tabular}} & 70 & 7 & [0.04] & [0.36] &     0.10 &     0.14 &     0.09 &     0.06 &     0.06 &     0.06 &     0.06\\
 & 90 & 9 & [0.22] & [7.11] &     1.10 &     0.27 &     0.31 &     0.21 &     0.20 &     0.21 &     0.11\\
 & 110 & 12 & [1.09] & TL &     5.39 &     1.37 &     1.60 &     2.45 &     1.59 &     1.50 &     1.34\\
 & 130 & 17 & [200.22] & TL &    33.96 &   174.41 &    16.52 &    13.70 &     8.34 &     7.16 &    12.81\\
 & 150 & 24 & TL & TL & TL & TL & TL &   141.38 &    97.72 &   119.13 &   107.00\\
\hline
\multirow{2}{*}{\begin{tabular}{l}$N = 500$\\$R = 150$\end{tabular}} & 70 & 8 & [1.89] &     0.19 &     0.12 &     0.09 &     0.09 &     0.09 &     0.09 &     0.09 &     0.09\\
 & 90 & 11 & [2.87] & [282.65] &     0.35 &     0.38 &     0.32 &     0.29 &     0.29 &     0.33 &     0.33\\
 & 110 & 15 & [9.42] & TL & TL &     4.15 &     6.90 &     1.44 &     2.14 &     2.13 &     1.72\\
 & 130 & 20 & [3.29] & TL & TL & TL & TL & TL &   240.58 &   216.74 &    30.76\\
 & 150 & 29 & TL & TL & TL & TL &   248.85 &   279.56 & TL & TL &   227.78\\
\hline
\multicolumn{3}{r}{\textbf{Average:}} & \multicolumn{1}{r}{\textbf{   65.58}} & \multicolumn{1}{r}{\textbf{  157.20}} & \multicolumn{1}{r}{\textbf{  125.44}} & \multicolumn{1}{r}{\textbf{  122.68}} & \multicolumn{1}{r}{\textbf{   77.90}} & \multicolumn{1}{r}{\textbf{   64.95}} & \multicolumn{1}{r}{\textbf{   50.63}} & \multicolumn{1}{r}{\textbf{   35.61}} & \multicolumn{1}{r}{\textbf{   37.98}}\\
\end{tabular}
}
\caption{Results: picker rack sequencing, no revisits allowed, $O=150$ orders, PRSF1}
\label{table150org}
\end{table}


\begin{table}[tbp]
\centering
{\tiny
\renewcommand{\tabcolsep}{1mm} 
\renewcommand{\arraystretch}{1.3} 
\begin{tabular}{|c|cc|rrrrrrrrr|}
\hline
& Formulation & &\multicolumn{1}{c}{$B {=} 2$} & \multicolumn{1}{c}{$B {=} 3$} & \multicolumn{1}{c}{$B {=} 4$} & \multicolumn{1}{c}{$B {=} 5$} & \multicolumn{1}{c}{$B {=} 6$} & \multicolumn{1}{c}{$B {=} 7$} & \multicolumn{1}{c}{$B {=} 8$} & \multicolumn{1}{c}{$B {=} 9$} & \multicolumn{1}{c|}{$B {=} 10$}\\
\thickhline
\multirow{6}{*}{$O=50$} 
& \multirow{2}{*}{ PRSF, Table~\ref{table5a}} 
 & Average time & 3.89	& 2.18	& 0.59	
& 0.42	& 0.42	& 0.25	& 0.25	
& 0.24	& 
0.23 \\
& & Number infeasible/TL &0/0 & 0/0 & 0/0 & 0/0 & 0/0 & 0/0 & 0/0 & 0/0 & 0/0 \\
\cline{2-12}
& \multirow{2}{*}{ PRSF1, Table~\ref{table5}} 
& Average time &  27.03	& 0.75	& 0.38
& 	0.34	& 0.24	& 0.22	& 0.23	
& 0.21	 & 0.19 \\
& & Number infeasible/TL &11/1 & 0/0 & 0/0 & 0/0 & 0/0 & 0/0 & 0/0 & 0/0 & 0/0 \\
\cline{2-12}

& \multirow{2}{*}{Both, Parallel} 
& Average time & 1.94	& 0.67	& 0.37 &	0.33	
& 0.22	& 0.16	& 0.18	& 0.17 &	0.16 \\
& & Number infeasible/TL &0/0 & 0/0 & 0/0 & 0/0 & 0/0 & 0/0 & 0/0 & 0/0 & 0/0 \\
\thickhline
\multirow{6}{*}{$O=100$} 
& \multirow{2}{*}{ PRSF, Table~\ref{table6a}} 
& Average time & 88.34	& 69.99	& 43.48	
& 16.48	& 18.44	& 15.66	& 10.39
& 	8.33 &	7.48 \\
& & Number infeasible/TL &10/10 & 5/7 & 0/4 & 0/0 & 0/0 & 0/0 & 0/0 & 0/0 & 0/0 \\
\cline{2-12}
& \multirow{2}{*}{ PRSF1, Table~\ref{table6}} 
& Average time & 62.99	& 98.04	& 64.56	& 26.50
& 	9.24	& 4.46	& 3.99	& 3.21	& 3.74 \\
& & Number infeasible/TL &25/6 & 10/10 & 5/6 & 0/1 & 0/0 & 0/0 & 0/0 & 0/0 & 0/0 \\
\cline{2-12}

& \multirow{2}{*}{ Both, Parallel} 
& Average time & 88.18	& 69.98	& 33.94	
& 11.99	& 6.87	& 4.45	& 3.99 &	3.15 &	3.51 \\
& & Number infeasible/TL &10/10 & 5/7 & 0/2 & 0/0 & 0/0 & 0/0 & 0/0 & 0/0 & 0/0 \\
\thickhline
\multirow{6}{*}{$O=150$} 
& \multirow{2}{*}{ PRSF, Table~\ref{table150e}} 
& Average time & 154.87	& 145.03 &	113.12 &	94.67 &
	78.10	 &64.47	 &53.41	& 60.74	& 46.35 \\
& & Number infeasible/TL &11/12 & 5/11 & 2/8 & 1/7 & 0/5 & 0/3 & 0/3 & 0/4 & 0/3 \\
\cline{2-12}
& \multirow{2}{*}{ PRSF1, Table~\ref{table150org}} 
& Average time & 65.58	&157.20	&125.44	&122.68	&77.90	&64.95 &	50.63  &	35.61	 &37.98 \\
& & Number infeasible/TL &21/4 & 11/12 & 5/10 & 2/9 & 1/5 & 0/3 & 0/2 & 0/1 & 0/0 \\
\cline{2-12}

& \multirow{2}{*}{ Both, Parallel} 
& Average time & 154.87&	145.03	 &102.26 &	89.00	 &72.59	 &45.96	 &36.79 &	35.60	 &30.73 \\
& & Number infeasible/TL & 11/12 & 5/11 & 2/7 & 1/6 & 0/4 & 0/1 & 0/2 & 0/1 & 0/0 \\

\thickhline
\end{tabular}
}
\caption{Picker rack sequencing comparison }
\label{tableseqcom}
\end{table}

\subsubsection{Rack sequencing with revisits}

In order to investigate the problem of rack sequencing when rack revisits are allowed we took all of the problems with $O=100$ in 
Table~\ref{table6} for which either we proved (within our 300 second time limit, using PRSF1) that there was no feasible rack sequence (involving no revisits) or terminated at time limit without resolving the situation. Here termination at time limit means that we  have not found a feasible (no revisit) rack sequence, but neither have we proved that there is no possible (no revisit)  rack sequence.
We then investigated these problems with varying values of  $M$ (the maximum number of rack revisits allowed). 

Our results for applying PRSF1 with $M = 0,1,2,3$ revisits allowed, with Theorem~\ref{jebth2} applied, are shown in Table~\ref{table7}. 
All infeasible or unsolved instances from Table~\ref{table6} are given in Table~\ref{table7}, with the exception (due to space reasons) of $O = 100$, $N = 400$, $R = 100$ and $B = 5$, which was the only unresolved instance with $B \geq 5$. For all $M = 0,1,2,3$ it remained unresolved (having reached the time limit). 

There were originally 40 instances,
with times shown in square brackets in 
Table~\ref{table7},
 for which we proved, with $M = 0$, that no feasible rack sequencing solution existed.
Out of these 40, when $M = 1$ we were able to find 28 feasible solutions. Out of the remaining 12, we found another 6 with $M = 2$, and a further 3 with $M = 3$. Overall in only 3 of these 40 instances were we not able to find any feasible solution.
There were also 22 unresolved instances with $M=0$ (unresolved due to time limit). We found feasible solutions with $M = 1$ for 5 of these,  with $M = 2$ we found another 5 feasible solutions,  and with $M=3$ a further solution, leaving 11 unresolved instances.

Overall, of the 62 infeasible or unresolved instances in
Table~\ref{table7},
 we were able to find feasible rack sequences for 48 instances by allowing up to $3$ revisits.
In the light of the results in  Table~\ref{table7} we believe that there are 
potential benefits from running different values of $M$ in parallel and taking the feasible solution (if any are found) with the smallest $M$.

On  a technical note here when we allow rack revisits although all of the original racks (and their $M$ revisit copies) will be sequenced as a result of solving our formulation there is no constraint requiring a rack to supply any product to any order when presented to the picker. As a result  we automatically  know that if there is a feasible rack sequence for $M$ revisits then there is also a feasible rack sequence for $M{+}1$ revisits
(every rack in the $M$ revisit solution revisits one more time, but supplies no product). However we made no use of the solution for $M$ revisits in solving for $M{+}1$ revisits in the results shown.

Table~\ref{table7e} shows the results with PRSF for the same problems as considered in 
Table~\ref{table7}.
There were originally 15 instances
 in 
Table~\ref{table7e},
 for which we proved, with $M = 0$, that no feasible rack sequencing solution existed. When $M = 1$ we were able to find  feasible solutions for all of these instances. 
There were also 19 unresolved instances with $M=0$ (unresolved due to time limit). 18 of these instances were still unresolved with $M=1,2,3$. This is an indication of the additional computational effort needed to resolve PRSF as it involves more variables/constraints (as discussed previously above).

As before for rack sequencing with no revisits, parallel execution of both PRSF and PRSF1 yields benefits for rack sequencing with revisits. In Table~\ref{table7} the average computation time is 147.1 seconds with 85 instances going to time limit and in Table~\ref{table7e} the average computation time is 143.2 seconds with 103 instances going to time limit. With parallel running the average (elapsed) computation time is   113.7 seconds with 74 instances going to time limit.


\begin{table}[!htb]
\centering
{\tiny
\renewcommand{\tabcolsep}{1mm} 
\renewcommand{\arraystretch}{1.3} 
\begin{tabular}{|c|c|c|rrrr|rrrr|rrrr|}
\hline
  \multirow{2}{*}{$N$} & \multirow{2}{*}{$R$} & \multirow{2}{*}{$C_p$} & \multicolumn{4}{c|}{$B {=} 2$} & \multicolumn{4}{c|}{$B {=} 3$} & \multicolumn{4}{c|}{$B {=} 4$}\\
\cline{4-15}
 &  &  & \multicolumn{1}{c}{$M {=} 0$} & \multicolumn{1}{c}{$M {=} 1$} & \multicolumn{1}{c}{$M {=} 2$} & \multicolumn{1}{c|}{$M {=} 3$} & \multicolumn{1}{c}{$M {=} 0$} & \multicolumn{1}{c}{$M {=} 1$} & \multicolumn{1}{c}{$M {=} 2$} & \multicolumn{1}{c|}{$M {=} 3$} & \multicolumn{1}{c}{$M {=} 0$} & \multicolumn{1}{c}{$M {=} 1$} & \multicolumn{1}{c}{$M {=} 2$} & \multicolumn{1}{c|}{$M {=} 3$}\\

\hline
100 & 100 & 60 & [0.00] &     0.01 &     0.03 &     0.32  &  &  &  &  &  &  &  & \\
 & & 70 & [0.06] &     0.14 &     0.31 &     0.77 &  &  &  &  &  &  &  & \\
 & & 80 & [0.02] &     7.80 &     6.74 &    12.10 & [0.02] &     0.39 &     0.63 &     1.48 & [0.06] &     0.47 &     0.45 &     1.65\\
 & & 90 & [0.05] & [158.17] & TL &   211.62 & [0.05] &     1.75 &     2.69 &     2.75 & [0.20] &     0.98 &     2.28 &     4.65\\
 & & 100 & [0.63] & TL & TL & TL & [4.41] &    20.92 &    26.12 &    48.00 & [49.28] &    13.31 &    85.64 &    54.04\\
\hline
 200 & 100 & 40 &  [0.01] &     0.01 &     0.02 &     0.52 &  &  &  &  &  &  &  & \\
 & & 50 & [0.02] &     0.88 &     0.31 &     1.41 & [0.03] &     0.51 &     0.43 &     0.76 &  &  &  & \\
 & & 60 & [0.05] &    29.32 &     2.48 &     7.89 & [0.15] &     1.72 &     1.34 &     2.50 & [0.37] &     0.95 &     1.79 &     1.59\\
 & & 70 & [0.05] & TL &   172.31 &   143.12 & [0.26] &    15.80 &     7.62 &    12.72 & [1.41] &     4.18 &     7.70 &     7.47\\
 & & 80 & [0.66] & TL &   184.64 &   152.08 & [17.50] &    24.46 &    23.64 &    83.09 &  &  &  & \\
 & & 90 & [11.47] & TL & TL & TL &  &  &  &  &  &  &  & \\
 & & 100 & TL & TL & TL & TL & TL & TL & TL & TL & TL & TL &   263.09 & TL\\
\hline
 300 & 100 & 40 & [0.12] &     0.27 &     0.46 &     0.94 &  &  &  &  &  &  &  & \\
 & & 50 & [0.25] & TL &    12.48 &     2.94 & [0.72] &     0.87 &     1.31 &     2.47 &  &  &  & \\
 & & 60 & [0.54] &     2.70 &    10.23 &    13.53 &  &  &  &  &  &  &  & \\
 & & 70 & [6.69] &    45.27 &    25.17 &    11.52 &  &  &  &  &  &  &  & \\
 & & 80 & [47.60] & TL &   126.73 & TL &  &  &  &  &  &  &  & \\
 & & 90 & [51.26] & TL & TL &   298.58 & TL &    97.85 &   132.93 &    62.72 & TL &    51.18 &    56.75 &    69.53\\
 & & 100 & TL & TL & TL & TL & TL & TL & TL & TL & TL & TL & TL & TL\\
\hline
 400 & 100 & 50 & [0.07] &     6.59 &     8.01 &    10.42 & [1.63] &     2.39 &     4.83 &     3.35 &  &  &  & \\
 & & 60 & [0.19] &    14.13 &     9.25 &    23.59 &  &  &  &  &  &  &  & \\
 & & 70 & [13.32] &    27.95 &    25.77 &    15.54 &  &  &  &  &  &  &  & \\
 & & 80 & [30.12] & TL & TL & TL & TL & TL &    46.91 &    28.84 & TL &    29.50 &    17.58 &    23.22\\
 & & 90 & TL & TL & TL & TL & TL & TL &   144.53 &   214.35 & TL & TL &    56.41 & TL\\
 & & 100 & TL & TL & TL & TL & TL & TL & TL & TL &  &  &  & \\
\hline
 500 & 100 & 40 & [0.06] &     0.99 &     0.74 &     1.10 &  &  &  &  &  &  &  & \\
 & & 60 & [1.68] & TL &    33.38 &     9.79 & [18.89] &     8.19 &     6.32 &    12.95 &  &  &  & \\
 & & 70 & [0.43] & TL &   145.51 &   105.80 & TL &    25.23 &    45.93 &    31.24 &  &  &  & \\
 & & 80 & [239.20] & TL & TL &   238.59 & TL &    56.95 &    39.52 &    98.69 &  &  &  & \\
 & & 90 & TL & TL & TL & TL & TL & TL & TL &   178.83 & TL & TL &   200.37 &   295.54\\
 & & 100 & TL & TL & TL & TL & TL & TL & TL & TL &  &  &  & \\
\hline
\end{tabular}
}

\caption{Results: picker rack sequencing with revisits allowed, $O=100$, PRSF1}
\label{table7}
\end{table}

\begin{table}[tbp]
\centering
{\tiny
\renewcommand{\tabcolsep}{1mm} 
\renewcommand{\arraystretch}{1.3} 
\begin{tabular}{|c|c|c|rrrr|rrrr|rrrr|}
\hline
\multirow{2}{*}{$N$} & \multirow{2}{*}{$R$} & \multirow{2}{*}{$C_p$} & \multicolumn{4}{c|}{$B {=} 2$} & \multicolumn{4}{c|}{$B {=} 3$} & \multicolumn{4}{c|}{$B {=} 4$}\\
\cline{4-15}
 &  &  & \multicolumn{1}{c}{$M {=} 0$} & \multicolumn{1}{c}{$M {=} 1$} & \multicolumn{1}{c}{$M {=} 2$} & \multicolumn{1}{c|}{$M {=} 3$} & \multicolumn{1}{c}{$M {=} 0$} & \multicolumn{1}{c}{$M {=} 1$} & \multicolumn{1}{c}{$M {=} 2$} & \multicolumn{1}{c|}{$M {=} 3$} & \multicolumn{1}{c}{$M {=} 0$} & \multicolumn{1}{c}{$M {=} 1$} & \multicolumn{1}{c}{$M {=} 2$} & \multicolumn{1}{c|}{$M {=} 3$}\\
\hline
100 & 100 & 60 &     0.01 &     0.06 &     0.11 &     0.20 &  &  &  &  &  &  &  & \\
& & 70 &     0.05 &     0.12 &     0.22 &     0.39 &  &  &  &  &  &  &  & \\
& & 80 & [0.06] &     3.00 &     3.98 &    21.63 & [0.13] &     1.43 &     5.66 &     8.65 &     0.09 &     1.81 &     5.11 &     6.47\\
  & & 90 & [0.13] &     9.62 &    57.87 &    32.80 & [0.41] &     7.64 &    11.45 &    21.57 &     0.50 &     7.60 &     5.90 &    13.88\\
 &  & 100 & [17.70] &   210.73 &   147.46 & TL & [75.43] &    64.63 &   111.92 & TL &     7.01 &    48.85 &    71.13 &   148.50\\
\hline
 200 & 100 
& 40 &     0.01 &     0.05 &     0.70 &     1.04 &     &     &     &      &      &      &      &     \\
 & & 50 & [0.05] &     2.29 &     0.17 &     2.65 &     0.01 &     0.71 &     0.17 &     2.80 &  &  &  & \\
  & & 60 & [0.18] &    16.75 &     6.67 &    10.82 & [0.64] &     2.49 &     3.42 &     8.52 &     0.22 &     0.16 &     5.62 &     5.98\\
 &  & 70 & [0.37] &    40.55 &    92.07 &    33.58 & [1.34] &    22.28 &    34.62 &    34.17 &     0.93 &     7.49 &    48.04 &    30.73\\

 &  & 80 & [12.51] &   124.02 &   256.84 & TL &     2.68 &    15.53 &   121.13 &   285.23 &  &  &  & \\
& & 90 &    16.90 & TL & TL & TL &     &   &  &  &      &     &  & \\
  & & 100 & TL & TL & TL & TL & TL & TL & TL & TL &    52.11 & TL & TL & TL\\
\hline
 300 & 100
 & 40 &     0.09 &     0.79 &     1.75 &     0.35 &  &  &  &  &  &  &  & \\

 & & 50 & [0.79] &    12.65 &     5.58 &     6.47 &     0.35 &     2.91 &     6.99 &     8.81 &  &  &  & \\
 && 60 &     0.32 &    31.18 &    11.91 &    15.58 &  &  &  &  &  &  &  & \\     
& & 70 &     5.95 &    61.60 &   225.85 &    89.42 &  &  &  &  &  &  &  & \\     
     
 &  & 80 &    21.17 & TL & TL & TL &  &  &  &  &  &  &  & \\

  & & 90 & TL & TL & TL & TL &    21.12 & TL & TL & TL &    36.36 & TL & TL & TL\\

 &  & 100 & TL & TL & TL & TL & TL & TL & TL & TL & TL & TL & TL & TL\\
\hline
 400 & 100 
 & 50 & [2.08] &     4.70 &    11.38 &    21.73 &     0.45 &     4.28 &     5.67 &    10.83 &    &     &     &     \\
  && 60 &     0.77 &    27.38 &    33.40 &    32.97 &  &  &  &  &  &  &  & \\     
  & & 70 &     5.32 &    85.99 &    85.40 & TL &  &  &  &  &  &  &  & \\     

& & 80 & TL & TL & TL & TL & TL &   149.04 & TL & TL &   131.27 &   126.27 & TL & TL\\
& & 90 & TL & TL & TL & TL & TL & TL & TL & TL & TL & TL & TL & TL\\
& & 100 & TL & TL & TL & TL & TL & TL & TL & TL &  &  &  & \\

\hline
 500 & 100 
& 40 &     0.08 &     2.11 &     2.31 &     3.27  &  &  &  &  &  &  &  & \\

 & & 60 & [7.04] &    22.51 &    56.37 &    33.52 &     1.87 &    16.16 &    14.93 &    30.58 &  &  &  & \\

 & & 70 & TL & TL & TL & TL &   211.68 &    97.92 & TL & TL &  &  &  & \\
& & 80 & TL & TL & TL & TL &    13.82 &   285.38 & TL & TL &    &   &  & \\
&  & 90 & TL & TL & TL & TL & TL & TL & TL & TL &    57.88 & TL & TL & TL\\
& & 100 & TL & TL & TL & TL & TL & TL & TL & TL &  & &  & \\

 \hline
\end{tabular}
}

\caption{Results: picker rack sequencing with revisits allowed, $O=100$, PRSF}
\label{table7e}
\end{table}


\subsection{Miniload instances}

We commented above that the work presented in this paper, both for order and rack allocation, and rack sequencing, applies to other automated fulfilment systems. In particular to miniload systems~(\cite{bozer18, fusser19}) where conveyors (rather than mobile robots) bring racks of product to pickers. The major difference between using mobile robots and a miniload system for picking customer orders in the B2C market relates to the nature of the racks presented to the picker. With mobile robots these racks typically contain many different products, but with only a few units of each product. In miniload systems the racks (more commonly referred to as trays)   that are brought to the pickers by conveyors will typically contain just a  few products, but many units of each product.

Although the predominant focus of this paper has been on mobile robots we did consider a limited number of problem instances with characteristics more appropriate to a miniload environment
in order to provide some insight into how the work presented in this paper performs computationally in a miniload environment.

We took the set of $O=50$ orders as in Table~\ref{table1}, but amended the racks such that each rack now contains only four distinct products, but with sufficient in each rack to supply all orders for those products, so $s_{ir} =  \sum_{o=1}^O  q_{io}$ if rack $r$ contains product $i$,  zero otherwise.  The results for order and rack allocation for these instances with miniload characteristics are shown in Table~\ref{mini50}. This table has the same format as Table~\ref{table1}. The results for rack sequencing for these miniload instances using PRSF are shown in Table~\ref{mini50seq}. This table has the same format as Table~\ref{table5}.

With respect to Table~\ref{mini50} we can see that for all instances with one picker our order and rack formulation solves the problem optimally very quickly. Although the average computation time for these instances is higher than for the instances (albeit with different characteristics)  seen in Table~\ref{table1} all but three of 35 problems in  Table~\ref{mini50} are solved to proven optimality within our 300 second time limit. 
With respect to Table~\ref{mini50seq} we  proved that there was no feasible rack sequence, or found a feasible rack sequence,  for all of the 180 instances in Table~\ref{mini50seq} within our 300 second time limit.

\begin{table}[tbp]
\centering
{\tiny
\renewcommand{\tabcolsep}{1mm} 
\renewcommand{\arraystretch}{1.4} 
\begin{tabular}{|c|c|c|rrrrrrrr|}
\hline
Instance & $P$ & $C_p$ & \multicolumn{1}{c}{T(s)} & \multicolumn{1}{c}{B(s)} & \multicolumn{1}{c}{UB} & \multicolumn{1}{c}{GAP} & \multicolumn{1}{c}{LB} & \multicolumn{1}{c}{FGAP} & \multicolumn{1}{c}{FLB} & \multicolumn{1}{c|}{NS}\\
\hline
\multirow{2}{*}{\begin{tabular}{l}$N = 100$\\$R = 50$\end{tabular}} 
& 1 & 25 &     0.01 &     0.01 & 5 & -- & UB &     0.00 &     5.00 & 1\\
 &  & 50 &     0.00 &     0.00 & 16 & -- & UB &     0.00 &    16.00 & 1\\
\cline{2-11}
 & 5 & 5 &     7.02 &     0.66 & 7 & -- & UB &    28.57 &     5.00 & 2298\\
 &  & 7 &    17.72 &     2.03 & 10 & -- & UB &    24.78 &     7.52 & 3109\\
 &  & 10 &     2.76 &     2.64 & 17 & -- & UB &     5.88 &    16.00 & 367\\
\cline{2-11}
 & 10 & 2 &     0.14 &     0.13 & 10 & -- & UB &     0.00 &    10.00 & 1\\
 &  & 5 &   295.08 &    80.03 & 20 & -- & UB &    20.00 &    16.00 & 20762\\
\cline{1-11}
\multirow{2}{*}{\begin{tabular}{l}$N = 200$\\$R = 75$\end{tabular}} 
& 1 & 25 &     0.01 &     0.01 & 6 & -- & UB &     0.00 &     6.00 & 1\\
 &  & 50 &     0.02 &     0.01 & 17 & -- & UB &     0.00 &    17.00 & 1\\
\cline{2-11}
 & 5 & 5 &   241.62 &     0.63 & 8 & -- & UB &    31.25 &     5.50 & 119250\\
 &  & 7 &     3.40 &     1.05 & 10 & -- & UB &    17.53 &     8.25 & 750\\
 &  & 10 &     4.91 &     1.85 & 18 & -- & UB &     9.51 &    16.29 & 990\\
\cline{2-11}
 & 10 & 2 &     0.13 &     0.07 & 10 & -- & UB &     0.00 &    10.00 & 1\\
 &  & 5 &   151.79 &   117.56 & 20 & -- & UB &    18.55 &    16.29 & 9396\\
\cline{1-11}
\multirow{2}{*}{\begin{tabular}{l}$N = 300$\\$R = 75$\end{tabular}} 
& 1 & 25 &     0.01 &     0.00 & 5 & -- & UB &     0.00 &     5.00 & 1\\
 &  & 50 &     0.01 &     0.00 & 16 & -- & UB &     0.00 &    16.00 & 1\\
\cline{2-11}
 & 5 & 5 &     3.57 &     0.91 & 6 & -- & UB &    16.67 &     5.00 & 1028\\
 &  & 7 &     2.48 &     1.99 & 9 & -- & UB &    15.89 &     7.57 & 339\\
 &  & 10 &     1.73 &     1.73 & 16 & -- & UB &     2.50 &    15.60 & 124\\
\cline{2-11}
 & 10 & 2 &     0.09 &     0.06 & 10 & -- & UB &     0.00 &    10.00 & 1\\
 &  & 5 &    66.81 &    66.81 & 17 & -- & UB &     8.82 &    15.50 & 5643\\
\cline{1-11}
\multirow{2}{*}{\begin{tabular}{l}$N = 400$\\$R = 75$\end{tabular}} & 1 & 25 &     0.05 &     0.01 & 7 & -- & UB &     0.00 &     7.00 & 1\\
 &  & 50 &     0.01 &     0.00 & 20 & -- & UB &     0.00 &    20.00 & 1\\
\cline{2-11}
 & 5 & 5 & TL &     0.63 & 10 &    15.78 &     8.42 &    31.11 &     6.89 & 186649\\
 &  & 7 &    13.41 &    13.20 & 12 & -- & UB &    15.42 &    10.15 & 2689\\
 &  & 10 &     9.91 &     9.91 & 20 & -- & UB &     0.00 &    20.00 & 789\\
\cline{2-11}
 & 10 & 2 &     0.20 &     0.06 & 10 & -- & UB &     0.00 &    10.00 & 1\\
 &  & 5 & TL &    81.69 & 23 &     5.27 &    21.79 &    13.04 &    20.00 & 22635\\
\cline{1-11}
\multirow{2}{*}{\begin{tabular}{l}$N = 500$\\$R = 75$\end{tabular}} & 1 & 25 &     0.04 &     0.01 & 7 & -- & UB &     0.00 &     7.00 & 1\\
 &  & 50 &     0.01 &     0.00 & 19 & -- & UB &     0.00 &    19.00 & 1\\
\cline{2-11}
 & 5 & 5 & TL &     0.46 & 10 &    22.93 &     7.71 &    35.45 &     6.46 & 229648\\
 &  & 7 &    23.30 &     6.66 & 12 & -- & UB &    18.86 &     9.74 & 5094\\
 &  & 10 &     6.91 &     6.91 & 19 & -- & UB &     3.57 &    18.32 & 1221\\
\cline{2-11}
 & 10 & 2 &     0.15 &     0.06 & 10 & -- & UB &     0.00 &    10.00 & 1\\
 &  & 5 &    35.02 &    11.09 & 22 & -- & UB &    15.94 &    18.49 & 2050\\
\hline
\multicolumn{3}{r}{\textbf{Average:}} & \multicolumn{1}{r}{\textbf{   51.09}} & \multicolumn{1}{r}{\textbf{   11.68}} & \multicolumn{1}{c}{} & \multicolumn{1}{r}{\textbf{    1.68}} & \multicolumn{1}{c}{} & \multicolumn{1}{r}{\textbf{    9.52}} & \multicolumn{1}{c}{} & \multicolumn{1}{r}{\textbf{17567}}\end{tabular}
}
\caption{Results: order and rack allocation, miniload instances, $O=50$ orders}
\label{mini50}
\end{table}

\begin{table}[tbp]
\centering
{\tiny
\renewcommand{\tabcolsep}{1mm} 
\renewcommand{\arraystretch}{1.3} 
\begin{tabular}{|c|c|c|rrrrrrrrr|}
\hline
\multicolumn{1}{|c|}{Instance} & \multicolumn{1}{|c|}{$C_p$} & \multicolumn{1}{|c|}{$K$} & \multicolumn{1}{c}{$B{=}2$} & \multicolumn{1}{c}{$B{=}3$} & \multicolumn{1}{c}{$B{=}4$} & \multicolumn{1}{c}{$B{=}5$} & \multicolumn{1}{c}{$B{=}6$} & \multicolumn{1}{c}{$B{=}7$} & \multicolumn{1}{c}{$B{=}8$} & \multicolumn{1}{c}{$B{=}9$} & \multicolumn{1}{c|}{$B{=}10$}\\
\hline
\multirow{2}{*}{\begin{tabular}{l}$N = 100$\\$R = 50$\end{tabular}} 
& 20 & 4 &     0.01 &     0.01 &     0.01 &     0.01 &     0.01 &     0.01 &     0.01 &     0.01 &     0.01\\
 & 30 & 7 & [0.32] &     0.14 &     0.03 &     0.03 &     0.03 &     0.03 &     0.03 &     0.03 &     0.03\\
 & 40 & 10 & [19.32] &     2.52 &     1.59 &     0.66 &     0.26 &     0.26 &     0.26 &     0.26 &     0.26\\
 & 50 & 16 &    19.45 &     5.38 &     4.31 &     3.98 &     3.60 &     3.60 &     1.45 &     1.43 &     1.43\\
\hline
\multirow{2}{*}{\begin{tabular}{l}$N = 200$\\$R = 75$\end{tabular}} 
& 20 & 5 &     0.01 &     0.01 &     0.01 &     0.01 &     0.01 &     0.01 &     0.01 &     0.01 &     0.01\\
 & 30 & 7 &     0.10 &     0.10 &     0.02 &     0.02 &     0.02 &     0.02 &     0.02 &     0.02 &     0.02\\
 & 40 & 11 &     0.39 &     0.24 &     0.30 &     0.32 &     0.31 &     0.31 &     0.31 &     0.24 &     0.33\\
 & 50 & 17 &     7.42 &     2.61 &     2.00 &     1.41 &     1.25 &     1.25 &     1.25 &     1.25 &     1.25\\
\hline
\multirow{2}{*}{\begin{tabular}{l}$N = 300$\\$R = 75$\end{tabular}} 
& 20 & 4 &     0.01 &     0.01 &     0.01 &     0.01 &     0.01 &     0.01 &     0.01 &     0.01 &     0.01\\
 & 30 & 7 &     0.02 &     0.02 &     0.02 &     0.02 &     0.02 &     0.02 &     0.02 &     0.02 &     0.02\\
 & 40 & 10 &     0.22 &     0.22 &     0.22 &     0.19 &     0.22 &     0.19 &     0.19 &     0.19 &     0.19\\
 & 50 & 16 &     0.87 &     0.54 &     0.54 &     0.54 &     0.54 &     0.54 &     0.54 &     0.54 &     0.54\\
\hline
\multirow{2}{*}{\begin{tabular}{l}$N = 400$\\$R = 75$\end{tabular}} 
& 20 & 6 &     0.01 &     0.01 &     0.01 &     0.01 &     0.01 &     0.01 &     0.01 &     0.01 &     0.01\\
 & 30 & 9 &     0.20 &     0.12 &     0.12 &     0.12 &     0.14 &     0.14 &     0.14 &     0.14 &     0.11\\
 & 40 & 13 &     0.97 &     0.93 &     0.54 &     0.69 &     0.53 &     0.54 &     0.53 &     0.54 &     0.53\\
 & 50 & 20 &    12.72 &    36.46 &     2.96 &     5.80 &     2.68 &     2.87 &     3.03 &     3.01 &     3.01\\
\hline
\multirow{2}{*}{\begin{tabular}{l}$N = 500$\\$R = 75$\end{tabular}} 
& 20 & 5 &     0.01 &     0.01 &     0.01 &     0.01 &     0.01 &     0.01 &     0.01 &     0.01 &     0.01\\
 & 30 & 9 &     0.16 &     0.16 &     0.03 &     0.03 &     0.03 &     0.03 &     0.03 &     0.03 &     0.03\\
 & 40 & 12 &     0.94 &     0.46 &     0.44 &     0.44 &     0.44 &     0.44 &     0.44 &     0.44 &     0.48\\
 & 50 & 19 &    13.35 &     2.31 &    16.81 &     4.14 &     3.36 &     3.30 &     2.59 &     2.57 &     2.57\\
\hline
\multicolumn{3}{r}{\textbf{Average:}} & \multicolumn{1}{r}{\textbf{    3.82}} & \multicolumn{1}{r}{\textbf{    2.61}} & \multicolumn{1}{r}{\textbf{    1.50}} & \multicolumn{1}{r}{\textbf{    0.92}} & \multicolumn{1}{r}{\textbf{    0.67}} & \multicolumn{1}{r}{\textbf{    0.68}} & \multicolumn{1}{r}{\textbf{    0.54}} & \multicolumn{1}{r}{\textbf{    0.54}} & \multicolumn{1}{r}{\textbf{    0.54}}
\end{tabular}
}
\caption{Results: picker rack sequencing, no revisits allowed, miniload instances, $O=50$ orders, PRSF}
\label{mini50seq}
\end{table}

\subsection{Research questions and contribution}

Recall the  research questions posed previously above that we aimed to address in this paper:  

\begin{itemize}
\item Is it possible to develop a mathematical formulation for the problem for simultaneously allocating both orders and racks to multiple pickers? Moreover, if the answer to this question is positive, can that formulation then be used as an explicit basis for the development of computationally effective solution algorithms? 

\item  Is it possible to develop a mathematical formulation for the rack sequencing problem, both when a rack can only visit a picker once and when rack revisits are allowed, that explicitly considers rack inventory positions and decides the number of units of each product allocated to each order from each rack? Moreover, if the answer to this question is positive, can that formulation then be used as an explicit basis for the development of a computationally effective solution algorithm?

\end{itemize}

\noindent We believe that the work presented in this paper has answered these questions in a positive way. Specifically with regard to our first research question we have:
\begin{compactitem}
\item presented a formulation for simultaneously allocating both orders and racks to multiple pickers in Section~\ref{sec:formjeb}
\item proposed heuristics explicitly based upon this formulation in Section~\ref{sec:heuristics}
\item presented computational results for a large number of test instances both for direct solution of our formulation using Cplex (Section~\ref{section72}) and for our heuristics 
(Sections~\ref{section73},\ref{section74})
\end{compactitem}

\vspace{\baselineskip}
\noindent
With regard to our second research question we have:
\begin{compactitem}
\item presented a formulation in Section~\ref{sec:formjebrack}
for the rack sequencing problem, both when a rack can only visit a picker once and when rack revisits are allowed, that explicitly considers rack inventory positions and decides the number of units of each product allocated to each order from each rack
\item presented computational results for a large number of test instances via direct solution of our formulation using Cplex (Section~\ref{section75})
\end{compactitem}

\vspace{\baselineskip}
\noindent
In the light of the work presented above we
 believe that 
the contribution of this paper is:
\begin{compactitem}

\item to be one of the first in the literature to present an optimisation based approach for simultaneously allocating both orders and racks to multiple pickers

\item to use our order and rack allocation formulation as the basis for two matheuristics for the problem of simultaneously allocating both orders and racks to multiple pickers

\item to present an innovative formulation for the rack sequencing problem, both when a rack can only visit a picker once and when rack revisits are allowed, that explicitly considers rack inventory positions and decides the number of units of each product allocated to each order from each rack

\item to prove that, subject to certain conditions being satisfied, in   generating a feasible rack sequence for a single picker we can  initially neglect  all orders which comprise a single unit of one product since a feasible rack sequence for the remaining orders
can be easily modified to incorporate the neglected orders

\item to investigate how our approaches for order and rack allocation, and rack sequencing, perform computationally for test problems that are made publicly available for use by future researchers

\end{compactitem}

\vspace{\baselineskip}
\noindent In terms of the computation results presented (on the test problems considered) we believe that this paper demonstrates that:
\begin{itemize}
\item for order and rack allocation:
\begin{itemize}
\item  optimal solutions can be found very quickly using our order and rack allocation formulation for problems with $O \leq 150$ orders  when we have just a  single picker
\item our single picker based (SPB) and partial integer optimisation (PIO) heuristics perform well as the number of orders increases
\item there are benefits from running a number of the solution approaches given in this paper in parallel and taking the best solution found after a predefined time limit
\end{itemize}
\item for rack sequencing:
\begin{itemize}
\item feasible rack sequences (with no revisits) can be found very quickly in many cases, even for a relatively small number of bin positions
\item feasible rack sequences (when revisits are allowed) can be found for most cases where the approach with no revisits did not find a feasible rack sequence
 \end{itemize}
\end{itemize}

\section{Conclusions and future work}
\label{sec:conclusions}

In this paper we 
considered the problem of simultaneously 
 allocating orders and mobile storage racks to pickers.  
We presented a formulation of the problem as an integer program
 and
 discussed the complexity of the problem. We presented two heuristics (matheuristics) for the problem,
one using partial integer optimisation, that were directly based upon our formulation.

We considered the problem of how to sequence the racks for presentation at each individual picker and
formulated this problem as an integer program, both when no rack revisits to the picker are allowed and
when rack revisits are allowed. We proved that, subject to certain conditions being satisfied, a feasible
rack sequence for all orders can be produced by focusing on just a subset of the orders to be dealt with by the picker.

We discussed the application of the work presented in this paper within a dynamic
environment, as well as the wider applicability of our work 
to other automated (non-mobile robot) fulfilment systems.

Computational results were presented, both for order and rack allocation, and for rack sequencing,  for test problems (that are made publicly
available) involving up to 1000 products, 200 orders, 500 racks and 10 pickers. 

As noted in our literature  survey much remains to be done in the scientific literature with regard to  robotic mobile fulfilment systems. We hope that this paper will encourage other researchers  to turn their attention to problems such as those considered in this paper.

In future we plan to focus our research on studying new approaches for the order and rack allocation problem. We aim to develop valid inequalities that could be added to our formulation of this problem to improve the linear programming relaxation bound, thereby allowing larger instances to be solved optimally within a given time limit. 
Being able to solve larger instances would also (potentially) improve the performance of our two matheuristics for the problem.
We also intend  to develop problem-specific metaheuristics for the problem.

 \clearpage
\newpage
 \pagestyle{empty}
\linespread{1}
\small \normalsize

\section*{Acknowledgments}

\noindent
Cristiano Arbex Valle was funded by FAPEMIG grant APQ-01267-18.

\vspace{\baselineskip}
\noindent
The authors would like to thank the anonymous reviewers for their comments and suggestions to improve the quality of the paper. 


\bibliographystyle{plainnat}
\bibliography{tesco}

\begin{thebibliography}{54}
\providecommand{\natexlab}[1]{#1}
\providecommand{\url}[1]{\texttt{#1}}
\expandafter\ifx\csname urlstyle\endcsname\relax
  \providecommand{\doi}[1]{doi: #1}\else
  \providecommand{\doi}{doi: \begingroup \urlstyle{rm}\Url}\fi

\bibitem[{Amazon}({2019})]{amazon19}
{Amazon}.
\newblock {What robots do (and don't do) at Amazon fulfilment centres.
  Available from
  https://www.aboutamazon.com/amazon-fulfillment/our-innovation/what-robots-do-and-dont-do-at-amazon-fulfillment-centers/
  last accessed December 12 2019}, {2019}.

\bibitem[Angelelli et~al.({2012})Angelelli, Mansini, and
  Speranza]{angelelli2012}
E.~Angelelli, R.~Mansini, and M.~G. Speranza.
\newblock {Kernel search: A general heuristic for the multi-dimensional
  knapsack problem}.
\newblock \emph{{Computers \& Operations Research}}, {37}\penalty0
  ({11}):\penalty0 {2017--2026}, {2012}.
\newblock \doi{10.1016/j.cor.2010.02.002}.

\bibitem[Asahiro et~al.({2012})Asahiro, Kawahara, and Miyano]{asahiro2012}
Y.~Asahiro, K.~Kawahara, and E.~Miyano.
\newblock {NP-hardness of the sorting buffer problem on the uniform metric}.
\newblock \emph{{Discrete Applied Mathematics}}, {160}\penalty0
  ({10--11}):\penalty0 {1453--1464}, {2012}.
\newblock \doi{10.1016/j.dam.2012.02.005}.

\bibitem[Azadeh et~al.({2019})Azadeh, de~Koster, and Roy]{azadeh2019}
K.~Azadeh, R.~de~Koster, and D.~Roy.
\newblock {Robotized and automated warehouse systems: review and recent
  developments}.
\newblock \emph{{Transportation Science}}, {53}\penalty0 ({4}):\penalty0
  {917--945}, {2019}.
\newblock \doi{10.1287/trsc.2018.0873}.

\bibitem[Banker({2016})]{banker16}
S.~Banker.
\newblock {Robots in the warehouse: It's not just Amazon}.
\newblock {Available from
  https://www.forbes.com/sites/stevebanker/2016/01/11/robots-in-the-warehouse-its-not-just-amazon/\#5492edc740b8
  last accessed December 12 2019}, {2016}.

\bibitem[Beasley et~al.({2000})Beasley, Krishnamoorthy, Sharaiha, and
  Abramson]{beasley00}
J.~E. Beasley, M.~Krishnamoorthy, Y.~M. Sharaiha, and D.~Abramson.
\newblock {Scheduling aircraft landings - the static case}.
\newblock \emph{{Transportation Science}}, {34}\penalty0 ({2}):\penalty0
  {180--197}, {2000}.
\newblock \doi{10.1287/trsc.34.2.180.12302}.

\bibitem[Beasley et~al.({2004})Beasley, Krishnamoorthy, Sharaiha, and
  Abramson]{beasley04}
J.~E. Beasley, M.~Krishnamoorthy, Y.~M. Sharaiha, and D.~Abramson.
\newblock {Displacement problem and dynamically scheduling aircraft landings}.
\newblock \emph{{Journal of the Operational Research Society}}, {55}\penalty0
  ({1}):\penalty0 {54--64}, {2004}.
\newblock \doi{10.1057/palgrave.jors.2601650}.

\bibitem[Beckschafer et~al.({2017})Beckschafer, Malberg, Tierney, and
  Weskamp]{beckschafer2014}
M.~Beckschafer, S.~Malberg, K.~Tierney, and C.~Weskamp.
\newblock \emph{{Simulating storage policies for an automated grid-based
  warehouse system}}, volume {10572} of \emph{{Lecture Notes in Computer
  Science}}, pages {468--472}.
\newblock {Springer, Berlin}, {2017}.
\newblock \doi{10.1007/978-3-319-68496-3_31}.

\bibitem[Boschetti et~al.({2009})Boschetti, Maniezzo, Roffilli, and
  Rohler]{boschetti2009}
M.~A. Boschetti, V.~Maniezzo, M.~Roffilli, and A.~B. Rohler.
\newblock \emph{{Matheuristics: Optimization, simulation and control}}, volume
  {5818} of \emph{{Lecture Notes in Computer Science}}, pages {171--177}.
\newblock {Springer, Berlin}, {2009}.
\newblock \doi{10.1007/978-3-642-04918-7_13}.

\bibitem[Boyle({2019})]{amazon19a}
A.~Boyle.
\newblock {Amazon Robotics unveils two new breeds of robots for its fulfillment
  centers}.
\newblock {Available from
  https://www.geekwire.com/2019/amazon-robotics-unveils-two-new-breeds-robots-fulfillment-centers/
  last accessed December 12 2019}, {2019}.

\bibitem[Boysen et~al.({2017}{\natexlab{a}})Boysen, Briskorn, and
  Emde]{boysen17}
N.~Boysen, D.~Briskorn, and S.~Emde.
\newblock {{\GG{1}}Sequencing of picking orders in mobile rack warehouses}.
\newblock \emph{{European Journal of Operational Research}}, {259}\penalty0
  ({1}):\penalty0 {293--307}, {2017}{\natexlab{a}}.
\newblock \doi{10.1016/j.ejor.2016.09.046}.

\bibitem[Boysen et~al.({2017}{\natexlab{b}})Boysen, Briskorn, and
  Emde]{boysen17a}
N.~Boysen, D.~Briskorn, and S.~Emde.
\newblock {{\GG{2}}Parts-to-picker based order processing in a rack-moving
  mobile robots environment}.
\newblock \emph{{European Journal of Operational Research}}, {262}\penalty0
  ({2}):\penalty0 {550--562}, {2017}{\natexlab{b}}.
\newblock \doi{10.1016/j.ejor.2017.03.053}.

\bibitem[Boysen et~al.({2019}{\natexlab{a}})Boysen, de~Koster, and
  Weidinger]{boysen19}
N.~Boysen, R.~de~Koster, and F.~Weidinger.
\newblock {Warehousing in the e-commerce era: a survey}.
\newblock \emph{{European Journal of Operational Research}}, {277}\penalty0
  ({2}):\penalty0 {396--411}, {2019}{\natexlab{a}}.
\newblock \doi{10.1016/j.ejor.2018.08.023}.

\bibitem[Boysen et~al.({2019}{\natexlab{b}})Boysen, Stephan, and
  Weidinger]{boysen19b}
N.~Boysen, K.~Stephan, and F.~Weidinger.
\newblock {Manual order consolidation with put walls: the batched order bin
  sequencing problem}.
\newblock \emph{{EURO Journal on Transportation and Logistics}}, {8}\penalty0
  ({2}):\penalty0 {169--193}, {2019}{\natexlab{b}}.
\newblock \doi{10.1007/s13676-018-0116-0}.

\bibitem[Bozer and Aldarondo({2018})]{bozer18}
Y.~A. Bozer and F.~J. Aldarondo.
\newblock {A simulation-based comparison of two goods-to-person order picking
  systems in an online retail setting}.
\newblock \emph{{International Journal of Production Research}}, {56}\penalty0
  ({11}):\penalty0 {3838--3858}, {2018}.
\newblock \doi{10.1080/00207543.2018.1424364}.

\bibitem[Chan et~al.({2012})Chan, Megow, Sitters, and Van~Stee]{chan2012}
H.~L. Chan, N.~Megow, R.~Sitters, and R.~Van~Stee.
\newblock {A note on sorting buffers offline}.
\newblock \emph{{Theoretical Computer Science}}, {423}\penalty0 ({}):\penalty0
  {11--18}, {2012}.
\newblock \doi{10.1016/j.tcs.2011.12.077}.

\bibitem[{CNN Business}({2018})]{cnn18}
{CNN Business}.
\newblock {Life inside an Amazon fulfillment center}.
\newblock {Available from https://www.youtube.com/watch?v=iXxPabWb9nI last
  accessed December 12 2019}, {2018}.

\bibitem[{CPLEX Optimizer}({2018})]{cplex128}
{CPLEX Optimizer}.
\newblock {IBM. Available from
  https://www.ibm.com/products/ilog-cplex-optimization-studio, last accessed
  July 2 2019}, {2018}.

\bibitem[de~Koster({2018})]{koster2019}
R.~de~Koster.
\newblock {Automated and robotic warehouses: developments and research
  opportunities}.
\newblock \emph{{Logistics and Transport}}, {38}\penalty0 ({2}):\penalty0
  {33--40}, {2018}.
\newblock \doi{10.26411/83-1734-2015-2-38-4-18}.

\bibitem[de~Koster et~al.({2007})de~Koster, Le-Duc, and
  Roodbergen]{deKoster2007}
R.~de~Koster, T.~Le-Duc, and K.~J. Roodbergen.
\newblock {Design and control of warehouse order picking: a literature review}.
\newblock \emph{{European Journal of Operational Research}}, {182}\penalty0
  ({2}):\penalty0 {481--501}, {2007}.
\newblock \doi{10.1016/j.ejor.2006.07.009}.

\bibitem[Defoort and Veluvolu({2014})]{defoort14}
M.~Defoort and K.~C. Veluvolu.
\newblock {A motion planning framework with connectivity management for
  multiple cooperative robots}.
\newblock \emph{{Journal of Intelligent \& Robotic Systems}}, {75}\penalty0
  ({2}):\penalty0 {343--357}, {2014}.
\newblock \doi{10.1007/s10846-013-9872-0}.

\bibitem[Foroughi et~al.({2020})Foroughi, Boysen, Emde, and
  Schneider]{foroughi20}
A.~Foroughi, N.~Boysen, S.~Emde, and M.~Schneider.
\newblock {High-density storage with mobile racks: Picker routing and product
  location}.
\newblock \emph{{Journal of the Operational Research Society}}, {to appear},
  {2020}.
\newblock \doi{10.1080/01605682.2019.1700180}.

\bibitem[Fussler and Boysen({2019})]{fusser19}
D.~Fussler and N.~Boysen.
\newblock {High-performance order processing in picking workstations}.
\newblock \emph{{EURO Journal on Transportation and Logistics}}, {8}\penalty0
  ({1}):\penalty0 {65--90}, {2019}.
\newblock \doi{10.1007/s13676-017-0113-8}.

\bibitem[Garey and Johnson(1979)]{garey79}
M.~R. Garey and D.S. Johnson.
\newblock \emph{{Computers and intractability: a guide to the theory of
  NP-completeness}}.
\newblock W. H. Freeman and company, San Francisco, 1979.

\bibitem[Guastaroba et~al.({2017})Guastaroba, Savelsbergh, and
  Speranza]{guastaroba2017}
G.~Guastaroba, M.~Savelsbergh, and M.~G. Speranza.
\newblock {Adaptive kernel search: A heuristic for solving mixed integer linear
  programs}.
\newblock \emph{{European Journal of Operational Research}}, {263}\penalty0
  ({3}):\penalty0 {789--804}, {2017}.
\newblock \doi{10.1016/j.ejor.2017.06.005}.

\bibitem[Hanson et~al.({2018})Hanson, Medbo, and Johansson]{hansen18}
R.~Hanson, L.~Medbo, and M.~I. Johansson.
\newblock {Performance characteristics of robotic mobile fulfilment systems in
  order picking applications}.
\newblock \emph{{IFAC PapersOnLine}}, {51}\penalty0 ({11}):\penalty0
  {1493--1498}, {2018}.
\newblock \doi{10.1016/j.ifacol.2018.08.290}.

\bibitem[Herrero-Perez and Martinez-Barbera({2011})]{herrero11}
D.~Herrero-Perez and H.~Martinez-Barbera.
\newblock {Decentralized traffic control for non-holonomic flexible automated
  guided vehicles in industrial environments}.
\newblock \emph{{Advanced Robotics}}, {25}\penalty0 ({6--7}):\penalty0
  {739--763}, {2011}.
\newblock \doi{10.1163/016918611X563283}.

\bibitem[Karp({1972})]{karp1972}
R.~M. Karp.
\newblock \emph{{Reducibility among combinatorial problems}}, pages {85--103}.
\newblock {Springer}, {Boston, MA, USA}, {1972}.
\newblock \doi{10.1007/978-1-4684-2001-2_9}.

\bibitem[Kiel({1992})]{kiel92}
J.~M. Kiel.
\newblock {On the complexity of scheduling tasks with discrete starting times}.
\newblock \emph{{Operations Research Letters}}, {12}\penalty0 ({5}):\penalty0
  {293--295}, {1992}.
\newblock \doi{10.1016/0167-6377(92)90087-J}.

\bibitem[Lamballais et~al.({2017})Lamballais, Roy, and de~Koster]{lamballais17}
T.~Lamballais, D.~Roy, and M.~B.~M. de~Koster.
\newblock {Estimating performance in a robotic mobile fulfillment system}.
\newblock \emph{{European Journal of Operational Research}}, {256}\penalty0
  ({3}):\penalty0 {976--990}, {2017}.
\newblock \doi{10.1016/j.ejor.2016.06.063}.

\bibitem[Li et~al.({2017})Li, Zhang, Zhang, and Hua]{li17}
Z.~P. Li, J.~L. Zhang, H.~J. Zhang, and G.~W. Hua.
\newblock {Optimal selection of movable shelves under cargo-to-person picking
  mode}.
\newblock \emph{{International Journal of Simulation Modelling}}, {16}\penalty0
  ({1}):\penalty0 {145--156}, {2017}.
\newblock \doi{10.2507/IJSIMM16(1)CO2}.

\bibitem[Nakajima and Hakimi({1982})]{nakajima1982}
K.~Nakajima and S.~L. Hakimi.
\newblock {Complexity results for scheduling tasks with discrete starting
  times}.
\newblock \emph{{Journal of Algorithms}}, {3}\penalty0 ({4}):\penalty0
  {344--361}, {1982}.
\newblock \doi{10.1016/0196-6774(82)90030-X}.

\bibitem[Nigam et~al.({2014})Nigam, Roy, de~Koster, and Adan]{nigam14}
S.~Nigam, D.~Roy, R.~de~Koster, and I.~Adan.
\newblock \emph{{Analysis of class-based storage strategies for the mobile
  shelf-based order pick system}}.
\newblock {College Industry Council on Material Handling Education (CICMHE),
  Available from
  http://www.mhi.org/downloads/learning/cicmhe/colloquium/2014/18-Nigam\%20paper.pdf
  last accessed December 12 2019}, {2014}.

\bibitem[Qiu et~al.({2002})Qiu, Hsu, Huang, and Wang]{qiu02}
L.~Qiu, W.~J. Hsu, S.~Y. Huang, and H.~Wang.
\newblock {Scheduling and routing algorithms for AGVs: a survey}.
\newblock \emph{{International Journal of Production Research}}, {40}\penalty0
  ({3}):\penalty0 {745--760}, {2002}.
\newblock \doi{10.1080/00207540110091712}.

\bibitem[Roy et~al.({2019})Roy, Nigam, de~Koster, Adan, and Resing]{roy19}
D.~Roy, S.~Nigam, R.~de~Koster, I.~Adan, and J.~Resing.
\newblock {Robot-storage zone assignment strategies in mobile fulfillment
  systems}.
\newblock \emph{{Transportation Research Part E - Logistics and Transportation
  Review}}, {122}:\penalty0 {119--142}, {2019}.
\newblock \doi{10.1016/j.tre.2018.11.005}.

\bibitem[Sanders and Kaul({2019})]{sanders19}
G.~Sanders and A.~Kaul.
\newblock {Warehousing and logistics robots. Executive summary}.
\newblock {Available from
  https://www.tractica.com/research/warehousing-and-logistics-robots/ last
  accessed December 12 2019}, {2019}.

\bibitem[Shahriari and Biglarbegian({2018})]{shahriari18}
M.~Shahriari and M.~Biglarbegian.
\newblock {A new conflict resolution method for multiple mobile robots in
  cluttered environments with motion-liveness}.
\newblock \emph{{IEEE Transactions on Cybernetics}}, {48}\penalty0
  ({1}):\penalty0 {300--311}, {2018}.
\newblock \doi{10.1109/TCYB.2016.2633331}.

\bibitem[Spieksma and Crama({1992})]{spieksma92}
F.~C.~R. Spieksma and Y.~Crama.
\newblock {The complexity of scheduling short tasks with few starting times}.
\newblock {Reports in operations research and systems theory, report number
  M92-06, University of Limburg, The Netherlands. Available from
  https://orbi.uliege.be/bitstream/2268/138861/1/ShortTasksFewStartTimes.pdf
  last accessed April 16 2020}, {1992}.

\bibitem[Tessensohn et~al.({2020})Tessensohn, Roy, and de~Koster R.
  B.~M.]{tessensohn2020}
T.~L. Tessensohn, D.~Roy, and de~Koster R. B.~M.
\newblock {Inventory allocation in robotic mobile fulfillment systems}.
\newblock \emph{{IISE Transactions}}, {52}\penalty0 ({1}):\penalty0 {1--17},
  {2020}.
\newblock \doi{10.1080/24725854.2018.1560517}.

\bibitem[Van~Gils et~al.({2018})Van~Gils, Ramaekers, Caris, and
  de~Koster]{gils2018}
T.~Van~Gils, K.~Ramaekers, A.~Caris, and R.~B.~M. de~Koster.
\newblock {Designing efficient order picking systems by combining planning
  problems: State-of-the-art classification and review}.
\newblock \emph{{European Journal of Operational Research}}, {267}\penalty0
  ({1}):\penalty0 {1--15}, {2018}.
\newblock \doi{10.1016/j.ejor.2017.09.002}.

\bibitem[Vis({2006})]{vis06}
I.~F.~A. Vis.
\newblock {Survey of research in the design and control of automated guided
  vehicle systems}.
\newblock \emph{{European Journal of Operational Research}}, {170}\penalty0
  ({3}):\penalty0 {677--709}, {2006}.
\newblock \doi{10.1016/j.ejor.2004.09.020}.

\bibitem[Wang et~al.({2020})Wang, Yang, and Li]{wang2019}
K.~Wang, Y.~M. Yang, and R.~X. Li.
\newblock {Travel time models for the rack-moving mobile robot system}.
\newblock \emph{{International Journal of Production Research}}, {58}\penalty0
  ({14}):\penalty0 {4367--4385}, {2020}.
\newblock \doi{10.1080/00207543.2019.1652778}.

\bibitem[Weidinger({2018})]{weidinger18c}
F.~Weidinger.
\newblock \emph{{A precious mess: on the scattered storage assignment
  problem}}, pages {31--36}.
\newblock {Operations Research Proceedings}. {Springer International
  Publishing}, {2018}.
\newblock \doi{10.1007/978-3-319-55702-1_5}.

\bibitem[Weidinger and Boysen({2018})]{weidinger18b}
F.~Weidinger and N.~Boysen.
\newblock {Scattered storage: how to distribute stock keeping units all around
  a mixed-shelves warehouse}.
\newblock \emph{{Transportation Science}}, {52}\penalty0 ({6}):\penalty0
  {1412--1427}, {2018}.
\newblock \doi{10.1287/trsc.2017.0779}.

\bibitem[Weidinger et~al.({2018})Weidinger, Boysen, and
  Briskhorn]{weidinger18a}
F.~Weidinger, N.~Boysen, and D.~Briskhorn.
\newblock {Storage assignment with rack-moving mobile robots in KIVA
  warehouses}.
\newblock \emph{{Transportation Science}}, {52}\penalty0 ({6}):\penalty0
  {1479--1495}, {2018}.
\newblock \doi{10.1287/trsc.2018.0826}.

\bibitem[Weidinger et~al.({2019})Weidinger, Boysen, and
  Schneider]{weidinger2019}
F.~Weidinger, N.~Boysen, and M.~Schneider.
\newblock {Picker routing in the mixed-shelves warehouses of e-commerce
  retailers}.
\newblock \emph{{European Journal of Operational Research}}, {274}\penalty0
  ({2}):\penalty0 {501--515}, {2019}.
\newblock \doi{10.1016/j.ejor.2018.10.021}.

\bibitem[Xiang et~al.({2018})Xiang, Liu, and Miao]{xiang18}
X.~Xiang, C.~C. Liu, and M.~X. Miao.
\newblock {Storage assignment and order batching problem in Kiva mobile
  fulfilment system}.
\newblock \emph{{Engineering Optimization}}, {50}\penalty0 ({11}):\penalty0
  {1941--1962}, {2018}.
\newblock \doi{10.1080/0305215X.2017.1419346}.

\bibitem[Xue et~al.({2019})Xue, Tang, Su, and Li]{xue19}
F.~Xue, H.~Tang, Q.~Su, and T.~Li.
\newblock {Task allocation of intelligent warehouse picking system based on
  multi-robot coalition}.
\newblock \emph{{KSII Transactions on Internet and Information Systems}},
  {13}\penalty0 ({7}):\penalty0 {3566--3582}, {2019}.
\newblock \doi{10.3837/tiis.2019.07.013}.

\bibitem[Yu({2016})]{yu16}
J.~J. Yu.
\newblock {Intractability of optimal multirobot path planning on planar
  graphs}.
\newblock \emph{{IEEE Robotics and Automation Letters}}, {1}\penalty0
  ({1}):\penalty0 {33--50}, {2016}.
\newblock \doi{10.1109/LRA.2015.2503143}.

\bibitem[Yuan et~al.({2018})Yuan, Cezik, and Graves]{yuan18}
R.~Yuan, T.~Cezik, and S.~C. Graves.
\newblock {Stowage decisions in multi-zone storage systems}.
\newblock \emph{{International Journal of Production Research}}, {56}\penalty0
  ({1--2}):\penalty0 {333--343}, {2018}.
\newblock \doi{10.1080/00207543.2017.1398428}.

\bibitem[Yuan and Gong({2017})]{yuan17}
Z.~Yuan and Y.~M. Gong.
\newblock {Bot-in-time delivery for robotic mobile fulfillment systems}.
\newblock \emph{{IEEE Transactions on Engineering Management}}, {64}\penalty0
  ({1}):\penalty0 {83--93}, {2017}.
\newblock \doi{10.1109/TEM.2016.2634540}.

\bibitem[Zhang et~al.({2019})Zhang, Yang, and Weng]{zhang19}
J.~T. Zhang, F.~X. Yang, and X.~Weng.
\newblock {A building-block-based genetic algorithm for solving the robots
  allocation problem in a robotic mobile fulfilment system}.
\newblock \emph{{Mathematical Problems in Engineering}}, {Article Number:
  6153848}, {2019}.
\newblock \doi{10.1155/2019/6153848}.

\bibitem[Zou et~al.({2017})Zou, Gong, Xu, and Yuan]{zou17}
B.~P. Zou, Y.~M. Gong, X.~H. Xu, and Z.~Yuan.
\newblock {Assignment rules in robotic mobile fulfilment systems for online
  retailers}.
\newblock \emph{{International Journal of Production Research}}, {55}\penalty0
  ({20}):\penalty0 {6175--6192}, {2017}.
\newblock \doi{10.1080/00207543.2017.1331050}.

\bibitem[Zou et~al.({2018})Zou, de~Koster, and Xu]{zou18}
B.~P. Zou, R.~de~Koster, and X.~H. Xu.
\newblock {Operating policies in robotic compact storage and retrieval
  systems}.
\newblock \emph{{Transportation Science}}, {52}\penalty0 ({4}):\penalty0
  {788--811}, {2018}.
\newblock \doi{10.1287/trsc.2017.0786}.

\end{thebibliography}

\end{document}